\documentclass[11pt,twoside, reqno]{amsart}

\usepackage{charter}
\usepackage[colorlinks,pagebackref]{hyperref}
\PassOptionsToPackage{linktocpage}{hyperref}
\hypersetup{
linkcolor={blue}, 
citecolor={cyan},
}

\usepackage{amsmath}
\usepackage{amsthm}
\usepackage{amsfonts, amssymb}
\usepackage{mathrsfs}
\usepackage[all]{xy}
\usepackage{url}

\usepackage{latexsym}
\usepackage[dvips]{graphicx}

\theoremstyle{plain}
\newtheorem{lema}{Lemma}
\newtheorem{prop}[lema]{Proposition}
\newtheorem{teo}[lema]{Theorem}
\newtheorem{conj}[lema]{Conjecture}
\newtheorem{ques}[lema]{Question}

\newtheorem{coro}[lema]{Corollary}
\theoremstyle{remark}

\newtheorem{obs}[lema]{Remark}
\theoremstyle{definition}
\newtheorem{defi}[lema]{Definition}
\newtheorem{ej}[lema]{Example}

\newcommand{\Z}{\mathbb{Z}}

\newcommand{\R}{\mathbb{R}}

\newcommand{\w}{\mathrm{w}}
\newcommand{\PP}{\mathcal{P}}
\newcommand{\Q}{\mathcal{Q}}
\newcommand{\F}{\mathbb{F}_2}
\newcommand{\FF}{\F'}
\newcommand{\FFF}{\F''}
\newcommand{\M}{\mathbb{M}_2}
\newcommand{\W}{\mathrm{W}}
\newcommand{\g}{\textrm{Cl}}
\newcommand{\lax}{\textrm{Cl}_{\M}}
\newcommand{\Sq}{\textrm{Sq}}
\newcommand{\ZZ}{\Z[X^{\pm 1},Y^{\pm 1}]}
\newcommand{\vv}{\mathrm{v}}

\begin{document}

\title[The winding invariant]{The winding invariant}

\author[J.A. Barmak]{Jonathan Ariel Barmak $^{\dagger}$}

\thanks{$^{\dagger}$ Researcher of CONICET. Partially supported by grant UBACyT 20020160100081BA}

\address{Universidad de Buenos Aires. Facultad de Ciencias Exactas y Naturales. Departamento de Matem\'atica. Buenos Aires, Argentina.}

\address{CONICET-Universidad de Buenos Aires. Instituto de Investigaciones Matem\'aticas Luis A. Santal\'o (IMAS). Buenos Aires, Argentina. }

\email{jbarmak@dm.uba.ar}

\begin{abstract}
Every element $w$ in the commutator subgroup of the free group $\F$ of rank 2 determines a closed curve in the grid $\Z\times \R \cup \R \times \Z\subseteq \R^2$. The winding numbers of this curve around the centers of the squares in the grid are the coefficients of a Laurent polynomial $P_w$ in two variables. This basic definition is related to well-known ideas in combinatorial group theory. We use this invariant to study equations over $\F$ and over the free metabelian group of rank $2$. We give a number of applications of algebraic, geometric and combinatorial flavor.  
\end{abstract}



\maketitle

\footnotesize
\tableofcontents
\normalsize

\section{Introduction}

We introduce an invariant which associates a Laurent polynomial in two variables with integer coefficients to every element $w$ in the commutator subgroup $\FF$ of the free group in two generators $x,y$. The \textit{winding invariant} $\W:\FF \to \ZZ$ is in fact a group homomorphism. Given a word $w\in \F$, there is a polygonal curve $\gamma_w$ in the square grid $\Z\times \R \cup \R \times \Z$, which is closed when $w\in \FF$. This curve starts at the origin $(0,0)$ and it is the path induced by $w$ in the Cayley graph $\Gamma(\Z\times \Z, \{x,y\})=\Z\times \R \cup \R \times \Z$. The coefficients of the polynomial $P_w=\W(w)=\sum\limits_{i,j} a_{i,j}X^iY^j$ are the winding numbers $a_{i,j}=\w(\gamma_w, c_{i,j})$ of the curve around the centers of the squares in the grid. Alternatively the winding invariant can be seen as the quotient $N\to N / N'$ where $N$ is the normal closure of $[x,y]$ in $\F$ and $N/N'$ is the relation module of the presentation $\langle x,y | [x,y] \rangle$. The definition of $\W$ is elementary and there exist several articles in the literature which deal with it in some way. Most of them use the winding invariant as an algebraic non-geometric tool \cite{BM6, BNNN, BaMe}. In contrast, some others use geometric ideas but do not exploit algebraic properties. For example, the family of coefficients of $P_w$ appears in an article of Conway and Lagarias \cite[Section 5]{CL}, but the authors do not use the ring structure of $\ZZ$. If $w, w'\in \FF$ are conjugate in $\F$, then $P_w=X^iY^jP_{w'}$ for some $i,j\in \Z$. In \cite{Sar}, Sarkar considers a map $S:\FF \to \ZZ$ without appealing to winding numbers. The coefficient of $X^iY^j$ in $S(w)$ is the number of times $\gamma_w$ travels through the edge $(i,j)$, $(i+1,j)$, taking into account orientations. Sarkar's invariant $S$ is in fact the Magnus embedding $\F' / \F '' \to \Z[\F / \F ']\oplus \Z[\F / \F ']$ composed with the projection onto the first coordinate (that is, the Fox derivative with respect to the generator $x$).  In the particular case that $\gamma_w$ is simple, positively-oriented and the boundary of a convex polygon, $P_w$ is known as the characteristic/generating function of the polygon \cite{Bri}.

In this article we will present a number of results in combinatorial group theory which can be obtained as an application of the winding invariant. Most of the results that we exhibit here are original, but some others are known and we only provide alternative proofs using our ideas. We will illustrate the strength of this tool with different applications which appear when we use the particular geometric description of $\W$. Since the results in this paper lie within several areas of group theory and algebraic topology, we have tried to give a self contained presentation of the subject.



We will use the winding invariant to study equations over $\F$. For instance we will study properties of elements of $\FF$ which are products of $k$-th powers. A result by Lyndon and Morris Newman \cite{LN} states that $[x,y]$ is not a product of two squares in $\F$. It is known now that there is an algorithm which decides for an element $w\in \F$ which is a product $a_1^2a_2^2\ldots a_n^2$ of squares, what is the minimum number $n$ in such an expression. This is the \textit{square length} of $w$ \cite{Cul, Edm, Edm2, GT}. The \textit{commutator length} of an element $w\in \F'$ is the minimum number of commutators whose product is $w$. This is also algorithmically computable, and in fact this generalizes to all quadratic equations (see \cite{Edm, Edm2} and \cite{CE0, KHAR}). The stable version of the commutator length has received a lot of attention in the last decade \cite{BBF,Cal,CF} because of its connections with Dynamics and two-dimensional bounded cohomology. We will give a three-line proof of Lyndon-Newman's result using the winding invariant and we will state a generalization in Proposition \ref{lngeneral}, providing a necessary condition for an element $w\in \F$ to be a product of two $k$-th powers. This is related to the problem of deciding whether the Burnside group $B(2,5)$ is finite. Kostrikin proved \cite{Kos} that if the sixth Engel word $e_6=[y,[y,[y,[y,[y,[y,x]]]]]]$ is not a product of fifth powers in $\F$, then $B(2,5)$ is infinite. We will see in Corollary \ref{coroengel2} that no Engel word is a product of two $k$-th powers for $k\ge 2$. It is known that $e_5$ is a product of fourth powers, but we will prove in Corollary \ref{etres} that $e_3$ is not. This could be proved alternatively using coset enumeration for some presentation of $B(2,4)$ (see \cite{Hav4, New}), but our computations are simpler. Using the Cayley graph of $B(2,3)$ it is not hard to check for a given element in $\F$ if it is a product of cubes. In Theorem \ref{caracterizacioncubos} we will use the winding invariant to obtain a different characterization of the products of cubes.  For $k\ge 4$ we provide necessary conditions for an element in $\F$ to be a product of $k$-th powers (Theorem \ref{npowers}, Corollary \ref{npowersf} and Proposition \ref{color}). Another result by Lyndon and Newman in \cite{LN} says that in the free group $F(x_1,x_2,\ldots,x_n)$ freely generated by $x_1,x_2,\ldots, x_n$, the element $x_1^2x_2^2\ldots x_n^2$ is not a product of fewer than $n$ squares. In Theorem \ref{lnnuevo} we generalize the case $n=3$ by showing that for $n\ge 3$ the equation $x_1^mx_2^m\ldots x_n^m=a^kb^k$ has no solution $a,b\in F(x_1,x_2,\ldots ,x_n)$ if $m\ge 1$ and $k\ge 2$.

Since the kernel of the winding invariant $W:\F' \to \ZZ$ is $\F''$, we can use it to study problems concerning the free metabelian group $\M=\F/\F''$ of rank $2$. Recall that a group is termed metabelian if its derived subgroup is abelian. Philip Hall noticed \cite{PHal} that if $G$ is metabelian, then $G'$ has a structure of $\Z[G/G']$-module, and we can use commutative algebra to study $G'$. In our case, $\M'$ is a free $\Z[\Z \times \Z]$-module of rank $1$ and the isomorphism $\M' \to \ZZ$ is given by the winding invariant. On the other hand, $\M'$ is the relation module of the presentation $\langle x,y | [x,y]\rangle$ of $\Z \times \Z$, and the relation module $N(R)/N(R)'$ of a presentation $\langle X| R\rangle$ of a group $G$ always has a structure of $\Z[G]$-module. Laurent polynomials have been used before to study metabelian groups \cite{BM6, BaMe}. For instance Bavard and Meigniez \cite{BaMe} proved that the commutator length of $\M$ is $2$ (the maximum of the commutator lengths of elements in $\M'$). The computation of the square length of $\M$ seems to be a more difficult problem and we will prove in Theorem \ref{mainsquarel} that this number is $4$ or $5$. This contrasts with the case of finite simple groups in which it is known that every element is a product of two squares \cite{GuMa, LOST}, and the case of non-cyclic free groups, which contain elements of arbitrarily large square length. Theorem \ref{mainsquarel} is one of the results which best represent the ideas exposed in this paper. We don't know whether $e_6$ is a product of fifth powers in $\F$, but this holds in $\M$. In fact, for every positive prime $p$ and every $n>p$, $e_n\in \M$ is a product of $p$-th powers (Proposition \ref{ppowers}). 

Square length and commutator length are just two examples of the notion of $u$-length. If $u$ is an element in the free group $F(x_1,x_2, \ldots, x_n)$ and $G$ is any group, there is a \textit{word map} $G^n\to G$ which maps $(g_1,g_2,\ldots ,g_n)$ to $u(g_1,g_2,\ldots ,g_n)$. The elements in the image of the word map and their inverses are the \textit{$u$-elements}. The \textit{verbal subgroup} $u(G)$ is the subgroup of $G$ generated by the $u$-elements and the \textit{$u$-length} of an element $g\in u(G)$ is the smallest $m$ such that $g$ can be written as a product of $m$ $u$-elements. The \textit{$u$-width} (or \textit{$u$-length}) of $u(G)$ is the supremum of the $u$-lengths of elements in $u(G)$. If $G$ is a finite simple group and $u(G)\neq 1$, then $u(G)=G$ and in fact there is a universal bound for the $u$-width of $G$ for all such finite simple groups \cite[Theorem 1.6]{LiSh}. The problem of finding a constant $c$ such that the $u$-width of $u(G)$ is smaller than or equal to $c$ for a family of groups $G$ is known as a Waring type problem for its resemblance with the classical problem of the same name. For $u=[x,y]$, the $u$-width of every non-abelian finite simple group $G$ is $1$, that is, every element of $G$ is a commutator. This is the statement of the Ore conjecture \cite{Ore} whose proof was finally completed in 2010 \cite{LOST}. On the other hand there are finitely generated infinite simple groups with infinite commutator length and infinite square length \cite{Mur}. Products of two $k$-th powers have also been studied in the context of finite simple groups \cite{GLOST}. For more references on the theory of word maps and Waring type problems see \cite{Sha2}. Shalev proposed in \cite[Conjecture 3.5]{Sha2} a generalization of the Ore conjecture: for every Engel word $e_n$ and every finite simple group $G$, the word map $e_n:G^2\to G$ is surjective. In Section \ref{secverbal} we study $e_n$-lengths of elements in $\M$ and we prove in Theorem \ref{nge3} that the verbal subgroup $e_n(\M)$ coincides with the $n$-th term $\gamma_{n+1}(\M)$ of the lower central series for $n=1,2$, but they are different for $n\ge 3$.   

We will study a question of Sims who asked whether the normal subgroup generated by the basic commutators of weight $m$ in $\mathbb{F}_n$ is $\gamma_m(\mathbb{F}_n)$ \cite{Sim} (Theorem \ref{sims}). Baumslag and Mikhailov ask in \cite{BM2} whether all one-relator groups with two generators, whose relator is a basic commutator, are residually torsion-free nilpotent. We prove in Theorem \ref{teobm} that this holds modulo $\F''$.

We have mentioned that the commutator length is algorithmically computable in $\F$. The first result in this direction, due to Wicks \cite{Wic}, was a characterization of those elements with commutator length 1, i.e. commutators.

\begin{teo}(Wicks) \label{wicks}
An element $w\in \F'$ is a commutator $[u,v]$ if and only if there is a cyclically reduced conjugate of $w$ which coincides with a cyclically reduced word of the form $ABCA^{-1}B^{-1}C^{-1}$.
\end{teo}

In Section \ref{secparalelogramos} we use the winding invariant to study commutators from a different perspective. A commutator $[u,v]$ in $\F$ determines a parallelogram in $\R^2$ and this is used to prove that the coefficients of $P_{[u,v]}$ can be partitioned in classes, the sum of the members in each class is a constant number $\iota \in \{-1,0,1\}$. In Proposition \ref{propce} we give an alternative proof of a result of Comerford and Edmunds that there exists a commutator $[u,v]$ which is a product of two squares but it is not of the form $[a^2,b]$. We also use the ideas of this section to give a short proof to a theorem of Baumslag-Neumann-Neumann-Neumann which says that if $w,u\in \F$ are independent modulo $\F'$, then the equation $[w,u]=z^k$ has no solution in $\M$ for $k\ge 2$ (Proposition \ref{propbnnn}). We also give a short proof of a result by Baumslag and Mahler which is the analogue of the Vaught conjecture for $\M$ (Proposition \ref{propbm}). The methods of this section can be used to study commutators in $\M$. Commutators in this group have been also studied in \cite[Théorème 2]{BaMe}.


In Section \ref{sectionac} we use the winding invariant to study $Q^*$-equivalence of presentations. Our article \cite{Barac} was motivated by this section. We prove there that there exist presentations which are not $Q^*$-equivalent but which have simple homotopy equivalent standard complexes, thus disproving a strong version of the Andrews-Curtis conjecture. The key ingredient here is the existence of a matrix $M\in GL_2(\ZZ)$ which is not in $GE_2(\ZZ)$. Bachmuth and Mochizuki proved in \cite{BM} that $\ZZ$ is not generalized Euclidean, and Evans was the first to give an example of a concrete matrix in \cite{Eva}. The existence of infinitely many presentations which are simple homotopy equivalent but pairwise $Q$-inequivalent is discussed at the end of Section \ref{seccayley} in connection with the relation lifting problem. In Theorem \ref{qineq} we prove that $\M$ contains infinitely many pairs with the same normal closure and which are pairwise $Q$-inequivalent. This is similar to a result obtained by Myasnikov-Myasnikov-Shpilrain in \cite{MMS}. Inspired by \cite{MMS} we use the winding invariant to study the following open question: is $GL_2(\Z[X,X^{-1}])$ generated by elementary and diagonal matrices? If we replace $2$ by $3$, the answer is affirmative and proved by Suslin \cite{Sus}.

In Section \ref{sectionhow} we introduce more general versions of the winding invariant. If $\mathcal{P}=\langle X| R\rangle$ is a presentation of a group $G$ and $N(R)$ is the normal closure of $R$ in $F(X)$, the winding invariant associated to $\mathcal{P}$ is the abelianization map from $N(R)$ to the relation module $N(R)/N(R)'$. The original winding invariant corresponds to $\langle x,y| [x,y]\rangle$. If $\mathcal{P}$ is aspherical, the relation module is a free $\Z[G]$-module (See Appendix A). We use this together with the fact that one-relator presentations whose relator is not a proper power are aspherical \cite{Lyn}, to prove a generalization (Lemma \ref{Baum}) of a result by Baumslag, Miller and Troeger \cite{BMT} which says that in a free group, two non-commuting elements $u,v$ cannot satisfy that $vuv^{-1}u^{-2}$ is a conjugate of $u$ or $u^{-1}$. We use this to prove in Propositions \ref{ressol} and \ref{resnil} that certain one-relator groups are not residually solvable or not residually nilpotent. The connection of the winding invariant with the Magnus embedding is mentioned in Proposition \ref{embedding}.

In Section \ref{seccayley} we use variants of the winding invariant to study presentations other than $\langle x,y | [x,y]\rangle$ with planar Cayley graphs, and derive new results on equations over free groups.

Although this article has followed a different path, the original motivation for the definition of the winding invariant was a combinatorial problem about tilings of regions in the plane. The bisection problem is presented in Section \ref{secbisec}, together with applications to equations over $\F$ and other tiling problems.  
   
\section{Products of two $k$-th powers} \label{secuno}

In \cite{LN}, Lyndon and Morris Newman proved the following result.

\begin{prop} \label{ln}
The commutator $[x,y]=xyx^{-1}y^{-1}$ is not a product of two squares in $\F$, the free group of rank 2 generated by $x$ and $y$.
\end{prop}

The article \cite{LN} contains in fact two proofs of Proposition \ref{ln}. The first one, due to Newman, considers a representation $\F \to PSL(2, \Z)$ and then performs many computations with matrices. The second proof, due to Lyndon, works in a quotient of $\F$ where the elements have a canonical form. Both proofs require a lot of clever ideas and hard work.

In 1973 Wicks \cite{Wic2} extended his ideas from \cite{Wic} to classify the elements of $\F$ which are products of two squares. Independently in 1975 Edmunds obtained the same result \cite{Edm}. Culler obtained the same description in \cite[Remark 3.2 (2)]{Cul} using maps from surfaces. This can also be proved using results from \cite{GT}.

\begin{teo}[Edmunds \cite{Edm}, Culler \cite{Cul}, Wicks \cite{Wic2}] \label{culler}
A word $w\in \F$ is a product of two squares in $\F$ if and only if it has a cyclically reduced conjugate which can be obtained by non-cancelling substitution from one of the four words $AABB$, $ABA^{-1}B$, $ABCA^{-1}CB$, $AABCCB^{-1}$.
\end{teo}

A non-cancelling substitution is a substitution of each letter $A,B,C$ by reduced non-trivial words $\phi (A)$, $\phi(B)$, $\phi(C)\in \F$ such that the resulting word is reduced.

Culler \cite{Cul}, Edmunds \cite{Edm,Edm2} and Goldstein-Turner's \cite{GT} results give in fact algorithms to decide for every $k\ge 1$ if a given element of $\mathbb{F}_n$ is a product of $k$ squares or a product of $k$ commutators. For the case of products of two squares, an algorithm was previously known to exist by Schupp \cite{Sch2}. Schupp proves in fact that for any word $w\in \F$ and any element $g$ in a free group $F$, there is an algorithm which decides if the equation $w(a,b)=g$ has a solution $a,b\in F$ (see also \cite[Theorem 4.4]{Edm2}).

Proposition \ref{ln} trivially follows from Theorem \ref{culler}. A much simpler proof of Proposition \ref{ln} can be found in \cite{Sar}, in which Sarkar proves more generally the following.

\begin{prop}[Sarkar] \label{sarkar}
$[x^m,y^n]$ is a product of two squares in $\F$ if and only if $mn$ is even.
\end{prop}

In order to prove Proposition \ref{sarkar}, Sarkar introduces an invariant $[\F,\F]=\FF\to \Z$ which maps the product of two squares to an even number but maps $[x,y]$ to $1$. Sarkar's invariant is in fact a particular case of the winding invariant. The winding invariant appears in an article of Conway and Lagarias \cite{CL} where it is used to prove that certain regions in the square grid $\Z\times \R \cup \R \times \Z$ cannot be tiled with some particular sets of tiles. We define now this invariant and in order to illustrate its simplicity and strength, we will give a self contained proof of Proposition \ref{ln} in a couple of lines and a generalization of Proposition \ref{sarkar}. 


\begin{defi} \label{definicion}
Let $w\in \F$. Then $w=x_1^{\epsilon_1}x_2^{\epsilon_2}\ldots x_l^{\epsilon_l}$ where $x_i\in \{x,y\}$ and $\epsilon_i \in \{1,-1\}$ for each $i$. The word $w$ determines a path $\gamma_{w}$ in $\R^2$ which begins in $(0,0)$ and is a concatenation of paths $\gamma_1, \gamma_2,\ldots ,\gamma_l$. The path $\gamma_i$ moves one unit parallel to the $x_i$-axis and with positive or negative direction depending on the sign $\epsilon_i$. The image of $\gamma_w$ is contained in the grid $\Z\times \R \cup \R \times \Z$. Note that the ending point of $\gamma_w$ is $(n,m)$ where $n$ is the total exponent of $x$ in $w$ and $m$ is the total exponent of $y$. Recall that $w\in \FF$ if and only if $n=m=0$. Suppose $w\in \FF$, so $\gamma_w$ finishes in $(0,0)$. For each $(i,j)\in \Z \times \Z$, let $a_{i,j}$ be the winding number $\w(\gamma_w, i+\frac{1}{2},j+\frac{1}{2})$ of $\gamma_w$ around the point $p=(i+\frac{1}{2},j+\frac{1}{2})$, that is the number of times the curve $\gamma_w$ travels counterclockwise around $p$. Define the \textit{winding invariant} $P_w\in \ZZ$ of $w$ to be the Laurent polynomial $P_w= \sum a_{i,j} X^iY^j$. The map $\W: \FF \to \ZZ$, $w\mapsto P_w$ is the \textit{winding invariant map}. 
\end{defi}

The curve $\gamma_w$ can also be interpreted in the following way. The word $w\in \F$ induces a cycle in the standard complex $K=S^1\vee S^1$ of the presentation $\langle x,y | \ \rangle$. The curve $\gamma_w$ is the lift of this cycle to the maximal abelian cover $\widetilde{K}=\Z\times \R \cup \R \times \Z$, starting at $(0,0)$. The curve $\gamma_w$ can be defined easily as a product in the \textit{group of paths} introduced in \cite{Ken}. 

\begin{ej}
Let $w=[x,y]=xyx^{-1}y^{-1}$. Then $\gamma_w$ is the curve in $\R^2$ which begins in the origin and moves one unit to the right, then one unit upwards, one to the left and one downwards. Therefore all the winding numbers $a_{i,j}=\w(\gamma_w, i+\frac{1}{2},j+\frac{1}{2})$ are zero with exception of $a_{0,0}=\w(\gamma_w, 0+\frac{1}{2},0+\frac{1}{2})=1$. Hence, the winding invariant of $w$ is $P_w=1\in \ZZ$. In Figure \ref{ej1} at the left we can see a picture with the image of the curve $\gamma_w$. The origin is represented with a small black square and a small arrow indicates the orientation of the curve. Inside the square with vertices $(0,0)$, $(1,0)$, $(1,1)$, $(0,1)$, there is a number $1$ which indicates that the winding number of $\gamma_w$ around the center of the square is $1$. The remaining squares are labeled with the number $0$, which we will always omit in this graphical representation. A polynomial $P$ can be represented graphically as a finite subset of squares in the grid $\Z\times \R \cup \R \times \Z$, each of them labeled with an integer number. The square with center $(i+\frac{1}{2},j+\frac{1}{2})$ is labeled with the coefficient of the monomial $X^iY^j$ in $P$.
\begin{figure}[h] 
\begin{center}
\includegraphics[scale=1.6]{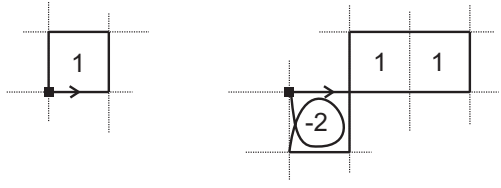}
\caption{The curves $\gamma_{[x,y]}$ and $\gamma_{x^3yx^{-2}y^{-2}x^{-1}yxy^{-1}x^{-1}y}$ and the graphical representation of the corresponding winding invariants.}\label{ej1}
\end{center}
\end{figure}
The second example is $w'=x^3yx^{-2}y^{-2}x^{-1}yxy^{-1}x^{-1}y$. Now the curve $\gamma_{w'}$ has self intersections. Moreover, some edges of the grid appear twice in $\gamma_{w'}$. In order to make the graphical representation clear, we choose to represent a curve $\gamma$ homotopic to $\gamma_{w'}$. Here we consider $\gamma, \gamma_{w'}:S^1\to \R^2\smallsetminus C$, where $C$ is the subset of $\R^2$ which contains the centers $(i+\frac{1}{2},j+\frac{1}{2})$ of the squares (see the picture in the right of Figure \ref{ej1}). In this case the winding numbers $a_{i,j}=\w(\gamma_w, i+\frac{1}{2},j+\frac{1}{2})$ are $a_{1,0}=a_{2,0}=1$, $a_{0,-1}=-2$, and all the others are zero. Therefore $P_{w'}=X+X^2-2Y^{-1}$.
\end{ej}

\begin{prop} \label{basica}
Let $w,w' \in \FF, u\in \F$. Then the following hold:

(i) $P_{w^{-1}}=-P_w$.

(ii) $P_{ww'}=P_w+P_{w'}$.

(iii) $P_{uwu^{-1}}=X^nY^mP_w$, where $n$ and $m$ are the total exponents of $x$ and $y$ in $u$.

(iv) $P_{[u,w]}=(X^nY^m-1)P_w$.
\end{prop}
\begin{proof}
(i) and (ii) follow from the fact that $\gamma_{w^{-1}}$ is the inverse $\overline{\gamma}_w$ \label{caminoinverso} of $\gamma_w$ and $\gamma_{ww'}$ is the concatenation $\gamma_w*\gamma_{w'}$, which we will also denote $\gamma_w \gamma_{w'}$. To prove (iii) observe that $\gamma_{uwu^{-1}}=\gamma_u(\gamma_w+(n,m))\overline{\gamma}_u$, where $\gamma+(n,m)$ is $\gamma$ composed with the translation $\R^2\to \R^2$ which maps $(x_1,x_2)$ to $(x_1+n,x_2+m)$. The last item follows from the previous three.
\end{proof}

\textit{Proof of Proposition \ref{ln}}. Suppose $w=[x,y]=a^2b^2$ for certain $a,b\in \F$. Note that $ab\in \FF$. Since $a^2b^2=(ab)b^{-1}(ab)b$, by Proposition \ref{basica}, $P_w=(1+X^nY^m)P_{ab}$, where $n,m$ are the total exponents of $x$ and $y$ in $b^{-1}$. In particular $P_w(1,1)\in \Z$ is even. However, $P_{[x,y]}=1\in \ZZ$. \qed

\medskip

The fact that a word $w\in \FF$ has winding invariant $P_w$ divisible by a factor $1+X^nY^m$ does not imply that $w$ is a product of two squares, however we will show below that $w$ is a product of two squares in certain quotient of $\F$. First we prove a generalization of Proposition \ref{ln}.

\begin{prop} \label{lngeneral}
Let $w\in \FF$ and let $k\ge 1$. Suppose there exist $a,b\in \F$ such that $w=a^kb^k$. Then there exist $n,m\in \Z$ such that $1+X^nY^m+X^{2n}Y^{2m}+\ldots +X^{(k-1)n}Y^{(k-1)m}$ divides $P_w$. 
\end{prop}
\begin{proof}
As before, since $a^kb^k\in \FF$, then $ab\in \FF$. Moreover $$a^kb^k=(ab) . b^{-1}(ab) b . b^{-2} (ab) b^2 \ldots b^{-k+1} (ab) b^{k-1}.$$ Then if $n,m$ are the total exponents of $x$ and $y$ in $b^{-1}$, $P_w=P_{ab}+X^nY^mP_{ab}+X^{2n}Y^{2m}P_{ab}+\ldots +X^{(k-1)n}Y^{(k-1)m}P_{ab}$ by Proposition \ref{basica}. 
\end{proof}

Most of the results in this section are consequences of Proposition \ref{lngeneral}. In Section \ref{secverbal} we will study products of $k$-th powers in general. Some of these results require that we develop first techniques to understand the winding invariants of commutators. This will be achieved in Section \ref{secparalelogramos}.

Proposition \ref{lngeneral} is related to the Burnside problem. Given $n,m\ge 1$, the Burnside group $B(n,m)$ is defined to be the group with generators $x_1,x_2,\ldots , x_n$ and with relators $w^m$ for each $w\in F(x_1,x_2,\ldots,x_n)$. \label{pagburnside} In particular $B(n,m)$ is finitely generated and every element has finite order. In general, a finitely generated torsion group need not be finite. This was proved by Golod and Shafarevich \cite{Gol} in 1964 and gave an answer to the General Burnside Problem. What about a finitely generated group $G$ such that each element $g\in G$ satisfies $g^m=1$ for a \textit{fixed} $m\ge 1$? Is such a group necessarily finite? This question, known as the the Burnside Problem was solved by Novikov and Adian \cite{NA}, who showed that in general a finitely generated group with bounded exponent $m$ can be infinite. Any group with $n$ generators and exponent $m$ is a quotient of $B(n,m)$. Therefore the Burnside Problem asks about finiteness of Burnside groups. It is known that $B(n,3)$, $B(n,4)$ and $B(n,6)$ are finite for every $n$, but it is still unknown whether $B(2,5)$ is.
The \textit{Engel words} are defined recursively by $e_1=[y,x]\in \F$ and $e_{n+1}=[y,e_n]$ for $n\ge 1$ \footnote{Some alternative definitions may be found in the literature, but they are all equivalent up to inversion and an automorphism of $\F$.}. It is not hard to see that the second Engel word $e_2$ is a product of cubes in $\F$. We recall the argument here. First note that $e_2$ is a conjugate of $[xy^{-1}x^{-1},y^{-1}]$. Now, in a group $H$ of exponent $3$ any element commutes with its conjugates. Burnside argument in \cite{Bur4} is the following. For $a,b\in H$, we have $1=(ab)^3=a^3(a^{-2}ba^2)(a^{-1}ba)b=(a^{-2}ba^2)(a^{-1}ba)b$. Replacing $a$ by $a^2$ we obtain $1=(a^{-1}ba)(a^{-2}ba^2)b$, so $(a^{-1}ba)b=b(a^{-1}ba)$. In particular, if $G$ is any group and $a,b\in G$, then $[a^{-1}ba,b]$ is trivial in the group $H=G/N$ where $N$ is the normal subgroup generated by all the cubes in $G$. Thus $[a^{-1}ba,b]\in G$ is a product of cubes in $G$. This proves that $e_2$ is a product of cubes in $\F$. Lyndon proved in \cite{Lyn4} that $[xy^{-1}x^{-1},y^{-1}]\in \F$ is a product of 4 cubes. Havas \cite{Hav4} and Akhavan-Malayeri \cite{Akh} show that in fact $e_2$ is a product of only $3$ cubes and that it is not a product of $2$. Concretely $$[xy^{-1}x^{-1},y^{-1}]=(xy^{-1}x^{-2})^3(x^{2}y)^3(y^{-1}x^{-1}y)^3.$$
Moreover, in \cite{Akh} it is proved that the equation $e_2^k=a^3b^3$ has a solution in $\F$ if and only if $3| k$. We will say more about the elements in $\F'$ which are products of cubes in Subsection \ref{subsecpowers}.

On the other hand, Havas proved \cite{Hav4} that $e_5$ is a product of $250$ fourth powers. Later, Korlyukov showed (unpublished) that $28$ fourth powers suffice. In Corollary \ref{etres} we will prove that $e_3$ is not a product of fourth powers. This can also be proved by showing that $e_3$ is nontrivial in $B(2,4)$ using coset enumeration for the presentations obtained in \cite{Hav4, New}, but our methods are simpler in this case.

Kostrikin proved that if the sixth Engel word $e_6$ is not a product of fifth powers, then $B(2,5)$ is infinite (see \cite[p. 184]{Kos}).

We we will deduce from Proposition \ref{lngeneral} that $e_6$ is not a product of two fifth powers and that $e_5$ is not a product of two fourth powers. Moreover, we give an alternative proof of the fact that $e_2$ is not a product of two cubes. In Proposition \ref{ppowers} we will prove that $e_6$ is a product of three fifth powers in the quotient $\F/\F''$. 

\begin{coro} \label{coroengel2}
The $n$-th Engel word $e_n$ is not a product of two $m$-th powers for any $n\ge 1$ and $m\ge 2$. Moreover, if $r\ge 1$ and $m\nmid r$, then $(e_n)^r$ is not a product of two $m$-th powers.
\end{coro}
\begin{proof}
By Proposition \ref{basica}, the winding invariant $P_{e_{n}}$ is $(Y-1)P_{e_{n-1}}$ for $n\ge 2$. Since $P_{e_1}=-1\in \ZZ$, we conclude that $P_{e_n}=-(Y-1)^{n-1}$ and then $P_{e_n^r}=-r(Y-1)^{n-1}$. By Proposition \ref{lngeneral}, if $e^r_n=a^mb^m$ for certain $a,b\in \F$, then there exist $k,l\in \Z$ such that $1+X^kY^l+X^{2k}Y^{2l}+\ldots +X^{(m-1)k}Y^{(m-1)l}$ divides $P_{e^r_n}$ in $\ZZ$. Putting $X=1$, we have that $Q=1+Y^l+Y^{2l}+\ldots +Y^{(m-1)l}$ divides $r(Y-1)^{n-1}$ in $\Z [Y^{\pm 1}]$. Since $Y^{(m-1)l}$ is a unit, we can assume $l\ge 0$ and then $Q$ divides $r(Y-1)^{n-1}$ in $\Z[Y]$. Since $m\nmid r$, then $l\neq 0$. Thus $Y^{ml}-1| r(Y^l-1)(Y-1)^{n-1}$, which is absurd since $m\ge 2$.
\end{proof}

Another direct consequence of Proposition \ref{lngeneral} is the following.

\begin{coro} \label{npowers2}
If $w\in \FF$ is a product of two $k$-th powers in $\F$, then $k$ divides $P_w(1,1)$.
\end{coro}

In Theorem \ref{npowers} we will see that Corollary \ref{npowers2} holds also if we replace the word ``two'' by any positive integer, when $k$ is odd. If $k$ is even, in this general version we will only be able to conclude that $\frac{k}{2}$ divides $P_w(1,1)$.

\begin{coro}
Let $r,s,k \ge 1$. Then $w=[x^r,y^s]$ is a product of two $k$-th powers in $\F$ if and only if $k$ divides $r$ or $s$. 
\end{coro}
\begin{proof}
If $k | r$, say $r=nk$, then $w=(x^n)^k(y^sx^{-n}y^{-s})^k$. Similarly if $k|s$, $w$ is a product of two $k$-th powers.

Conversely, suppose $w$ is a product of two $k$-th powers. Then there exist $n,m\in \Z$ such that $1+X^nY^m+X^{2n}Y^{2m}+\ldots +X^{(k-1)n}Y^{(k-1)m}$ divides $P_w$. On the other hand it is easy to see that $P_w=(1+X+X^2+\ldots +X^{r-1})(1+Y+Y^2+\ldots +Y^{s-1})$. Putting $Y=1$, we deduce that $Q=1+X^n+\ldots +X^{(k-1)n}$ divides $s(1+X+X^2+\ldots +X^{r-1})$ in $\Z[X^{\pm 1}]$. We can assume $n\ge 0$. Thus $Q$ divides $sP$ in $\Z[X]$, where $P=1+X+X^2+\ldots +X^{r-1}$. Then $X^{kn}-1$ divides $(X^r-1)(X^n-1)$. If $n\ge 1$ and $k\ge 2$, this implies that $kn| r$ and we are done. If $k=1$, then $k|r$. Therefore, if $k\nmid r$, then $n=0$. Symmetrically, if $k\nmid s$, then $m=0$. Thus $k| P_w$ in $\ZZ$, so $k|1$, a contradiction.  
\end{proof}

A very simple but useful idea to extend our results on $\F$ to study the free group $\mathbb{F}_n$ of rank $n$ is to consider a map $p:\mathbb{F}_n\to \F$. Then the composition $Wp|_{\mathbb{F}_n'}: \mathbb{F}_n' \to \ZZ$ shares many of the properties of the winding invariant.

\begin{ej}
Let $\{x,y,z\}$ be a free basis of $\mathbb{F}_3$. The word $w=xy^2zx^{-1}y^{-2}z^{-1}$ is not a product of two squares in $\mathbb{F}_3$. If we consider the homomorphism $p:\mathbb{F}_3\to \F=\langle x,y\rangle$ which fixes $x,y$ and maps $z$ to $1$, $p(w)=[x,y^2]$ is a product of two squares in $\F$. However if $p$ fixes $x,y$ and maps $z$ to $y^{-1}$, instead, then $p(w)=[x,y]$ is not a product of two squares. Moreover $W(p(w))=1$ is not divisible by a polynomial of type $1+X^nY^m$, so Proposition \ref{lngeneral} applies.  
\end{ej}

In \cite[Theorem 2]{LN}, Lyndon and Newman prove that if $\mathbb{F}_n$ is the free group of rank $n$ freely generated by $x_1,x_2,\ldots, x_n$, then $x_1^2x_2^2\ldots x_n^2$ is not a product of fewer that $n$ squares. Their result follows also from the subsequent ideas in \cite{Cul, Edm, Edm2, GT}. The following is a generalization of the case $n=3$.

\begin{teo} \label{lnnuevo}
Let $\mathbb{F}_n$ be the free group of rank $n\ge 3$ freely generated by $x_1,x_2, \ldots, x_n$. Let $m\ge 1$ and let $k\ge 2$. Then the equation $$x_1^mx_2^m\ldots x_n^m=a^kb^k$$ has no solution $a,b\in \mathbb{F}_n$.  
\end{teo}  
\begin{proof}
Suppose the equation has a solution. Using the map $\mathbb{F}_n\to \F$ which maps $x_1$ to $x$, $x_2$ to $y$, $x_3$ to $x^{-1}y^{-1}$ and $x_i$ to $1$ for each $i\ge 4$, we deduce that the equation $x^my^m(x^{-1}y^{-1})^m=a^kb^k$ has a solution $a,b\in \F$. The winding invariant of the left hand side (Figure \ref{escalera}) is $$P=\sum\limits_{i=0}^{m-1} Y^iX^i(1+X+X^2+\ldots +X^{m-i-1}).$$

\begin{figure}[h] 
\begin{center}
\includegraphics[scale=0.5]{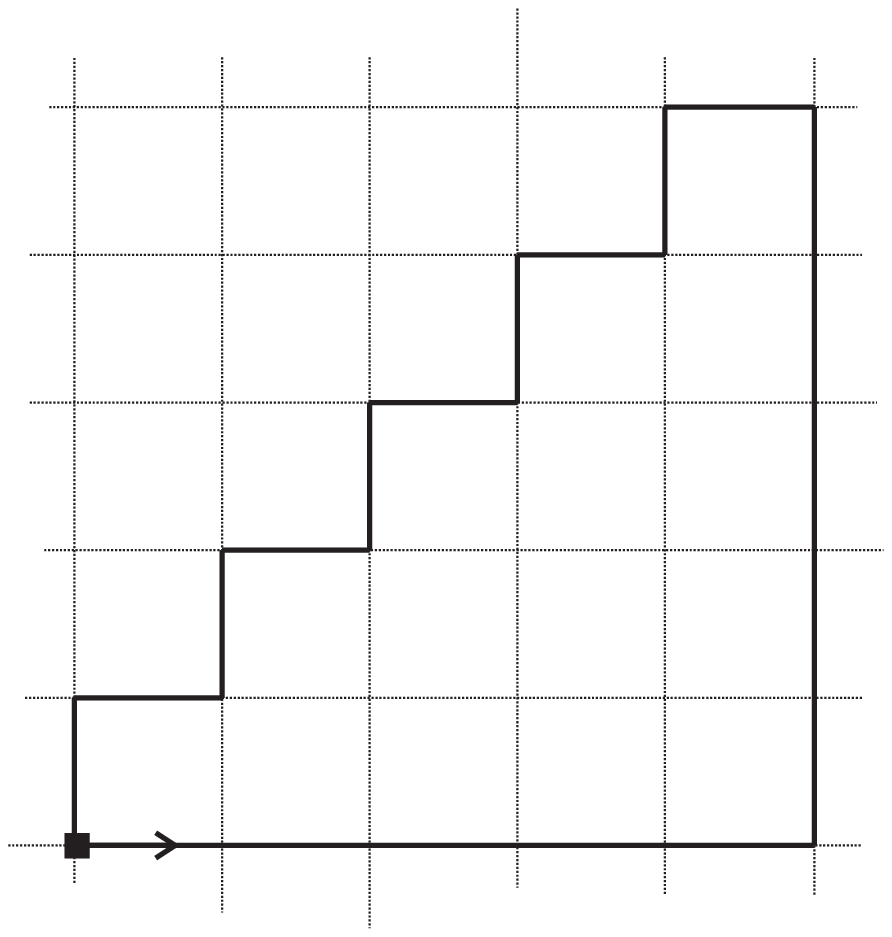}
\caption{The curve $\gamma_{w}$ for $w=x^5y^5(x^{-1}y^{-1})^5$.}\label{escalera}
\end{center}
\end{figure}

By Proposition \ref{lngeneral} there exist $r,s\in \Z$ such that $1+X^rY^s+X^{2r}Y^{2s}+\ldots +X^{(k-1)r}Y^{(k-1)s}$ divides $P$.
Taking $Y=X^{-1}$ we deduce that $Q=\sum\limits_{i=0}^{k-1} X^{i(r-s)}$ divides $P(X,X^{-1})=\sum\limits_{i=0}^{m-1} (1+X+X^2+\ldots +X^{m-i-1})$ in $\Z[X,X^{-1}]$. Multiplying by $(X-1)$ we obtain that $Q$ divides $\sum\limits_{i=0}^{m-1} (X^{m-i}-1)=X+X^2+\ldots+ X^{m}-m$. Since $k\nmid X+X^2+\ldots+ X^{m}-m$, then $r-s\neq 0$. Let $\xi \in \mathbb{C}$ be a primitive $(k|r-s|)$-th root of unity. Then $Q(\xi)=\frac{\xi^{k(r-s)}-1}{\xi^{r-s}-1}=0$, so $\xi+\xi^2+\ldots+\xi^m-m=0$. Then $m=|\xi+\xi^2+\ldots+\xi^m|\le |\xi|+|\xi^2|+\ldots+|\xi^m|=m$ implies that $\xi$ and $\xi^2$ have the same argument, so $\xi$ is a positive real number, a contradiction.   
\end{proof}

Conway and Lagarias prove in \cite[Theorem 5.1]{CL} that the winding invariant map $\FF \to \ZZ$ is surjective and its kernel is $\FFF$, the commutator of $\FF$. We exhibit here a much shorter proof. There is a shorter proof yet which is explained in Section \ref{sectionhow} and uses the fact that $\W$ is just the projection of $\pi_1(\widetilde{K})$ onto $H_1(\widetilde{K})$ for the covering $\widetilde{K}$ of the figure eight space $K$ corresponding to the subgroup $\F' \leqslant \pi_1(K)$. 

\begin{teo}[Conway-Lagarias] \label{conway}
The winding invariant map $\W: \FF\to \ZZ$ is a surjective group homomorphism whose kernel is $\FFF$.
\end{teo}
\begin{proof}
Proposition \ref{basica} (ii) says that $\W$ is a homomorphism. If $n,m\in \Z$, then $w=x^ny^m[x,y]$ $y^{-m}x^{-n}\in \FF$ satisfies $P_w=X^nY^m$. This proves that $\W$ is surjective. Since $\ZZ$ is abelian, the commutator $\FFF$ of $\FF$ is contained in $\textrm{ker}(\W)$. We prove now that $\textrm{ker}(\W) \subseteq \FFF$. Let $w\in \textrm{ker}(\W)$. Since $\FF$ is the normal subgroup of $\F$ generated by $[x,y]$, then 
\begin{equation}\label{area}
w=\prod\limits_{i=1}^k u_i[x,y]^{\epsilon_i}u_i^{-1}
\end{equation}
for certain $k\ge 0$, $u_i\in \F$ and $\epsilon_i=\pm 1$. We will prove that the class $\overline{w}$ of $w$ in $\FF / \FFF$ is trivial by induction on $k$. If $k=0$, $w=1$. Suppose $k\ge 1$. By Proposition \ref{basica}, $P_w=\sum\limits_{i=1}^k \epsilon_iX^{n_i}Y^{m_i}$, where $n_i$ and $m_i$ are the total exponents of $x$ and $y$ in $u_i$. Since $P_w=\W(w)=0$, there exists $2\le i\le k$ such that $\epsilon_i=-\epsilon_1$ and $(n_i,m_i)=(n_1,m_1)$. Since the factors of (\ref{area}) lie in $\FF$, they commute in $\FF/\FFF$. Therefore $\overline{w}$ coincides with the class of $$w'=u_1[x,y]^{\epsilon_1}u_1^{-1}u_i[x,y]^{-\epsilon_1}u_i^{-1}u$$ where $u=\prod\limits_{2\le j\le k, j\neq i} u_j[x,y]^{\epsilon_j}u_j^{-1}$. Since $(n_i,m_i)=(n_1,m_1)$, $u_1^{-1}u_i\in \FF$, so $\overline{w'}$ coincides with the class of $$u_1[x,y]^{\epsilon_1}[x,y]^{-\epsilon_1}u_1^{-1}u_iu_i^{-1}u=u.$$

Since $\overline{w}=\overline{u}$, $u\in \textrm{ker}(\W)$. By induction $\overline{w}=\overline{u}=0$. 
\end{proof}

\begin{obs}
Theorem \ref{conway} provides an algorithm for deciding whether an element $w\in \F$ is in $\FFF$.
\end{obs}

Theorem \ref{conway} says that the winding invariant is useful to study equations in the free metabelian group $\M=\F/\F''$ of rank $2$. Recall that a group $G$ is said to be metabelian if $G'=[G,G]$ is abelian. The free metabelian group of rank $k$ is $\mathbb{F}_k/\mathbb{F}_k''$. Theorem \ref{conway} says then that $\W$ induces a well-defined bijective map $\M' \to \ZZ$ which we will also denote $\W$. 
Recall that an equation over a group $G$ in the variables $z_1,z_2,\ldots ,z_n$ is a word $w\in G*F(z_1,z_2,\ldots, z_n)$. A solution of $w=1$ in $G$ is an $n$-tuple $a_1,a_2,\ldots, a_n\in G$ such that $w$ lies in the kernel of the homomorphism $G*F(z_1,z_2,\ldots,z_n)\to G$ which maps $G$ identically to $G$ and each $z_i$ to $a_i$. Similarly a solution of $w=1$ in an overgroup $H\ge G$ is an $n$-tuple in $H$ which induces a map $G*F(z_1,z_2,\ldots,z_n)\to H$ which maps $w$ to $1\in H$. It is worth mentioning that the general solvability problem for equations in the free metabelian group of rank $2$ is algorithmically unsolvable \cite{Rom2}. However, for some \textit{quadratic equations} the problem is solvable \cite{LU1, LU2} and other problems over metabelian groups are known to be algorithmically solvable (see \cite{Bau8}, \cite{Bau7} and \cite[Section 9.5]{LR}, for example). 

Our proof of Proposition \ref{ln} shows that the equation $[x,y]=z_1^2z_2^2$ does not have a solution in $\F$ and moreover, it does not even have a solution in $\M$. The following result is a converse of Proposition \ref{lngeneral} which characterizes the elements $w\in \F'$ such that $w=z_1^kz_2^k$ has a solution in $\M$.

\begin{prop} \label{lnconverse}
Let $k$ be a positive integer and let $w\in \FF$. Then the following are equivalent

\noindent (i) There exist $a,b\in \F$ such that $w=a^kb^k$ in $\M$.

\noindent (ii) There exist $n,m\in \Z$ such that $1+X^nY^m+X^{2n}Y^{2m}+\ldots +X^{(k-1)n}Y^{(k-1)m}$ divides $P_w$ in $\ZZ$.
\end{prop}
\begin{proof}
The proof of the implication $(i)\Rightarrow (ii)$ is the proof of Proposition \ref{lngeneral} together with the fact that $\FFF$ is the kernel of the winding invariant map.

Conversely, suppose $(ii)$ holds. And let $P\in \ZZ$ be such that $P_w=(1+X^nY^m+X^{2n}Y^{2m}+\ldots +X^{(k-1)n}Y^{(k-1)m})P$. Since the winding invariant map is surjective, there exists $u\in \FF$ with $P_u=P$. Let $b=x^{-n}y^{-m}$, $a=ub^{-1}\in \F$. By the proof of Proposition \ref{lngeneral}, $P_{a^kb^k}=P_w$. By Theorem \ref{conway}, $a^kb^k$ and $w$ differ in an element of $\FFF$. 
\end{proof}

The proof of Corollary \ref{coroengel2} shows that no Engel word is a product of two $k$-th powers in $\M$. In Proposition \ref{ppowers} we will see that Engel words can be products of three $k$-th powers in $\M$.
The fact that for a metabelian group $G$, $G'$ has a structure of $G/G'$-module can be traced back to Philip Hall \cite[2.3]{PHal}. Baumslag and Mahler \cite{BM6} implicitly used the isomorphism $W:\F'/\F''\to \ZZ$ to show that if $w,u \in \F$ are linearly independent in $\F/\F'$ and $k\ge 2$, then the equation $w^ku^k=z^k$ has no solution $z$ in $\F/\F''$. This result is a version for free metabelian groups of the famous Vaught conjecture that $w^2u^2=z^2$ for $w,u,z$ in a free group $F$ implies $wu=uw$ \cite{Lyn2}. In Proposition \ref{propbm} we give a simple proof of Baumslag and Mahler's result for $k=2$. The Vaught conjecture holds as well when the exponent $2$ is replaced by any $k\ge 2$ \cite{Sch} and when the three exponents in the equation are greater than or equal to $2$, even if they are different \cite{LS}. In $\F/\F''$, however, the equation $w^ku^l=z^m$ may have solutions when $w$ and $u$ are independent modulo $\F'$ \cite{Lyn3}. Other important open conjectures of equations over free groups which have been studied for metabelian groups are the Kervaire-Laudenbach conjecture and the Levin conjecture \cite{Rom}.

\section{Winding invariants for groups with relators in the commutator subgroup} \label{subcoco}

The winding invariant may be used to study groups other than free and free metabelian. Suppose $G$ is a group generated by two elements, say $G$ is presented by $\langle x,y | S \rangle$ for some subset $S\subseteq \F$. Moreover, suppose $S\subseteq \FF$. Then $G'=\FF / N(S)$, where $N(S) \vartriangleleft \F$ is the normal subgroup generated by $S$. Let $I=I(\W (S))=\langle P_{s} \rangle _{s\in S}$ be the ideal of $\ZZ$ generated by the polynomials $P_{s}$ with $s\in S$, and let $q:\ZZ \to \ZZ /I$ be the quotient map. Then $q\W:\FF \to \ZZ/I$ induces a map $\overline{\W}:G' \to \ZZ /I$. This is the \textit{winding invariant map} of $G$. We will see that $\ZZ / I$ is well defined up to isomorphism (given two presentations $\langle x,y| S\rangle$ and $\langle x,y| T \rangle$ of $G$ with $S,T\subseteq \FF$, $\ZZ / I(\W(S))$ and $\ZZ / I(\W(T))$ are isomorphic) and so is the winding invariant $\overline{W}$.


\begin{lema} \label{pregunta}
Let $S$ be a subset of $\FF$. An element $w\in \FF$ is in the subgroup $N(S)\F ''$ if and only if $P_w\in I(\W(S))$.
\end{lema}
\begin{proof}
The inclusion $\W(N(S)\F '')\subseteq I(\W(S))$ follows from the fact that $\W$ is a homomorphism. Conversely, suppose $P_w\in I(\W(S))$, that is $P_w=Q_1P_{s_1}+Q_2P_{s_2}+\ldots +Q_nP_{s_n}$ for certain $s_1, s_2, \ldots, s_n\in S$ and $Q_1, Q_2, \ldots, Q_n\in \ZZ$. Note that if $s\in S$ and $k,l\in \Z$, then $P_{x^ky^lsy^{-l}x^{-k}}=X^kY^lP_{s}$. Therefore if $Q=\sum\limits_{k,l} a_{k,l}X^kY^l \in \ZZ$, then $u=\prod\limits_{k,l} (x^ky^lsy^{-l}x^{-k})^{a_{k,l}} \in \FF$ satisfies $P_u=QP_{s}$ and $u$ is in the normal subgroup of $\F$ generated by $s$ (the element $u$ is in fact well-defined up to rearrangement of the factors in the product). We do this for every term $Q_iP_{s_i}$ and obtain $w'\in N(S)$, such that $P_{w'}=P_w$. Then by Theorem \ref{conway}, $w\in N(S)\F ''$.  
\end{proof}

\begin{prop} \label{conwayg}
The winding invariant $\overline{\W}:G' \to \ZZ /I$ is a surjective group homomorphism and its kernel is $G''$.
\end{prop}
\begin{proof}
The fact that $\overline{\W}$ is surjective follows from the surjectivity of $\W:\FF \to \ZZ$. If $g\in G''$, then $g$ is the class $\overline{w}$ of a word $w\in \F ''$, so $\overline{\W}(g)=q\W(w)=0$ by Theorem \ref{conway}. Conversely, let $g\in \textrm{ker}(\overline{\W})$. Let $w\in \FF$ be such that $\overline{w}=g$. Then $q\W(w)=\overline{\W}(g)=0$. Thus $P_w=\W(w)\in I$. By Lemma \ref{pregunta}, $w\in N(S)\F ''$ and then $g=\overline{w} \in G''$.   
\end{proof}

In particular, if $\langle x,y| S\rangle$ and $\langle x,y | T \rangle$ are two presentations of $G$ with $S,T\subseteq \F'$, then $\ZZ /I(W(S))$ and $\ZZ /I(W(T))$ are isomorphic groups. In fact, using a result by Dunwoody (\cite[Theorem 4.10]{Dun}), it is possible to show that these two rings are isomorphic rings. This is proved in Proposition \ref{isorings} below.

Recall that the Thompson group $F$ admits the following presentation $\langle x,y | [xy^{-1},x^{-1}yx],$ $[xy^{-1},x^{-2}yx^2]\rangle$. One important open question about $F$ is whether this group is amenable. It is known that $F$ is not simple, but its commutator subgroup $F'$ is. Here we give a short proof of the fact that $F'$ is perfect using the last result.
 
\begin{prop}
The commutator subgroup of the Thompson group $F$ is perfect.
\end{prop}
\begin{proof}
Call $r_1=[xy^{-1},x^{-1}yx]$ and $r_2=[xy^{-1},x^{-2}yx^2]$ the relators of the presentation of $F$. Then $P_{r_1}=1+X^{-1}-Y^{-1}$, $P_{r_2}=1+X^{-1}+X^{-2}-Y^{-1}-X^{-1}Y^{-1}$ (see Figure \ref{thompson}).

\begin{figure}[h] 
\begin{center}
\includegraphics[scale=0.6]{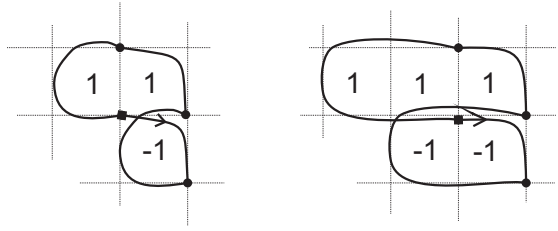}
\caption{The curves $\gamma_{r_1}$ and $\gamma_{r_2}$ together with the polynomials $P_{r_1}$ and $P_{r_2}$.}\label{thompson}
\end{center}
\end{figure}

Note that $(1+X)P_{r_1}-XP_{r_2}=1$. Therefore the ideal $I=\langle P_{r_1},P_{r_2} \rangle$ is the whole ring $\ZZ$, so $F'=\textrm{ker}(\overline{\W})=F''$ by Proposition \ref{conwayg}. That is, $F'$ is perfect.
\end{proof}

The previous example obviously generalizes to the following

\begin{coro} \label{gprimasimple}
Let $G$ be a group which admits a presentation $\langle x,y| S \rangle$ with $S\subseteq \FF$. Then $G$ is quasi-perfect (i.e. $G''=G'$) if and only if  $I(\W(S))=\ZZ$. In particular, if $I(\W(S))=\ZZ$ and $G'\neq 1$, $G$ is not residually solvable.
\end{coro}

In Subsection \ref{resfin} we will use a deeper result to obtain a class of non residually solvable one-relator groups.

The next result is an improved version of Proposition \ref{ln}, which applies for instance to groups presented by basic commutators of weight greater than or equal to 3 (see Lemma \ref{basiccom} below).

\begin{prop} \label{lnfuerte}
Let $G$ be a group presented by a presentation $\langle x,y | S\rangle$ with $S\subseteq \F'$. Moreover, suppose $P_{s}(1,1)\in \Z$ is even for every $s\in S$. Then the class of $[x,y]$ is not a product of two squares in $G$.
\end{prop}
\begin{proof}
The map $\epsilon: \ZZ \to \Z_2$ which evaluates in $(1,1)$ and reduces modulo $2$ induces a map $\overline{\epsilon}:\ZZ / I(\W(S))\to \Z_2$, by hypothesis. Let $g\in [G,G]$ be such that there exist $g_1,g_2\in G$ with $g=g_1^2g_2^2$. Let $u_1,u_2\in \F$ be representatives of $g_1$ and $g_2$. Let $u\in \FF$ be a representative of $g$. Thus $N(S)\ni u_1^2u_2^2u^{-1}=(u_1u_2)u_2^{-1}(u_1u_2)u_2u^{-1}$. Since $N(S)\vartriangleleft \FF$ and $u\in \FF$, then $u_1u_2\in \FF$. Therefore $\overline{\epsilon}\overline{\W}(g)=\epsilon \W ((u_1u_2)u_2^{-1}(u_1u_2)u_2)=\epsilon ((1+X^nY^m)P_{u_1u_2})=0\in \Z_2$. Here $n,m$ are the total exponents of $x,y$ in $u_2^{-1}$. On the other hand if $g=\overline{[x,y]}$, $\overline{\epsilon}\overline{\W}(g)=\epsilon \W([x,y])=\epsilon (1)=1\neq 0$.
\end{proof}

In the last result of this section we use the winding invariant $\overline{\W}:G\to \ZZ/I$ to give a necessary condition for an element $g\in G$ to be in the center $Z(G)$.


\begin{prop}
Let $G$ be a group with a presentation $\langle x,y | S \rangle$ for some $S\subseteq \FF$. If an element $g\in G$ is in the center of $G$, then any word $w\in \F$ representing $g$ satisfies the following. If $n,m$ are the total exponents of $x$ and $y$ in $w$, then $X^nY^m-1\in I(W(S))$. 
\end{prop}
\begin{proof}
Suppose $\overline{w}=g\in Z(G)$. Let $u=[x,y]w^{-1}\in \F$. By assumption $\overline{[w,u]}=1 \in G$, so $0=\overline{\W}(\overline{[w,u]})=q\W([w,u])$. But $\W([w,u])=P_{[w,u]}=P_{[w,[x,y]]}=(X^nY^m-1)P_{[x,y]}=X^nY^m-1$, by Proposition \ref{basica}. Then $q(X^nY^m-1)=0 \in \ZZ /I$.
\end{proof}

Given a group $G$ presented by a presentation with two generators and relators in $\FF$ there is a method which decides for every $w\in \F$ whether $g=\overline{w}\in G$ is in $G''$. First, $g\in [G,G]$ if and only if $w\in \FF$. Now, if $g\in G'$, then by Proposition \ref{conwayg}, $g\in G''$ if and only if $0=\overline{W}(g)$, which is equivalent to $P_w$ being in $I$, the ideal of $\ZZ$ generated by the winding invariants of the relators in the presentation. Now, ideal membership can be determined using Gr\"obner bases. Usually Gr\"obner bases are used for studying ideals in the polynomial ring $\Z [X,Y]$, but the ideas can be extended to Laurent polynomial rings (see \cite[Section 10.7]{Sim2}. Section 11.6 of the same book studies polycyclicity of the group $G/G''$ using Gr\"obner bases).     

\section{Lower central series and a problem of G. Baumslag and R. Mikhailov}

Recall that the lower central series of a group $G$ is defined inductively by $\gamma_1(G)=G$ and $\gamma_{n+1}(G)=[G,\gamma _n(G)]$. In particular $\gamma_2(\F)=\FF$ and $\gamma_3(\F)=[\F,\FF]$.

Wick's Theorem \ref{wicks} provides an algorithm for deciding whether an element $w\in \F$ is a commutator. On the other hand, Philip Hall's collecting process \cite[Chapter 11]{Hal} gives an algorithm that decides whether a given word $w$ lies in $\gamma_n(\F)$ for any given $n$. The following result gives a necessary condition for a word $w$ to be a simple commutator $[a_1,a_2,\ldots , a_k]$. Recall that the simple commutator $[a_1,a_2,\ldots , a_k]$ of weight $k$ is defined recursively as $[a_1,[a_2,\ldots , a_k]]$. Of course, the simple commutators of weight $k$ lie in $\gamma_k(\F)$ and for $k=2$ they are just the usual commutators.

\begin{prop} \label{lowercs}
Let $w\in \FF$ and let $k\ge 1$. Then the following are equivalent:

\noindent (i) There exist $a_1,a_2,\ldots , a_{k}\in \F$ and $u\in \FF$ such that $w=[a_1,a_2,\ldots , a_k,u]$ in $\M=\F/\FFF$.

\noindent (ii) There are monomials $M_i=X^{n_i}Y^{m_i}$ for $i=1,2,\ldots, k$ such that $(M_1-1)(M_2-1)\ldots (M_{k}-1)$ divides $P_w$ in $\ZZ$.
\end{prop}
\begin{proof} Suppose (i) holds. Then $P_w=P_{[a_1,a_2,\ldots , a_k,u]}$ by Theorem \ref{conway}. By Proposition \ref{basica}, $P_w=(X^{n_1}Y^{m_1}-1)P_{[a_2,\ldots , a_k, u]}$, where $n_1,m_1$ are the total exponents of $x$ and $y$ in $a_1$. By induction $P_w=(M_1-1)(M_2-1)\ldots (M_{k}-1)P_u$. Here $M_i=X^{n_i}Y^{m_i}$ where $n_i,m_i$ are the total exponents in $a_i$.

Conversely if (ii) holds, there exists $P\in \ZZ$ such that $P_w=(M_1-1)(M_2-1)\ldots (M_{k}-1)P$. By Theorem \ref{conway}, there exists $u\in \FF$ with $P_u=P$. Then $$[x^{n_1}y^{m_1},x^{n_2}y^{m_2}, \ldots x^{n_k}y^{m_k}, u]$$ has the same winding invariant as $w$, so by Theorem \ref{conway} they are equal modulo $\FFF$.
\end{proof} 

%

\begin{coro} \label{derivadas}
Let $k\ge 3$ and let $w\in \gamma_k(\F)$. Then $$ \frac{\partial ^{i} P_w}{\partial ^j X \partial ^{i-j} Y} (1,1)=0$$ for every $0\le i\le k-3$, $0\le j\le i$.
\end{coro}
\begin{proof}
Since $w\in \gamma_k(\F)=[\F, \gamma_{k-1}(\F)]$, $w$ is a product of commutators $[a,u]$ with $a\in \F$ and $u\in \gamma_{k-1}(\F)$. Therefore we may assume $w=[a,u]$ is one such commutator. By Proposition \ref{basica}, $P_w=(X^nY^m-1)P_u$, and then the corollary follows by induction.
\end{proof}

Let $w\in \FF$ and suppose $P_{w}=\sum\limits_{i,j} a_{i,j} X^iY^j \neq 0$ . Let $(n,m)\in \Z^2$ be such that $a_{n,m}\neq 0$ and $a_{i,j}=0$ for every $(n,m)\neq (i,j)\in \Z^2$ with $i\ge n$ and $j\ge m$. There is always at least one pair $(n,m)$ with this property. On the other hand, let $(n',m')$ be such that $a_{i,j}=0$ for every $(i,j)\in \Z^2$ with $i< n'$ or $j< m'$. Then $w\notin \gamma_{n+m-n'-m'+3}(\F)$. Otherwise, $u=x^{-n'}y^{-m'}wy^{m'}x^{n'}$ would also be in $\gamma_{n+m-n'-m'+3}(\F)$ and $P_u=X^{-n'}Y^{-m'}P_w\in \Z[X,Y]$. But $\frac{\partial ^{n+m-n'-m'} P_{u}}{\partial ^{n-n'} X \partial ^{m-m'} Y} (1,1)=(n-n')!(m-m')!a_{n,m}\neq 0$, which contradicts Corollary \ref{derivadas}.

Note that in particular we deduce that for every $w\in \FF \smallsetminus \FFF$, there exists $k\ge 3$ such that $w$ is not in $\gamma_k(\F)$. In fact any free group $F$ is residually nilpotent, i.e. $\bigcap\limits_{k\ge 1} \gamma_k(F)=0$. The inclusion $\bigcap\limits_{k\ge 1} \gamma_k(G)\subseteq G''$ however, does not hold for every group. For instance the alternating group $A_4$ is metabelian ($A_4''=0$) but not nilpotent.

\begin{obs}
$\F$ is not nilpotent, and moreover we can easily prove that $\gamma_k(\F) \nsubseteq \FFF$ for every $k\ge 1$. We may assume that $k\ge 2$. The Engel word $e_{k-1}=[\underbrace{y,y,\ldots, y}_{k-1},x]\in \gamma_k(\F)$ and $P_{e_{k-1}}=-(Y-1)^{k-2}\neq 0$. By Theorem \ref{conway}, $w\notin \FFF$.
\end{obs}

In \cite{BM2}, G. Baumslag and R. Mikhailov proposed the following problem: are all one-relator groups $G$ defined by basic commutators residually torsion-free nilpotent? The aim of this section is to prove that the $\textrm{mod } G''$-version of this holds. First we recall the definitions of the concepts involved.

Recall that $[a,b]$ denotes $aba^{-1}b^{-1}$. Let $\mathbb{F}_n$ be the free group with generators $x_1,x_2,\ldots,$ $x_n$. The \textit{basic commutators of weight 1} are the inverses $x_i^{-1}$ of the generators \footnote{Because of our notation for commutators, we have chosen our definition of basic commutator in such a way that it differs from the standard by the automorphism of $\mathbb{F}_n$ which inverts each generator $x_i$.}. Suppose that $m>1$, that the set of basic commutators of weight $<m$ has already been constructed and, moreover, a total order in that set has been defined. A \textit{basic commutator of weight $m$} is a commutator $[u^{-1},v^{-1}]$ where $u$ and $v$ are basic commutators of weights $l,k<m$ such that $l+k=m$, $u>v$, and, if $u$ is a basic commutator $[u_1^{-1},u_2^{-1}]$ of weight $>1$, then $v\ge u_2$. Extend the total order to the set of basic commutators of weight $\le m$ in such a way that every commutator of weight $m$ is greater than any commutator of weight $<m$. For instance, for $n=2$, $\F$ is generated by $x,y$ and if we order $x^{-1}>y^{-1}$, then the basic commutators of weight $\le 3$ are $y^{-1},x^{-1}, [x,y], [[x,y]^{-1},y], [[x,y]^{-1},x]$. Magnus proved \cite{Mag} that the basic commutators of weight $m$ are a basis of the free abelian group $\gamma_m(\mathbb{F}_n) / \gamma_{m+1}(\mathbb{F}_n)$. By means of the collecting process, any word $w\in \mathbb{F}_n$ can be rewritten as a product $c_1^{n_1} c_2^{n_2} \ldots c_k^{n_k}$ where $c_1<c_2<\ldots <c_k$ are basic commutators and $n_1, n_2,\ldots ,n_k$ are integers. This process can be used to determine the maximum number $m$ such that $w\in \gamma_m(\mathbb{F}_n)$. For more details on basic commutators see \cite[Chapter 11]{Hal}.

\begin{lema} \label{basiccom}
If $w\in \F$ is a basic commutator of weight $m>1$, $P_w=(X-1)^l(Y-1)^s$ for certain $l,s\ge 0$ with $l+s=m-2$ or $P_w=0$. Moreover, every polynomial $(X-1)^l(Y-1)^s$ with $l,s\ge 0$, $l+s=m-2$, is the winding invariant of a basic commutator of weight $m$. 
\end{lema}
\begin{proof}
If the weight of $w$ is 2, $w=[x,y]$, so $P_w=1$. If the weight of $w$ is greater that 2, then $w=[u^{-1},v^{-1}]$ for $u$ a basic commutator of weight $\ge 2$. If the weight of $v$ is also $\ge 2$, then $w\in \FFF$ and $P_w=0$. Otherwise, $v=x^{-1}$ or $v=y^{-1}$, so $P_w=(X-1)P_u$ or $P_w=(Y-1)P_u$, and by induction $P_w$ is of the form $(X-1)^l(Y-1)^s$ with $l+s=m-2$. Conversely, let $l,s\ge 0$ with $l+s=m-2$. Define recursively $u_0=[x,y]$, $u_{k+1}=[u_k^{-1}, y]$ for $k\ge 0$. Then $u_k$ is a basic commutator for every $k\ge 0$. Now, let $v_0=u_s$, $v_{k+1}=[v_k^{-1},x]$ for every $k\ge 0$. Then $v_k$ is a basic commutator of weight $s+2+k$ for every $k$ and $P_{v_l}=(X-1)^l(Y-1)^s$.    
\end{proof}

\begin{defi}
A group $G$ is said to be \textit{residually torsion-free nilpotent} if $\bigcap\limits_{k\ge 1} \tau_k(G)=\{1\}$, where $\tau_k(G)=\{g\in G | g^n\in \gamma_k(G) $ for some $n\ge 1 \}$. Equivalently $G$ is residually torsion-free nilpotent if for every $1\neq g\in G$ there exists a normal subgroup $N\vartriangleleft G$ not containing $g$ such that $G/N$ is torsion-free and nilpotent.
\end{defi}

In particular, residually torsion-free nilpotent implies residually nilpotent. Magnus proved \cite{Mag} that finitely generated free groups are residually torsion-free nilpotent. The Hydra groups introduced by Dison and Riley in \cite{DR} are one-relator groups with two generators and where the relator is a basic commutator. Baumslag and Mikhailov \cite{BM2} prove that these and more general groups are residually torsion-free nilpotent but leave as an open question whether all one-relator groups defined by a basic commutator have this property. The main result of this section is the following

\begin{teo} \label{teobm}
Let $G=\langle x,y | r\rangle$, where $r\in \F$ is a basic commutator. Then $\bigcap\limits_{k\ge 1}\tau_k(G) \subseteq G''$.
\end{teo}
\begin{proof}
If the weight of $r$ is 1, $G$ is an infinite cyclic group. We may assume $r\in \FF$. Let $g\in \bigcap\limits_{k\ge 1}\tau_k(G)$ and let $w\in \F$ be such that its class $\overline{w}$ in $G$ is $g$. Since $g\in \tau_2(G)$, there exists $n=n(2)\ge 1$ such that $g^n\in G'$. Thus, $\overline{w^n}=g^n=\overline{w}_2$ for certain $w_2\in \FF$. Since $r\in \FF$, then $w\in \FF$. We want to prove that $g\in G''$. We may assume that $P_w\in \Z[X,Y]$. Otherwise, take $p,q\ge 0$ such that $P_{x^py^qwy^{-q}x^{-p}}=X^pY^qP_w\in \Z[X,Y]$. Then $g'=\overline{x^py^qwy^{-q}x^{-p}}\in G$ is a conjugate of $g$, so $g' \in \bigcap\limits_{k\ge 1}\tau_k(G)$ and if we prove $g'\in G''$, we will have $g\in G''$.

Given any $k\ge 1$, since $g\in \tau_k(G)$, there exists $n=n(k)\ge 1$ such that $g^n\in \gamma_k(G)$. Therefore, for some $w_k\in \gamma_k(\F)$, $\overline{w^n}=\overline{w}_k$, so $w^nw_k^{-1}$ is contained in the normal subgroup $N(r)$ of $\F$ generated by $r$. In particular 
\begin{equation}
nP_w-P_{w_k}=P_{w^nw_k^{-1}}=Q_kP_r  	
\label{eqhydra}
\end{equation}
for certain $Q_k\in \ZZ$.

Let $P_w=\sum\limits_{i,j\ge 0} a_{i,j}(X-1)^i(Y-1)^j$ be the Taylor expansion of $P_w$ about $(1,1)$, i.e. $a_{i,j}=\frac{1}{i!j!}\frac{\partial^{i+j}P_w}{\partial^iX \partial^j Y}(1,1)$.

By Lemma \ref{basiccom}, $P_r=0$ or $P_r=(X-1)^l(Y-1)^s$ for some $l,s\ge 0$. In the first case, given $i,j\ge 0$, take $k=i+j+3$. Then $nP_w=P_{w_k}$ for certain $w_k\in \gamma_k(\F)$ and $n\ge 1$. Corollary \ref{derivadas} says that $\frac{\partial^{i+j}P_{w_k}}{\partial^iX \partial^j Y}(1,1)=0$, so $a_{i,j}=0$. Thus $P_w=0$, $w\in \FFF$ and then $g=\overline{w} \in G''$.

In the case that $P_r=(X-1)^l(Y-1)^s$ for some $l,s\ge 0$, given $i,j\ge 0$ with $i<l$ or $j<s$, take $k=i+j+3$, so Equation (\ref{eqhydra}) holds for certain $w_k\in \gamma_k(\F)$, $n\ge 1$ and $Q_k\in \ZZ$. Since $i<l$ or $j<s$, then $\frac{\partial^{i+j}(Q_kP_r)}{\partial^iX \partial^j Y}(1,1)=0$. By Corollary \ref{derivadas}, $\frac{\partial^{i+j}P_{w_k}}{\partial^iX \partial^j Y}(1,1)=0$. Then $a_{i,j}=0$ for each $(i,j)$ with $i<l$ or $j<s$. Thus $P_r$ divides $P_w$ in $\Z[X,Y]$. Then there exists $u\in N(r)$ such that $P_u=P_w$. By Theorem \ref{conway}, $wu^{-1}\in \FFF$ and then $g=\overline{wu^{-1}} \in G''$, as we wanted to prove.
\end{proof}


In \cite[Theorem 11]{Sim} Sims proves that the normal subgroup $N_m$ generated by the basic commutators of $\mathbb{F}_n$ of weight $m$ is $\gamma_m (\mathbb{F}_n)$ provided that $m\le 4$ or that $m=5$ and $n=2$. It is unknown whether this holds for every $n,m$. In \cite[Theorem 3.8]{JGS} it is proved that the answer is affirmative modulo $\mathbb{F}''_n$. Here we give a simple proof of this fact for the case $n=2$.

\begin{teo}(Jackson-Gaglione-Spellman) \label{sims}
Let $m\ge 1$ and let $N_m$ be the normal subgroup of $\F$ generated by the basic commutators of weight $m$. Then $N_m\F''=\gamma_m(\F)\F''$.
\end{teo}
\begin{proof}
We can assume $m\ge 2$. Since $N_m\subseteq \gamma_m(\F)$, in view of Theorem \ref{conway}, we only need to prove that for every $w\in \gamma_m(\F)$ there exists $w'\in N_m$ with $P_{w'}=P_w$, and, in view of Lemma \ref{basiccom}, this is equivalent to proving that there exist polynomials $Q_0,Q_1,\ldots, Q_{m-2}\in \ZZ$ such that $P_w=\sum\limits_{p=0}^{m-2} Q_p(X-1)^p(Y-1)^{m-2-p}$. By conjugating $w$ we may assume that $P_w\in \Z[X,Y]$. Let $P_w=\sum\limits_{i,j\ge 0} a_{i,j}(X-1)^i(Y-1)^j$ be the Taylor expansion of $P_w$ about $(1,1)$. By Corollary \ref{derivadas}, $a_{i,j}=0$ if $i+j\le m-3$. Then every term $a_{i,j}(X-1)^i(Y-1)^j$ is multiple of a polynomial $(X-1)^l(Y-1)^s$ with $l,s\ge 0$, $l+s=m-2$.
\end{proof}


\section{Commutators and parallelograms} \label{secparalelogramos}

Commutators in $\F$ are completely understood by Wicks' result, Theorem \ref{wicks}. However, we will associate a number $\iota$ to each commutator $[a,b]\in \F$, related to the winding invariant, and we will use it to obtain information about commutators that is harder to visualize using Wick's ideas.

\subsection{The invariant $\iota$} \label{ssc1p}

Given $w\in \F$, denote by $\vv_w \in \Z^2$ the endpoint of $\gamma_w$, that is the vector $(\exp(x,w),\exp(y,w))$ of exponents of $x$ and $y$ in $w$. Denote by $\sigma_w:[0,1]\to \R^2$ the straight path from $(0,0)$ to $\vv_w$. In general this path will not be contained in the grid $\Z\times \R \cup \R\times \Z$. Given $a,b\in \F$, the vectors $\vv_a$ and $\vv_b$ generate a second grid $(\vv_a \Z + \vv_b \R) \cup (\vv_a \R + \vv_b \Z)$ whose elements are integers multiples of $\vv_a$ plus real multiples of $\vv_b$ and symmetrically real multiples of $\vv_a$ plus integer multiples of $\vv_b$. This second grid is homeomorphic to the first one unless $\vv_a$ and $\vv_b$ are $\R$-linearly dependent. If $\dim _{\R} (\langle \vv_a,\vv_b\rangle)=1$, it is just one line, and if $\vv_a=\vv_b=0$, it has only one point. The double grid $D$ is the union of the first and second grids (see Figure \ref{paral}).

Note that

\begin{multline}
 \gamma_{[a,b]}=\gamma_{a}(\gamma_b+\vv_a)(\overline{\gamma}_a+\vv_b)\overline{\gamma}_b\simeq \\ 
\simeq \underbrace{\gamma_{a}\overline{\sigma}_a}_{\alpha}\sigma_a\underbrace{(\gamma_{b}+\vv_a)(\overline{\sigma}_b+\vv_a)}_{\overline{\beta}+\vv_a}(\sigma_b+\vv_a)(\overline{\sigma}_a+\vv_b)\underbrace{(\sigma_a+\vv_b)(\overline{\gamma}_a+\vv_b)}_{\overline{\alpha}+\vv_b}\overline{\sigma}_b\underbrace{\sigma_b\overline{\gamma}_b}_{\beta}, \ \ \ \label{para}
\end{multline}
where $\alpha=\gamma_{a}\overline{\sigma}_a$ and $\beta=\sigma_b\overline{\gamma}_b$ are loops based in $(0,0)$. The symbol $\simeq$ denotes here homotopy of paths in the double grid $D$.

\begin{figure}[h] 
\begin{center}
\includegraphics[scale=0.6]{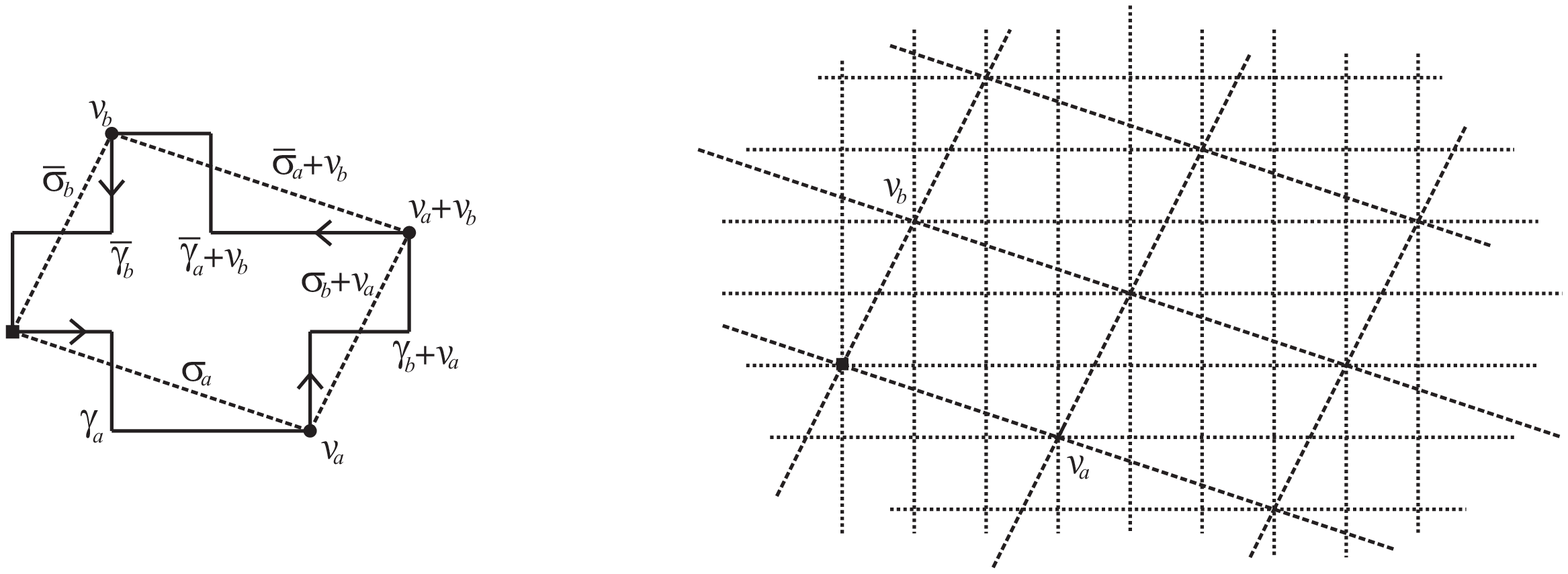}
\caption{The curve $\gamma_{[a,b]}$ and the parallelogram with vertices $(0,0)$, $\vv_a$, $\vv_a+\vv_b$, $\vv_b$ at the left. The double grid $D$ generated by the parallelogram at the right.}\label{paral}
\end{center}
\end{figure}

If $p\in \R^2$ is a point which is not in the double grid, then the winding number $\w(\alpha, p)$ coincides with $-\w(\overline{\alpha}+\vv_b, p+\vv_b)$ and $\w(\beta,p)=-\w(\overline{\beta}+\vv_a,p+\vv_a)$. Thus, if $L=L_{a,b}$ denotes the subgroup of $\Z^2$ generated by $\vv_a$ and $\vv_b$ we have $$ \iota_{p,a,b}=\sum\limits_{\lambda \in L} \w(\gamma_{[a,b]},p+\lambda)=\sum\limits_{\lambda \in L} \w(\sigma_a(\sigma_b+\vv_a)(\overline{\sigma}_a+\vv_b)\overline{\sigma}_b,p+\lambda).$$ \label{pagiota}

But among all the points $p+\lambda$ with $\lambda\in L$, there is only one inside the parallelogram described by the curve $\sigma_a(\sigma_b+\vv_a)(\overline{\sigma}_a+\vv_b)\overline{\sigma}_b$, in the case that $\vv_a$ and $\vv_b$ are $\R$-linearly independent, so $\iota_{p,a,b}=1$ or $-1$. If $\vv_a, \vv_b$ and the origin are collinear, $\iota_{p,a,b}=0$. Note that $\iota_{p,a,b}$ does not depend on the point $p$ but only on $\vv_a$ and $\vv_b$, so we may also write $\iota_{a,b}$. In conclusion we have the following result.

\begin{teo} \label{iota}
Let $a,b\in \F$. Then for any $p\in \R^2\smallsetminus D$ we have that $$\iota_{p,a,b}=\sum\limits_{\lambda \in L_{a,b}} \w(\gamma_{[a,b]},p+\lambda)$$ is $1$ if the vectors $\vv_a,\vv_b$ are positively oriented, $-1$ if they are negatively oriented and $0$ if $\{\vv_a,\vv_b\}$ is dependent.
\end{teo}

\begin{prop} \label{conmutador}
Let $w\in \FF$ be a word such that all the coefficients of the winding invariant $P_w$ are non-negative and at least one of them is greater than or equal to $2$. Then $w$ is not a commutator in $\F$. 
\end{prop} 
\begin{proof}
Assume $w=[a,b]$ for certain $a,b\in \F$. Let $(n,m)\in \Z^2$ be such that the coefficient $c \in \Z$ of $X^nY^m$ in $P_w$ is strictly greater that $1$. Let $p\in \R^2 \smallsetminus D$ be a point in the interior of the square of vertices $(n,m), (n+1,m), (n+1,m+1), (n,m+1)$. Then $\iota_{p,a,b}=\sum\limits_{\lambda \in L} \w(\gamma_w,p+\lambda)\ge \w(\gamma_w, p)=c\ge 2$, which contradicts Theorem \ref{iota}.
\end{proof}

\begin{ej}
$[x^n,y^m]^k$ is not a commutator in $\F$ if $nm\neq 0$ and $k\neq 0,1,-1$. We can assume $n,m\ge 1$ since $[x^n,y^m], [y^m,x^{-n}], [x^{-n}, y^{-m}]$ and $[y^{-m},x^n]$ differ in a cyclic permutation. Furthermore, we can assume $k$ is positive since $[a,b]=[b,a]^{-1}$. The winding invariant of $w=[x^n,y^m]^k$ is $P_w=k\sum\limits_{i=0}^{n-1}\sum\limits_{j=0}^{m-1}X^iY^j$, so the last proposition applies. 
\end{ej}

\begin{coro}
If $w\in \FF$ is such that no proper subword of $w$ lies in $\FF$, then $w^k$ is not a commutator for $k\neq 0,1,-1$.
\end{coro}
\begin{proof}
The hypothesis implies that $\gamma_w$ is a simple closed curve, so for all the squares outside the curve the corresponding coefficient in $P_{w^k}$ is $0$ and the coefficients corresponding to squares inside the curve are all $k$ or all $-k$. So Proposition \ref{conmutador} applies either to $w^k$ or $w^{-k}$.
\end{proof}   

If $P\in \ZZ$ and $(n,m)\in \Z^2$, we denote by $P\{(n,m)\}$ the coefficient of $X^nY^m$ in $P$. Let $a,b\in \F$. If $p$ is a point in the plane not in the double grid, and $p$ lies in the square of vertices $(n,m), (n+1,m), (n+1,m+1), (n, m+1)$, then $\w(\gamma_{[a,b]},p+\lambda)=P_{[a,b]}\{(n,m)+\lambda\}$ for every $\lambda \in L_{a,b}$. Therefore, Theorem \ref{iota} can be restated in the following way:

\begin{teo} \label{iotaprima}
Let $a,b\in \F$. Then for any $(n,m)\in \Z^2$ we have that $$\iota_{a,b}=\sum\limits_{\lambda \in L_{a,b}} P_{[a,b]}\{(n,m)+\lambda)\}$$ is $1$ if the vectors $\vv_a,\vv_b$ are positively oriented, $-1$ if they are negatively oriented and $0$ if $\{\vv_a,\vv_b\}$ is dependent.
\end{teo} 

\begin{coro} \label{obsiotaprima}
If $P\in \ZZ$ is the winding invariant of a commutator, there exists a number $\iota \in \{-1,0,1\}$ and a partition of $\Z^2$ (given by the cosets of a subgroup $L$) such that the sum of the coeficients $P\{(n,m)\}$ corresponding to pairs $(n,m)$ in each part is $\iota$.
\end{coro}

\begin{ej}
Let $w\in \FF$ be such that $P_w$ has only three non-zero coefficients: they are $2,2,-1$. Then $w$ is not a commutator. This is immediate from Corollary \ref{obsiotaprima}, since there is no partition of the multiset $\{2,2,-1\}$ in $2+2+(-1)=3$ parts such that the sum of the members of each part is $1$.
\end{ej}  

\begin{ej} \label{doscom}
The equation $[x^2,y]=[x,u]$ has no solution $u\in \F$. Otherwise, $P_{[x,u]}=P_{[x^2,y]}=1+X$. But then for every $(n,m)\in \Z^2$, $1=\iota_{x,u}=\sum\limits_{\lambda \in L_{x,u}} P_{[x^2,y]}\{(n,m)+\lambda \}$. Since $\vv_x=(1,0)\in L_{x,u}$, for $(n,m)=(0,0)$ we obtain $1=P_{[x^2,y]}\{(0,0)\}+P_{[x^2,y]}\{(0,1)\}=2$, a contradiction.
\end{ej}

We will generalize the previous example in Corollary \ref{doscom2}, using more advanced techniques.

We define an invariant which we have already used. The \textit{signed area} of a word $w\in \FF$ is defined as $P_w(1,1)\in \Z$. Note that the signed area of a commutator $[a,b]$ is by Theorem \ref{iotaprima} the number of cosets of $L_{a,b}$ in $\Z^2$ times $\iota$. Therefore we deduce:

\begin{coro} \label{signed}
Let $a,b\in \F$. Then $P_{[a,b]}(1,1)=\iota_{a,b}|\Z^2:L_{a,b}|$ (if the index of $L_{a,b}$ in $\Z^2$ is infinite, this equation says that $P_{[a,b]}(1,1)=\iota_{a,b}=0$). In particular, $\iota_{a,b}=1$ if the signed area of $[a,b]$ is positive, $\iota_{a,b}=-1$ if the signed area is negative and $\iota_{a,b}=0$ if the signed area is trivial.
\end{coro} 

The signed area of $w\in \F$ is the sum of the winding numbers $\w(\gamma_w, n+\frac{1}{2},m+\frac{1}{2})$ for $n,m\in \Z $. Since each square in the grid $\Z \times \R \cup \R \times \Z$ has area $1$, $P_w(1,1)=\int{ \w(\gamma_w, p) dp}$, where the domain of integration is the complement of the grid. If $w=[a,b]$ is a commutator, by Equation (\ref{para}), its signed area is $$\int_{\R^2 \smallsetminus D} \w(\alpha, p)+ \w(\overline{\alpha}+\vv_b, p)+\w(\overline{\beta}+\vv_b, p)+\w(\beta, p)+\w(\sigma_a(\sigma_b+\vv_a)(\overline{\sigma}_a+\vv_b)\overline{\sigma}_b, p) dp. $$
Now, $\overline{\alpha}+\vv_b$ is a translation of $\alpha$ with the opposite orientation, and $\overline{\beta}+\vv_b$ is a translation of $\beta$ with the opposite orientation. Therefore $P_{w}(1,1)=\int \w(\sigma_a(\sigma_b+\vv_a)(\overline{\sigma}_a+\vv_b)\overline{\sigma}_b, p) dp $ is, up to sign, the area of the (possibly degenerated) parallelogram with vertices $(0,0)$, $\vv_a$, $\vv_a+\vv_b$, $\vv_b$. When the parallelogram is non-degenerated, its area is the index $|\Z^2:L_{a,b}|$ and when it is degenerated, this index is infinite.



\subsection{Four examples, by Baumslag, Comerford, Edmunds, Mahler and  B., H. and P. Neumann} \label{ssc2p}

\begin{ej} \label{convexo}
In \cite{Bau}, Baumslag defines two classes of groups. The class $\mathcal{Y}$ contains those groups which admit a presentation $\langle x_1,x_2,\ldots , x_n | uv^{-1}\rangle$ where each word $u$ and $v$ is positive (the generators occur only with positive exponent) and $uv^{-1}\in [\mathbb{F}_n, \mathbb{F}_n]$. The class $\mathcal{X}$ consists of the groups which admit a presentation $\langle x_1,x_2,\ldots , x_n | [u,v]\rangle$ in which $u,v$ are positive. Obviously $\mathcal{X}$ is contained in $\mathcal{Y}$. The groups in $\mathcal{Y}$ are examples of residually finite one-realtor groups (see Subsection \ref{resfin}). In order to show that the classes are different, Baumslag considers the presentation $\langle x,y| (xy)^3(yx^2y^2x)^{-1}\rangle$ and uses Wicks' result, Theorem \ref{wicks} to prove that $w=(xy)^3(yx^2y^2x)^{-1}$ is not a commutator.

\begin{figure}[h] 
\begin{center}
\includegraphics[scale=0.6]{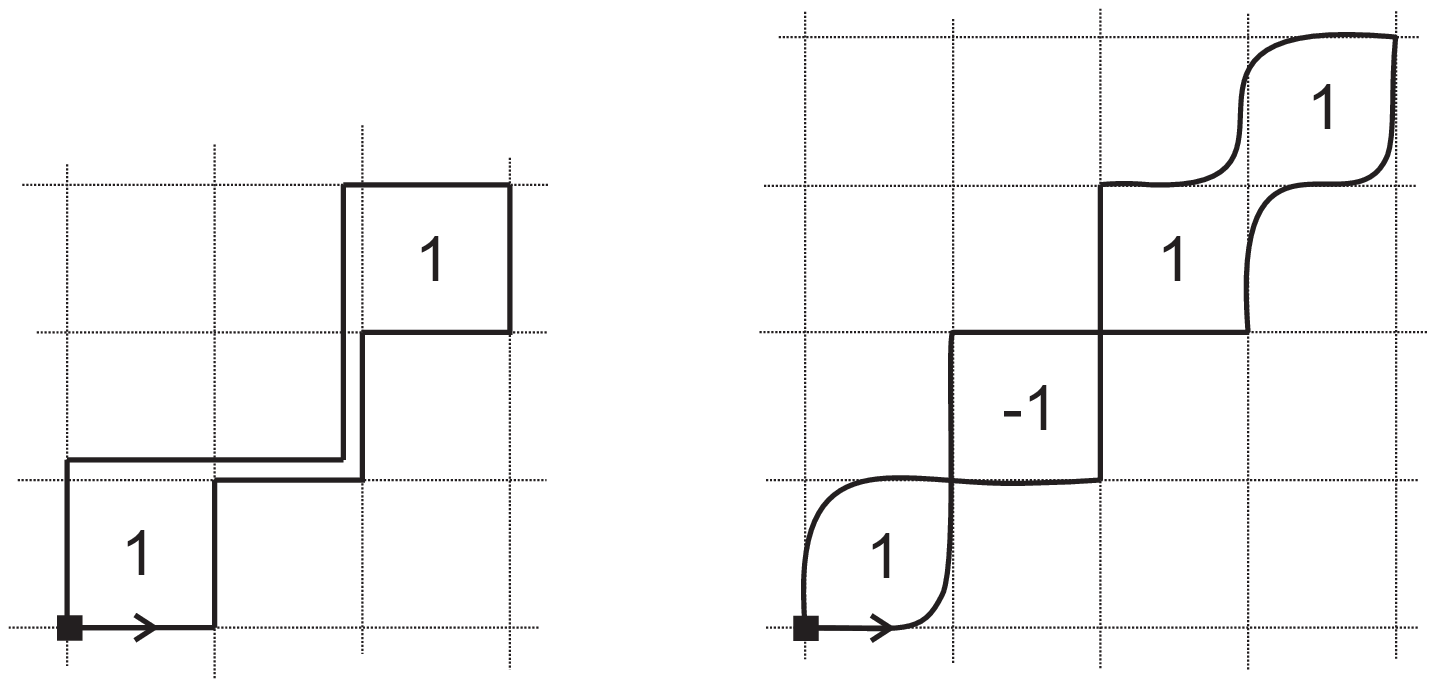}
\caption{The curves of $w=(xy)^3(yx^2y^2x)^{-1}$ and $w=xy^2x^2yxyx^{-1}y^{-1}x^{-1}y^{-2}x^{-2}y^{-1}$.}\label{positiva}
\end{center}
\end{figure}

In order to show that $w$ is not a commutator, instead of using Wicks' ideas, we propose to use a convexity argument together with Corollary \ref{obsiotaprima}. Suppose $w$ is a commutator. The winding invariant of $w$ is $P_w=1+X^2Y^2$ (see Figure \ref{positiva}). Since the sum of the coefficients of $P_w$ is positive, $\iota=1$. Let $L$ be a subgroup of $\Z^2$ as in the statement of Corollary \ref{obsiotaprima}. In the coset $(1,1)+L$ there must be a pair $(n,m)$ with $P_w\{(n,m)\}\neq 0$, that is $(n,m)=(0,0)$ or $(n,m)=(2,2)$. Therefore $(1,1)\in L$ and then $(0,0), (2,2)$ lie in the same coset, so the sum of the coefficients for that coset is $2$, a contradiction.

Many other examples of words $uv^{-1}\in \F'$ which are not commutators while $u,v$ are positive, can be constructed using this idea. We mention one more: $$w=xy^2x^2yxyx^{-1}y^{-1}x^{-1}y^{-2}x^{-2}y^{-1}\in \FF$$ is a product of a positive and a negative word. It is not a commutator since $P_w=1-XY+X^2Y^2+X^3Y^3$ (see Figure \ref{positiva}): as $P_w(1,1)=2>0$, we should have $\iota=1$. Since $P_w\{(1,1)\}=-1$, in the coset $(1,1)+L$ there must be exactly two pairs of the set $S=\{(0,0),(2,2),(3,3)\}$. Then $(1,1)\in L$ and all the elements of $S$ are in this coset.
\end{ej}


If $a,b\in \F$, then $[a^2,b]=a^2(ba^{-1}b^{-1})^2$ is a product of two squares. In \cite[Section 4]{CE}, Comerford and Edmunds prove that there exist commutators $[u,v]$ which are products of two squares and which are not of the form $[a^2,b]$. They observe \footnote{The convention for commutators used in \cite{CE} is not ours, so we have adapted their example.} that for $u=y^{-1}x^{-1}y^2xy^{-1}$ and $v=x$, $$[u,v]=(xy^{-1}x^{-1}yx^{-1}y^{-1}xy^{-1}xyx^{-1})^2(xy^{-1}x^{-1}yx^{-1}y^2xy^{-1}xyx^{-1}y^{-2}xyx^{-1})^2$$ is a product of two squares and they prove that there are no $a,b\in \FF$ such that $[u,v]=[a^2,b]$. Their proof studies the possible factorizations of $[u,v]$ as Wicks words of commutators. Here we give an alternative proof using the winding invariant.

\begin{prop}[Comerford-Edmunds] \label{propce}
Let $u=y^{-1}x^{-1}y^2xy^{-1}, v=x \in \F$. Then there are no $a,b\in \F$ such that $[u,v]=[a^2,b]$.
\end{prop}
\begin{proof}
The winding invariant of $w=[u,v]$ is $P_w=(1-X^{-1})(1+Y^{-1})$. Suppose $a,b\in \F$ are such that $w=[a^2,b]=[a,b]bab^{-1}[a,b]ba^{-1}b^{-1}$. Then $P_w=P_{[a,b]}(1+X^nY^m)$, where $(n,m)=\vv_a\in \Z^2$ is the endpoint of $\gamma_a$. Thus, $1+X^nY^m$ divides $(1-X^{-1})(1+Y^{-1})$. Putting $Y=1$ gives that $1+X^n$ divides $2(1-X^{-1})$ in $\Z[X,X^{-1}]$. If $n\ge 0$, $1+X^n | 2(X-1)$ in $\Z[X]$ and if $n<0$, $X^{-n}+1|2(X-1)$ in $\Z[X]$. We conclude that $n=0$, so $\vv_a=(0,m)$ ($m$ can be proved to be $1$ or $-1$, but we will not need that). 

Since the signed area $P_w(1,1)$ is $0$, $\iota_{a^2,b}=0$ by Corollary \ref{signed} and then $\{2\vv_a,\vv_b\} \subseteq \R^2$ is linearly dependent, so $\vv_b=(0,r)$ for some $r\in \Z$. Thus, $L_{a^2,b}=\langle 2\vv_a, \vv_b \rangle \leqslant 0\times \Z \leqslant \Z^2$.

We apply Theorem \ref{iotaprima} to the pair $(0,0)$ to obtain $0=\iota_{a^2,b}=\sum\limits_{\lambda \in L_{a^2,b}} P_{w}\{\lambda\}$. But $P_{w}\{(0,0)\}=P_w\{(0,-1)\}=1$, and $P_w\{(0,s)\}=0$ for every $s\in \Z\smallsetminus \{0,1\}$. Therefore, $\iota_{a^2,b}=1$ or $\iota_{a^2,b}=2$, a contradiction. 
\end{proof}

Our proof of Proposition \ref{propce} shows in fact that the equation $[u,v]=[a^2,b]$ has no solution in $\M=\F/\FFF$. The next two results concern equations in $\M$ which assume independence modulo $\F'$. The first one was proved in \cite[Theorem 5.1]{BNNN}. We provide here a short proof as an application of the invariant $\iota$.

\begin{prop} (Baumslag, Neumann, Neumann, Neumann) \label{propbnnn}
Let $w,u\in \F$ be two elements which are independent modulo $\F'$. Then the equation $[w,u]=z^k$ has no solution in $\M=\F/ \F''$ for any $k\ge 2$.
\end{prop}
\begin{proof}
Suppose $z\in \F$ is a solution. Then $z\in \F'$ and $P_{[w,u]}=kP_z$. The invariant $\iota_{w,u}$ is $1,-1$ or $0$, and since $k$ divides $\iota$, then $\iota=0$, which means that $w$ and $u$ are not independent in $\F/\F'$.
\end{proof}

In the next example we provide a simple proof of a particular case of a result of Baumslag and Mahler \cite{BM6}, which is the analogue of the original Vaught conjecture (see Conjecture \ref{vaughtconjecture}) for the free metabelian group $\M$. 

\begin{prop}(Baumslag-Mahler) \label{propbm}
Let $w,u\in \F$ be two elements which are independent modulo $\F'$. Then the equation $w^2u^2=z^2$ has no solution in $\M$.
\end{prop}
\begin{proof}
Suppose that $z\in \F$ is a solution of the equation. Then in $\F/\F'$, $(wu)^2=w^2u^2=z^2$. Thus $wu$ is equal to $z$ modulo $\F'$. Let $c\in \F'$ be such that $wu=zc$. We have then $z^2=w^2u^2=(wu)u^{-1}(wu)u=zcu^{-1}zcu$ in $\M$. This implies that $cu^{-1}zcuz^{-1}\in \F''$. In particular, the winding invariant is $$0=(1+X^nY^m)P_c+P_{[u^{-1},z]},$$ where $n,m$ are the exponents of $x,y$ in $u^{-1}z$. The invariant $\iota_{u^{-1},z}$ is $1,-1$ or $0$, but by the equation above this number must be even. Thus, $\iota_{u^{-1},z}=0$ and then $u^{-1}$ and $z$ are not independent modulo $\F'$. Thus, the same is true for $w$ and $u$, a contradiction.
\end{proof}

\subsection{Endomorphisms, the area formula and commutators determining the same parallelogram}

Now we will study the effect of an endomorphism $\phi:\F \to \F$ on the winding invariant, that is the relationship between $P_{\phi(w)}$ and $P_w$ for a word $w\in \FF$.

Since $w\in \FF$, Equation (\ref{area}) in page \pageref{area} holds for certain $k\ge 0$, $u_i\in \F$ and $\epsilon_i=\pm 1$. Then 
\begin{equation}
\phi(w)=\prod\limits_{i=1}^k \phi(u_i)[\phi(x),\phi(y)]^{\epsilon_i}\phi(u_i)^{-1},
\end{equation}
so by Proposition \ref{basica}, 

\begin{equation} \label{eqendo}
P_{\phi(w)}=P_{[\phi(x),\phi(y)]}\sum\limits_{i=1}^k \epsilon_i X^{n_i'}Y^{m_i'},
\end{equation}
where $n_i',m_i'$ are the total exponents of $x$ and $y$ in $\phi(u_i)$. In other words, $(n_i',m_i')=\overline{\phi}(n_i,m_i)$ where $\overline{\phi}:\Z^2\to \Z^2$ is the map induced by $\phi$ in the abelianizations and $n_i,m_i$ are the exponents of $x,y$ in $u_i$. If $\Z[\overline{\phi}]:\ZZ\to \ZZ$ denotes the ring homomorphism induced by $\overline{\phi}$, then Equation (\ref{eqendo}) says:

\begin{prop} \label{propzdephi}
If $\phi \in \textrm{End}(\F)$ and $w\in \F'$, then $P_{\phi(w)}=P_{\phi([x,y])}\Z[\overline{\phi}](P_w)$.
\end{prop}

If for $\phi \in \textrm{End}(\F)$ we define $W(\phi):\ZZ \to \ZZ$ by $W(\phi)(P)=P_{\phi([x,y])}\Z[\overline{\phi}](P)$, then the previous result simply says that $W(\phi)W(w)=W(\phi(w))$ for every $w\in \F'$.

Since $\Z[\overline{\phi}]$ preserves the augmentation $\ZZ\to \Z$ (the signed area), then we obtain the following

\begin{coro}(The area formula) \label{endo}
If $w\in \FF$ and $\phi \in \textrm{End}(\F)$, then the signed area of $\phi(w)$ is the product of the signed area of $w$ and the signed area of $\phi([x,y])$.
\end{coro}

\begin{coro}
Let $u\in \FF$ and let $n,m\in \Z$. If there exist $a,b\in \F$ such that $u=[a^n,b^m]$, then $nm$ divides the signed area of $u$.
\end{coro}
\begin{proof}
If $u=[a^n,b^m]$, then $u=\phi([x^n,y^m])$ for the endomorphism $\phi:\F \to \F$ which maps $x$ to $a$ and $y$ to $b$. By Corollary \ref{endo}, $nm=P_{[x^n,y^m]}(1,1)$ divides $P_u(1,1)$. 
\end{proof}
 
The following result is mentioned in \cite{Barac}.

\begin{coro} \label{corounit}
If $\phi \in \textrm{Aut}(\F)$, then $P_{\phi([x,y])}$ is a unit of $\ZZ$.
\end{coro}
\begin{proof}
We have $1=P_{[x,y]}=P_{\phi (\phi^{-1}([x,y]))}=P_{\phi([x,y])}\Z[\overline{\phi}](P_{\phi^{-1}([x,y])})$.
\end{proof}

The signed area of $[a,b]$ (and $\iota_{a,b}$) depends only on the vectors $\vv_a, \vv_b\in \Z^2$. What else can be said about commutators $[a_1,b_1]$, $[a_2,b_2]$ with equal associated vectors $\vv_{a_1}=\vv_{a_2}$, $\vv_{b_1}=\vv_{b_2}$? The following result gives a partial answer. 

\begin{prop} \label{paralelogramoposta}
Let $a_1,a_2,b_1,b_2\in \F$ and suppose $a_1a_2^{-1},b_1b_2^{-1}\in \FF$. Let $n, m\in \Z$ be the total exponents of $x$ and $y$ in $a_1$ and $n',m'$ the total exponents of $x,y$ in $b_1$. Then $P_{[a_1,b_1]}-P_{[a_2,b_2]}$ is in the ideal $\langle X^nY^m-1, X^{n'}Y^{m'}-1\rangle \subseteq \ZZ$. 
\end{prop}
\begin{proof}
We copy the idea of the beginning of the section
\begin{multline}
 \gamma_{[a_1,b_1]}=\gamma_{a_1}(\gamma_{b_1}+(n,m))(\overline{\gamma}_{a_1}+(n',m'))\overline{\gamma}_{b_1}\simeq \\ 
\simeq \underbrace{\gamma_{a_1}\overline{\gamma}_{a_2}}_{\alpha}\gamma_{a_2}\underbrace{(\gamma_{b_1}+(n,m))(\overline{\gamma}_{b_2}+(n,m))}_{\overline{\beta}+(n,m)}(\gamma_{b_2}+(n,m))(\overline{\gamma}_{a_2}+(n',m')) \\
\underbrace{(\gamma_{a_2}+(n',m'))(\overline{\gamma}_{a_1}+(n',m'))}_{\overline{\alpha}+(n',m')}\overline{\gamma}_{b_2}\underbrace{\gamma_{b_2}\overline{\gamma}_{b_1}}_{\beta} 
\end{multline}
Now $\alpha=\gamma_{a_1}\overline{\gamma}_{a_2}$ and $\beta=\gamma_{b_2}\overline{\gamma}_{b_1}$ are loops based at $(0,0)$ and $\simeq$ denotes homotopy of paths in the grid. For each $p\in \R^2$ out of the grid, we have $$\w(\gamma_{[a_1,b_1]},p)=\w(\gamma_{[a_2,b_2]},p)+\w(\alpha,p)+\w(\overline{\beta}+(n,m),p)+\w(\overline{\alpha}+(n',m'),p) +\w(\beta,p).$$

Thus, $P_{[a_1,b_1]}=P_{[a_2,b_2]}+P_{\alpha}-X^{n}Y^{m}P_{\beta}-X^{n'}Y^{m'}P_{\alpha}+P_{\beta}$, where $P_\alpha$ and $P_\beta$ denote the polynomials associated to the curves $\alpha$ and $\beta$ (i.e. $P_{a_1a_2^{-1}}$ and $P_{b_2b_1^{-1}}$). Then $P_{[a_1,b_1]}-P_{[a_2,b_2]}\in \langle 1-X^{n}Y^{m},1-X^{n'}Y^{m'}\rangle$.
\end{proof}

If in the previous proposition we have $b_1=b_2$, then $\beta$ is null-homotopic and $P_{\beta}=0$. Thus we deduce the following

\begin{coro} \label{postafuerte}
Let $a_1,a_2,b\in \F$ and suppose $a_1a_2^{-1}\in \FF$. Let $n', m'\in \Z$ be the total exponents of $x$ and $y$ in $b$. Then $P_{[a_1,b]}-P_{[a_2,b]}$ is in the ideal $\langle X^{n'}Y^{m'}-1\rangle \subseteq \ZZ$. 
\end{coro}

\section{Word maps and verbal subgroups for commutators, Engel words, squares and higher powers} \label{secverbal}

The \textit{genus} or \textit{commutator length} $\g _G(g)$ of an element $g$ in the derived subgroup $G'$ of a group $G$ is the smallest integer $n$ such that there exist $a_1,b_1,a_2,b_2,\ldots, a_n,b_n \in G$ which satisfy
\begin{equation}
w=\prod\limits_{i=1}^n  [a_i,b_i]=[a_1,b_1][a_2,b_2]\ldots [a_n,b_n].
\label{genero}
\end{equation}
The supremum of the commutator lengths is the \textit{commutator length} or \textit{commutator width} of $G'$. This invariant has been extensively studied for different classes of groups, including finite, simple, free and metabelian \cite{BaMe, CCE, CE, Cul, DH, Gri, Gur, Gur2, Mac, Mur, Ros}. For free groups there exists an algorithm which computes the commutator length \cite{Cul, Edm}. Culler proves in \cite[2.6]{Cul} that $\g _{\F}([x,y]^n)$ is $\frac{n+1}{2}$ if $n$ is odd and $\frac{n}{2}+1$ if $n$ is even. So in particular the commutator length of $\FF$ is infinite. 
There are many groups for which the commutator width is known to be finite (see \cite{AM1, AM2, BGL, Rhe2, Str2} for some examples). On the other hand the commutator width of non virtually cyclic hyperbolic groups is infinite (see Theorem \ref{teomn} below). The (already proved) Ore conjecture states that the commutator width of a non-abelian finite simple group is $1$.

The \textit{square length} $\Sq _G(g)$ of an element $g\in \Sq(G)=\langle h^2 | h\in G\rangle$ is the minimum $n\ge 0$ such that $g$ can be expressed as a product of $n$ squares, and the \textit{square width} or \textit{square length} of $\Sq (G)$ is the supremum of these numbers. \label{pagsq} This invariant has also been widely studied both for finite and infinite groups. Each commutator $[u,v]$ is a product of squares. Namely $[u,v]=u^2(u^{-1}v)^2v^{-2}$. Therefore, for finitely generated groups, finite commutator width implies finite square width. If $G=\F$, then the subgroup of $G$ generated by the squares is the set of words $w\in \F$ whose total exponent in $x$ and total exponent in $y$ are even. Proposition \ref{ln} says for instance that $\Sq _{\F}([x,y])\ge 3$.

The notions of commutator and square length generalize as follows. Given a group $G$, a word $u\in \mathbb{F}(x_1,x_2,\ldots, x_n)$ determines a function $G^n=G\times G\times \ldots \times G\to G$ which maps $(g_1,g_2,\ldots, g_n)$ to $u(g_1,g_2, \ldots, g_n)$. This is called the \textit{$u$ word map} and the elements in its image, along with their inverses, are the \textit{$u$-elements}. The \textit{verbal subgroup} $u(G)$ \label{pagverbal} is the subgroup of $G$ generated by the $u$-elements. The \textit{$u$-length} of an element $g\in u(G)$ is the least $n$ for which $g$ can be expressed as a product of $n$ $u$-elements. The $u$-width of $u(G)$ is defined as the supremum of the $u$-lengths of its elements. In the case that $u=[x,y]$, the verbal subgroup is $u(G)=G'$ and the $u$-length of $g\in G'$ is its commutator length. For $u=x^2$, $u(G)=\Sq (G)$ and the $u$-length is the square length. Culler's results imply that for $G=\F$, the commutator width of $\F'$ is infinite. This follows from a far more general result by Rhemtulla \cite{Rhe}.

\begin{teo}[Rhemtulla] \label{rhemtulla}
If $G$ is a free product of two non-trivial groups which is not infinite dihedral and $u\in F(x_1,x_2,\ldots, x_n)$ is non-trivial and $u(G)\neq G$, then the $u$-width of $u(G)$ is infinite.
\end{teo}

This was extended by Myasnikov and Nikolaev in \cite{MN}.

\begin{teo}[Myasnikov-Nikolaev] \label{teomn}
In a non-elementary hyperbolic group $G$ every proper verbal subgroup $u(G)$ has infinite width.
\end{teo}

Note that if $G$ is an epimorphic image of a group $H$ then the $u$-width of $u(H)$ is greater than or equal to the $u$-width of $u(G)$. In particular this holds for $G=\M$ and $H=\F$. The aim of this section is to study verbal subgroups and $u$-lengths in $\M$ and the consequences of this for $\F$. Note that since $\M$ contains a subgroup isomorphic to $\Z \times \Z$, it is not hyperbolic and the last result cannot be applied. We will recall a result of Bavard and Meigniez about the commutator length of $\M'$, we will analyze the difference between $u$-length and verbal subgroups for $u=c_n=[x_1,x_2, \ldots, x_{n+1}]$ and $u=e_n=[y,y,\ldots, y,x]$. We will find bounds for the square length of $\Sq(\M)$. Then we will characterize the verbal subgroup of $\M$ for $u=x^3$, that is the elements which are products of cubes. We will obtain a necessary condition for an element of $\M$ to be a product of $n$-th powers. This gives in turn a criterion for $\F$. For the case $n=4$ we will give a particular argument based on a coloring of the squares in the grid.
	
\subsection{Commutators, Engel words and higher commutators}

Given $w\in \F'$ we write $\lax (w)$ for the commutator length of the class of $w$ in $\M$. Theorem \ref{conway} implies that $\lax(w)$ is the minimum $n$ such that there are $a_1,b_1,\ldots, a_n,b_n\in \F$ with $P_w=\sum\limits_{i=1}^n P_{[a_i,b_i]}$.    

The results in Subsections \ref{ssc1p} and \ref{ssc2p} can be interpreted as statements on the commutator length of elements in $\M$. For instance, Proposition \ref{conmutador} says that if $w\in \FF$ is a word such that all the coefficients of the winding invariant $P_w$ are non-negative and at least one of them is greater than or equal to $2$, then $\lax(w)\ge 2$. 

The difference between $\g_{\F}(w)$ and $\lax(w)$ can be arbitrarily big as the next example shows.

\begin{ej} \label{vueltitas}
Let $n\ge 2$. Then $\lax([x,y]^n)=2$. Indeed, Proposition \ref{conmutador} implies that $\lax([x,y]^n)\ge 2$ and for $a_1=x$, $b_1=\prod\limits_{i=1}^n ([y,x^{-1}]^{n-i} y)$, $a_2=b_1y^{-n}$, $b_2=yx \in \F$ we have $$P_{[x,y]^n}=n=P_{[a_1,b_1]}+P_{[a_2,b_2]}.$$
This can be seen graphically (see Figure \ref{genus2}) or proved algebraically. We don't include here the complete proof since we will show that something much more general holds.


\begin{figure}[h] 
\begin{center}
\includegraphics[scale=1.6]{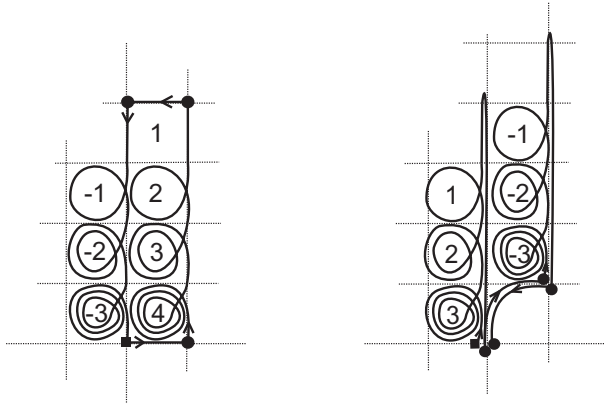}
\caption{The curves $\gamma_{[a_1,b_1]}$ and $\gamma_{[a_2,b_2]}$ for $n=4$.}\label{genus2}
\end{center}
\end{figure}

%
%
%
%
\end{ej}

Example \ref{vueltitas} leads to a natural question: what is the commutator length of $\M'$? In other words, how big can $\lax (w)$ be? The fact that this number is finite follows from another result by Rhemtulla \cite[Theorem 1]{Rhe2}. The answer to the question above was given by Bavard and Meigniez in \cite{BaMe} (see also Allambergenov and Roman'kov's paper \cite{AR}). We write here their proof using the concepts we have studied so far.

\begin{prop}[Bavard-Meigniez]   \label{bavard}
For every $w\in \F'$, $\lax (w)\le 2$.
\end{prop}
\begin{proof}
By Theorem \ref{conway} we only have to prove that for any polynomial $P\in \ZZ$ there exist commutators $c_1,c_2\in \F$ such that $P=P_{c_1}+P_{c_2}$.

Let $P\in \ZZ$. Call $k=P(1,1)\in \Z$. Let $n,m\in \Z$ be such that $X^nY^m(P-P_{[x,y^k]})\in \Z[X,Y]$. We apply the Euclidean algorithm to obtain polynomials $Q_1\in \Z[X,Y]$ and $Q_2\in \Z[Y]$ such that $$X^nY^m(P-P_{[x,y^k]})-(1-X)Q_1-(1-Y)Q_2\in \Z.$$

Since $P_{[x,y^k]}(1,1)=k$, we deduce that the integer above is $0$. Thus $$P=(1-X)Q_1X^{-n}Y^{-m}+P_{[x,y^k]}+(1-Y)Q_2X^{-n}Y^{-m}.$$

By the surjectivity of the winding map, there exist $w_1,w_2\in \F'$ such that $P_{w_i}=Q_iX^{-n}Y^{-m}$ for $i=1,2$. Thus, $P_{[w_1,x]}=(1-X)Q_1X^{-n}Y^{-m}$ and $P_{[w_2,y]}=(1-Y)Q_2X^{-n}Y^{-m}$. But $[w_1,x][x,y^{k}]=[w_1y^{-k},y^kxy^{-k}]$, so $$P=P_{[w_1y^{-k},y^kxy^{-k}]}+P_{[w_2,y]}$$ is the sum of the winding invariants of two commutators.
\end{proof}

Every element $w\in \gamma_3(\F)=[\F, [\F,\F]]$ is a product $w=\prod\limits_{i=1}^n [a_i,[b_i,c_i]]$ for some $n\ge 0$ and $a_i,b_i,c_i\in \F$ (see \cite[pp. 297]{MKS}). Let $c_2=[x,y,z]=[x,[y,z]]\in \mathbb{F}_3$. The verbal subgroup $c_2(\F)$ is $\gamma_3(\F)$ and the $c_2$-length of an element $w\in \gamma_3(\F)$ is the minimum possible $n$ in the expression above.

\begin{ej}
Let $w=[x,[x,y]x^2[x,y]x^{-2}]\in \F$. Since for every $a,b,c\in \F$ the identity $[a,bc]=[a,b][bab^{-1},bcb^{-1}]$ holds, we have that $$w=[x,[x,y]][[x,y]x[x,y]^{-1}, ([x,y]x^2)[x,y]([x,y]x^2)^{-1}].$$ Therefore the $c_2$-length of $w$ is smaller than or equal to $2$. We use the winding invariant to prove that it is exactly $2$.
$P_w=(X-1)(1+X^2)$. Suppose that $w=[a,[b,c]]\in \F$ for some $a,b,c\in \F$. Then $P_w=(X^nY^m-1)P_{[b,c]}$, where $n,m$ are the total exponents of $x,y$ in $a$. Then $m=0$ and $X^n-1|(X-1)(1+X^2)$. We deduce that $n=\pm 1$ so $P_{[b,c]}$ is equal to either $1+X^2$ or $-X(1+X^2)$. In Example \ref{convexo} we proved that $1+X^2Y^2$ cannot be the winding invariant of a commutator. The same convexity argument can be used to prove that $1+X^2$ and $-X(1+X^2)$ are not winding invariants of commutators. Thus, the $c_2$-length of $w$ is $2$.
\end{ej}    	

By Theorem \ref{rhemtulla} there are elements in $\gamma_3(\F)$ of arbitrarily large $c_2$-length. The $c_n$-length of elements in $\gamma_n(\F)$ can be defined similarly and can also be studied using the winding invariant.

\begin{ej} \label{eje2}
If $P_w=X^3-1=(X-1)(1+X+X^2)$, we will not be able to prove that the $c_2$-length of $w$ is greater than $1$ without extra assumptions. In fact $w=[x,[x^3,y]]$ has such winding invariant. However, we can prove that any word with $P_w=X^3-1$ is not in the image of the second Engel word map $e_2=[y,[y,x]]$, that is $w=[a,[a,b]]$ has no solution $a,b\in \F$. Otherwise, $P_w=(X^nY^m-1)P_{[a,b]}$ for $n,m$ the exponents of $x,y$ in $a$. From this we conclude that $m=0$ and $X^n-1|X^3-1$. Thus $n\in \{1,-1,3,-3\}$. If $n=1$, then $P_{[a,b]}=1+X+X^2$. We use the ideas of Section \ref{secparalelogramos}. Denote by $P_{[a,b]}\{(k,l)\}$ the coefficients of $P_{[a,b]}$ and by $L$ the subgroup of $\Z^2$ generated by $\vv_a, \vv_b$. Then $\iota=\iota_{a,b}=\sum\limits_{\lambda \in L} P_{[a,b]}\{(k,l)+\lambda \}$ for each $(k,l)\in \Z^2$ (see Theorem \ref{iotaprima}). In this case $\vv_a=(n,m)=(1,0)$, so in particular, for $k=l=0$, we obtain $\iota=3\neq -1,0,1$, which is absurd. If $n=3$, then $P_{[a,b]}=1$. In this case $\vv_a=(3,0)$, $\iota=1$ and we again arrive to a contradiction: taking $(k,l)=(-1,0)$ we deduce that $(1,0)\in \langle (3,0), \vv_b\rangle$ and taking $(k,l)=(0,-1)$ we obtain that $(0,1)\in \langle (3,0), \vv_b \rangle$. For $n=-1,-3$ we obtain contradictions with the same arguments. 

As an example, $w=[x,[x,y]][x,[x,xyx^{-1}]][x,[x,x^2yx^{-2}]]$ is in the verbal subgroup $e_2(\F)$ and $P_w=X^3-1$. Thus, the $e_2$-width of $w$ is greater than or equal to $2$.
\end{ej}

The element $w=[x,[x^3,y]]\in \M$ is clearly in the image of the word map $\M\times \M \times \M \to \M$ associated to $c_2=[x,y,z]=[x,[y,z]]$, and Example \ref{eje2} shows that it is not in the image of the word map associated to $e_2=[y,[y,x]]$. However, we will prove that the verbal subgroup $c_2(\M)=\gamma_3(\M)$ coincides with the verbal subgroup $e_2(\M)$. In general, for $n\neq 1,2$ we have $\gamma_{n+1}(\M)\neq e_n(\M)$.

\begin{teo} \label{nge3}
Let $e_n\in \F$ denote the $n$-th Engel word, and $c_n=[x_1,x_2, \ldots, x_{n+1}] \in \mathbb{F}_{n+1}$. Then the verbal subgroup $e_n(\M)$ coincides with $c_n(\M)=\gamma_{n+1}(\M)$ for $n=1,2$, but $e_n(\M)\neq \gamma_{n+1}(\M)$ for any $n\ge 3$. 
\end{teo}
\begin{proof}
For $n=1$ this is trivial since $e_1=[y,x]$ and $c_1=[x,y]$. For $n=2$ we have seen that the word maps associated to $e_2$ and $c_2$ do not have the same image. However, if $w\in \M$ is in the image of the word map associated to $c_2$, then by Proposition \ref{lowercs}, $W(w)=(X^kY^l-1)P$ for certain $k,l\in \Z$ and $P\in \ZZ$. Now, $X-1=W([x,[x,y]])$ and $1-Y=W([y,[y,x]])$ lie in $W(e_2(\M))$. Since $e_2(\M)$ is a normal subgroup of $\M$, $W(e_2(\M))$ is an ideal of $\ZZ$. Now $X^kY^l-1=(X^k-1)Y^l+Y^l-1\in \langle X-1,1-Y \rangle$, so $W(w)\in W(e_2(\M))$ and then $w\in e_2(\M)$. Since each generator of $c_2(\M)$ lies in $e_2(\M)$, then $c_2(\M)\subseteq e_2(\M)$.

Let $n\ge 3$. Then $w=[\underbrace{x,x,\ldots, x}_{n-2},y,x,y]\in \gamma_{n+1}(\M)=c_n(\M)$ and $W(w)=(X-1)^{n-2}(Y-1)$. Thus, $\frac{\partial^{n-1} W(w)}{\partial^{n-2}X \partial Y}=(n-2)!$. On the other hand, if $w'\in \M$ is an $e_n$-element (is in the image of the word map associated to $e_n$), say $w'=[\underbrace{u,u,\ldots, u}_{n},v]$ for certain $u,v\in \M$, then $W(w')=(X^kY^l-1)^{n-1}P_{[u,v]}$. It is easy to see that $$\frac{\partial^{n-1} W(w')}{\partial^{n-2}X \partial Y}=(n-1)!A+(X^kY^l-1)B$$ for certain $A, B \in \ZZ$. In particular $(n-1)!$ divides $\frac{\partial^{n-1} W(w')}{\partial^{n-2}X \partial Y}(1,1)$, and this holds as well for any $w'\in e_n(\M)$. Thus, $w \notin e_n(\M)$.   
\end{proof}

To finish this subsection we mention a conjecture of Shalev \cite[Conjecture 3.5]{Sha2} connected to the previous result and Example \ref{eje2}. 

\begin{conj}[Shalev]
For every finite simple non-abelian group $G$ and every $n\ge 1$, the $n$-th Engel word map $G\times G\to G$ is surjective.
\end{conj}

Recall that the statement of the Ore conjecture \cite{Ore} is that for every finite simple non-abelian group $G$ the $u=[x,y]$ word map $G\times G\to G$ is surjective (the $u$-width of $u(G)=G$ is $1$). In other words, every element of $G$ is a commutator. This conjecture stated in 1951 is the case $n=1$ of Shalev's conjecture. The proof of the Ore conjecture was completed in 2010 by Liebeck, O’Brien, Shalev and Tiep \cite{LOST}. Note that the Ore conjecture implies that for any simple group $G$, the word map $G^{n+1}\to G$ associated to $c_n=[x_1,x_2,\ldots, x_{n+1}]$ is surjective. Shalev's conjecture has been verified in several cases \cite{BGG, Gor, Elk, Kan}.


\subsection{Square length in $\M$}

An element $w\in \M$ is a product of squares if and only if the exponents $\exp(x,w)$, $\exp(y,w)$ of $x$ and $y$ in $w$ are both even. In particular the elements in $\M'$ are products of squares. We have seen in Proposition \ref{lnconverse} that not every $w\in \M'$ is a product of two squares. On the other hand, since every $w\in \M'$ is a product of two commutators (Proposition \ref{bavard}) and each commutator is a product of three squares, then every element in $\M'$ is a product of six squares. The following result improves this bound and establishes bounds for the square length of $\Sq (\M)$.

\begin{teo} \label{mainsquarel}
Every element of $\M$ which is a product of squares, is a product of five squares. There exists an element in $\M'$ which is not a product of three squares.
\end{teo}
\begin{proof}
\textit{Step 1. An element which is not a product of three squares.} Let $z\in \F'$ and let $u,v,w\in \F$ be such that $z=u^2v^2w^2$ in $\M$. Note that $u^2v^2w^2=(uvw)w^{-1}v^{-1}(uvw)w^{-1}vw^2$. Then $a=uvw\in \F'$ and $P_z=P_a(1+m_{w^{-1}v^{-1}})+P_{w^{-1}v^{-1}w^{-1}vw^2}$, where for $t\in \F$, $m_{t}$ denotes the monomial $X^kY^l$ for $k,l$ the exponents of $x,y$ in $t$. Thus, 

\begin{equation}
P_z=P_a(1+m_{w^{-1}v^{-1}})+m_{w^{-1}}P_{[v^{-1},w^{-1}]}.
\label{squarel}
\end{equation}

Consider the polynomial $P=X^{-2}Y^{-2}+X^2Y^2+X^{-1}+X+X^{-1}Y+XY^{-1}$ and let $z\in \F'$ be an element with winding invariant $P_z=P$ (see Figure \ref{seisunosfig}). We will prove that $z\in \M$ is not a product of three squares in $\M$. Otherwise there exist $a\in \F'$, $v,w\in \F$ such that the equation above holds. We claim that $\iota=\iota _{v^{-1},w^{-1}}= \sum\limits_{\lambda \in L} P_{[v^{-1},w^{-1}]}\{(s,t)+\lambda\}=0$ (for any $(s,t)\in \Z^2$). As always $L=\langle \vv_{v^{-1}}, \vv_{w^{-1}} \rangle$. For any $(s,t)\in \Z^2$ we have 
$$\sum\limits_{\lambda \in L} P_z\{(s,t)+\lambda \}=$$ $$=\sum\limits_{\lambda \in L} P_a\{(s,t)+\lambda \}+\sum\limits_{\lambda \in L} (m_{w^{-1}v^{-1}}P_a)\{(s,t)+\lambda\}+\sum\limits_{\lambda \in L} (m_{w^{-1}}P_{[v^{-1},w^{-1}]})\{(s,t)+\lambda\}.$$ 

\begin{figure}[h] 
\begin{center}
\includegraphics[scale=0.5]{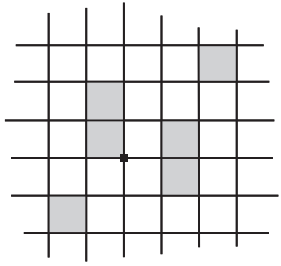}
\caption{The support of $P$. The shadow squares represent monomials with coefficient $1$.}\label{seisunosfig}
\end{center}
\end{figure}

The first and second summands on the right hand side are equal, and the third summand is $\iota$. Therefore the number $$n_{(s,t)}=\sum\limits_{\lambda \in L} P_z\{(s,t)+\lambda \}$$ and $\iota$ have the same parity for any $(s,t)\in \Z^2$. If we take now $(s,t)=(0,0)$ we note that $n_{(s,t)}$ is even since $P\{ \alpha\}=P\{-\alpha\}$ for every $\alpha \in \Z^2$, and $P\{(0,0)\}=0$. Therefore $\iota$ is even and then it is $0$.

We deduce then that $\vv_{v^{-1}}, \vv_{w^{-1}} \in \Z^2$ are linearly dependent, so $L$ is a cyclic group (a ``line"). Since $n_{(s,t)}$ is even for any $(s,t)\in \Z^2$, it means that if the coefficient $P\{(s,t)\}$ is $1$, then $P\{(s,t)+\lambda\}=1$ for some $0\neq \lambda \in L$. In other words, each line parallel to $L$ in $\Z^2$ which touches a non-trivial coefficient of $P_z$ must pass through another non-trivial coefficient. But this is not possible. If $(s,t)=(-2,-2)$, then $P\{(s,t)\}=1$, so $L=\langle (1,2)\rangle$, $\langle (3,2)\rangle$, $\langle (1,3)\rangle$, $\langle (3,1)\rangle$ or $L$ contains the point $(4,4)$, so $L\subseteq \langle (1,1)\rangle$. In any of these cases there is another non-trivial coefficient $(s',t')$ such that $L+(s',t')$ contains no other non-trivial coefficient. 

\textit{Step 2. The elements with odd area are a product of three squares.} In Equation (\ref{squarel}) take $v=c^{-1}x^{-1}$ and $w=d^{-1}y^{-1}$ for some $c,d\in \F'$. Then $m_{v^{-1}w^{-1}}=XY$, $m_{w^{-1}}=Y$ and the winding invariant of $[v^{-1},w^{-1}]=xcydc^{-1}x^{-1}d^{-1}y^{-1}$ is $P_c(X-XY)+P_d(XY-Y)+1$. Thus 

\begin{equation}
P_z=P_a(1+XY)+XYP_c(1-Y)+Y^2P_d(X-1)+Y.
\label{squarel2}
\end{equation} 

Then $P_z\in \langle 1+XY, Y-1,X-1\rangle+Y=\langle 2,Y-1,X-1\rangle +1 \subseteq \ZZ$. Conversely, if $P_z$ lies in $\langle 2,Y-1,X-1\rangle +1$, then we can choose $a,c,d\in \F'$ so that Equation (\ref{squarel2}) holds, which means that $z$ is a product of three squares in $\M$. But the ideal $\langle 2,Y-1,X-1\rangle$ is exactly the set of polynomials with even signed area, so if $P_z(1,1)$ is odd, then $z\in \M$ is a product of three squares. 

\textit{Step 3. If $P_z \equiv 1+X \textrm{ mod } \langle 2,Y-1,X^2-1\rangle$, $z$ is a product of three squares}. Take now $v=c^{-1}x^{-2}$, $w=d^{-1}y^{-1}$, so Equation (\ref{squarel}) becomes

\begin{equation}
P_z=P_a(1+X^2Y)+X^2YP_c(1-Y)+Y^2P_d(X^2-1)+Y(1+X).
\label{squarel3}
\end{equation}

Then $P_z\in \langle 1+X^2Y, Y-1, X^2-1 \rangle+Y(1+X)=\langle 2, Y-1, X^2-1 \rangle+1+X$. And conversely, if $P_z$ satisfies this, then $z\in \M$ is a product of three squares. So, when the signed area of $P_z$ is even, still $z$ can be a product of three squares. If $z$ is not a product of three squares then $P_z\in \langle 2, Y-1, X^2-1 \rangle$. 

\textit{Step 4. Every element in $\M$ which is a product of squares is a product of five squares}. Let $r \in \M$ be a product of squares, that is, $\exp(x,r)=2k$ and $\exp(y,r)=2l$ for some $k,l\in \Z$. Then $r=(x^{k+1}y^l)^2x^{-2}z$ for some $z\in \M'$. On the other hand $r=(x^{k+1}y^l)^2(x^{-1}yxy^{-1}x^{-1})^2z'$ for some $z'\in \M'$. Concretely $z'=[x,y]x[x,y]x^{-1}z$. The winding invariant of $z'$ is $P_{z'}=1+X+P_z$. Thus, $z$ or $z'$ (maybe both) is a product of three squares in $\M$ and then $r$ is a product of five squares.
\end{proof}

%
%
%
%

Theorem \ref{mainsquarel} leaves open the question whether the square length of $\Sq(\M)$ is $4$ or $5$. An alternative to Equation (\ref{squarel}) which characterizes the elements in $\M'$ which are products of three squares, can be obtained by applying Edmunds idea \cite{Edm3} that there is an automorphism of $\mathbb{F}_3=F(x,y,z)$ which maps $x^2y^2z^2$ to $[x,y]z^2$. Then an element $z\in \F'$ is a product of three squares if and only if there exist $u,v,w\in \F$ with $z=[u,v]w^2$. Thus, $z\in \M'$ is a product of three squares in $\M$ if and only if $P_z=2Q+P_{[u,v]}$ for certain $Q\in \ZZ$, $u,v\in \M$. Using this idea, an element $z\in \M'$ is a product of four squares if and only if there exist $u,v,w,t\in \M$ with $z=[u,v]w^2t^2$. This is equivalent to $P_z$ being of the form $P_{[u,v]}+Q(1+X^kY^l)$ for certain $u,v\in \M$, $Q\in \ZZ$, $k,l\in \Z$.

Other groups which are known to have finite square length have been studied in \cite{AM1, AM12, AM2}.

\subsection{Products of higher powers} \label{subsecpowers}

An element $z\in \F$ is a product of squares if and only if the class of $z$ in $\M$ is a product of squares in $\M$. The same is true for cubes.

\begin{prop} \label{propcubes}
Let $z\in \F$. Then $z$ is a product of cubes in $\F$ if and only if the class of $z$ in $\M$ is a product of cubes.
\end{prop}
\begin{proof}
The result follows from the fact that any element in $\F''$ is a product of elements of the form $[[u_i,v_i],[w_i,t_i]]$ for $u_i,v_i,w_i,t_i \in \F$ (see \cite[pp. 297]{MKS}) and each of these commutators is a product of cubes by \cite{LW} (see also \cite[p. 97]{Lyn4}). 
\end{proof}

Since $P_{[x,y]}=1$, the signed area $P_z(1,1)\in \Z$ of an element $z\in \F'$ which is a product of squares can be arbitrary. The signed area of an element $z\in \F'$ which is a product of cubes must be a multiple of $3$. More generally we prove:

\begin{teo} \label{npowers}
Let $z\in \F'$ be an element which is a product of $n$-th powers in $\F$ for some $n\ge 2$. Then the signed area $P_z(1,1)$ is a multiple of $n$ if $n$ is odd and it is a multiple of $\frac{n}{2}$ if $n$ is even. 
\end{teo}
\begin{proof}
Suppose $z=u_1^nu_2^n\ldots u_k^n$ for some $u_1,u_2,\ldots, u_k\in \F$. Since $\exp(x,u_i^n)$ and $\exp(y,u_i^n)$ are multiples of $n$ for each $i$, there exist $v_2,v_3,\ldots, v_{k-2}\in \F$ such that $u_1^nu_2^nv_2^n\in \F'$, $v_i^{-n}u_{i+1}^nv_{i+1}^n\in \F'$ for every $2\le i\le k-3$ and $v_{k-2}^{-n}u_{k-1}^nu_k^n \in \F'$. Thus, it suffices to prove the theorem in the case that $z$ is a product of three $n$-th powers in $\F$.

Suppose then $z=u^nv^nw^n$ for some $u,v,w\in \F$. If we replace $u$ by an element $\widetilde{u}\in \F$ with the same total exponents in $x$ and $y$, then we obtain $\widetilde{z}=(\widetilde{u})^nv^nw^n\in \F'$ and $P_{\widetilde{z}}(1,1)-P_{z}(1,1)$ is a multiple of $n$. Indeed, $z^{-1}\widetilde{z}=w^{-n}v^{-n}u^{-n}(\widetilde{u})^nv^nw^n$ is a conjugate of $u^{-n}\widetilde{u}^n$, which has signed area multiple of $n$ by Proposition \ref{lngeneral}.

If instead we replace $v$ by $\widetilde{v}\in \F$ or $w$ by $\widetilde{w}\in \F$ with the same total exponents, we again obtain an element $\widetilde{z}=u^n(\widetilde{v})^nw^n$ or $\widetilde{z}=u^nv^n(\widetilde{w})^n$ such that $P_{\widetilde{z}}(1,1)\equiv P_z(1,1) \mod (n)$. Therefore the class of $P_z(1,1) \mod (n)$ depends only on the total exponents of $x,y$ in $u,v,w$. Let $(a,b)=\vv_u=(\exp (x,u),\exp(y,u))$, $(c,d)=\vv_u+\vv_v=(\exp(x, uv), \exp(y,uv)) \in \Z^2$. Consider the rectangle $R$ contained in the grid $\Z \times \R \cup  \R \times \Z$ circumscribed about the triangle $T$ with vertices $(0,0)$, $(na,nb)$, $(nc,nd)$. Either 1. one vertex of $T$ is a vertex $\mathfrak{v}$ of $R$ and the other two vertices of $T$ lie in the two sides of $R$ which are not adjacent to $\mathfrak{v}$, or 2. two vertices of $T$ are opposite vertices of $R$ and the third vertex of $T$ lies in the interior or boundary of $R$ (see Figure \ref{triangulos}).

\begin{figure}[h] 
\begin{center}
\includegraphics[scale=0.6]{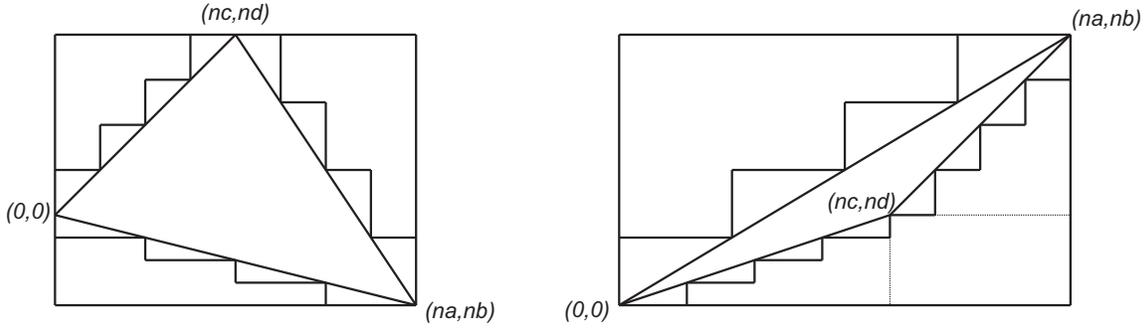}
\caption{The cases 1. and 2. in an example for $n=4$.}\label{triangulos}
\end{center}
\end{figure}

We choose $u=x^ay^b$ or $u=y^bx^a$, $v=x^{c-a}y^{d-b}$ or $v=y^{d-b}x^{c-a}$ and $w=x^{-c}y^{-d}$ or $w=y^{-d}x^{-c}$ in such a way that the curve $\gamma_{u^nv^nw^n}$ does not touch the interior of $T$. In case 1, the signed area of $z=u^nv^nw^n$ up to sign is obtained from the area of $R$, which is a multiple of $n^2$, by subtracting the areas of three regions, each of area a multiple of $\frac{n(n-1)}{2}$. In case 2, $P_z(1,1)$ up to sign is obtained from the area of $R$ by subtracting the area of four regions. Three of them have area multiple of $\frac{n(n-1)}{2}$ and the fourth (a rectangle) area a multiple of $n^2$. In any case the greatest common divisor of $n^2$ and $\frac{n(n-1)}{2}$ divides $P_z(1,1)$. This is the number in the statement of the theorem.
\end{proof}

\begin{coro} \label{npowersf}
Let $z\in \F$ and let $k=\exp(x,z)$, $l=\exp(y,z)$. If $z$ is a product of $n$-th powers then the signed area of $y^{-l}x^{-k}z$ is divisible by $n$ if $n$ is odd and by $\frac{n}{2}$ if $n$ is even.
\end{coro}
\begin{proof}
If $z$ is a product of $n$-th powers, $n |k,l$, so $y^{-l}x^{-k}z\in \F'$ is a product of $n$-th powers. 
\end{proof}

For $n=3$ we have the following characterization of the products of cubes in $\F$.

\begin{teo} \label{caracterizacioncubos}
Let $z\in \F$. Then $z$ is a product of cubes in $\F$ if and only the total exponents $k=\exp (x,z)$ and $l=\exp (y,z)$ are divisible by $3$ and the signed area of $y^{-l}x^{-k}z$ is divisible by $3$.
\end{teo}
\begin{proof}
One implication follows directly from Corollary \ref{npowersf}. For the other, note that the winding invariant of $x^3(x^{-1}y)^3y^{-3}(yx)^3(x^{-1}y^{-1}x)^3x^{-3}$ is $Y(1-X)$. Symmetrically, $X(1-Y)$ is the winding invariant of a product of cubes. The winding invariant $1+X+Y$ of $x^3(x^{-1}y)^3y^{-3}$ has signed area $3$. Therefore, the ideal $\langle 1-Y, 1-X, 1+X+Y \rangle \subseteq \ZZ$ of polynomials with signed area divisible by $3$ is contained in the ideal $\{P_w | w\in \F' $ is a product of cubes$\}$. We deduce then that if $z\in \F$ has exponents $k,l$ divisible by $3$ and $y^{-l}x^{-k}z$ has signed area divisible by $3$, then $y^{-l}x^{-k}z$ is a product of cubes in $\M$. By Proposition \ref{propcubes}, $y^{-l}x^{-k}z$ is a product of cubes in $\F$ and then so is $z$.
\end{proof}

\begin{coro}
Let $n,k\ge 1$. If $n$ is odd and $n \nmid k$, then $[x,y]^k\in \F$ is not a product of $n$-th powers. If $n$ is even and $\frac{n}{2} \nmid k$, then $[x,y]^k\in \F$ is not a product of $n$-th powers. 
\end{coro}

If $n|k$, clearly $[x,y]^k$ is a product of $n$-th powers. However, for $n$ even, the fact that $\frac{n}{2}| k$ does not guarantee that $[x,y]^k$ is a product of $n$-th powers as the next result shows. We are going to use a coloring argument to define a variant of the signed area. We will paint the squares determined by the grid $\Z \times \R \cup \R \times \Z$. All the squares in rows congruent to $0$ or $1$ modulo $4$ will be painted with black, while the remaining rows are white (see Figure \ref{sombra}).

\begin{figure}[h] 
\begin{center}
\includegraphics[scale=0.4]{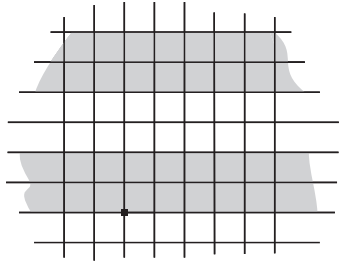}
\caption{A coloring of the squares in the grid.}\label{sombra}
\end{center}
\end{figure}

Now for a given $z\in \F'$ we will consider the coefficients $P_z\{(k,l)\}$ of the winding invariant $P_z$. We add all the coefficients corresponding to black squares and we subtract all the coefficients corresponding to white squares. Concretely, for $P \in \ZZ$, define $$\kappa_P=\sum\limits_{l \equiv 0,1 \ \textrm{mod} (4)} (\sum\limits_{k\in \Z} P\{(k,l)\})- \sum\limits_{l \equiv 2,3 \ \textrm{mod} (4)} (\sum\limits_{k\in \Z} P\{(k,l)\}) \in \Z,$$ and for $z\in \F'$ define $\kappa_z=\kappa_{P_z}$.

\begin{prop} \label{color}
Let $z\in \F'$. If $z$ is a product of fourth powers in $\F$, then $4|\kappa_z$.
\end{prop}
\begin{proof}
We follow the proof of Theorem \ref{npowers}. Since $\kappa: \F' \to \Z$ is a group homomorphism, we can assume that $z=u^4v^4w^4$ is a product of three fourth powers. Replacing $u$, $v$ or $w$ by other elements in $\F$ with the same total exponents does not change the congruence of $\kappa_z$ modulo $4$ since for any monomials $m=X^kY^l, m'=X^{k'}Y^{l'}$, the polynomial $P=m'(1+m+m^2+m^3)$ has $\kappa_P \equiv 0 \mod 4$. Indeed, if $l\equiv 0 \mod 4$, then $\kappa_P=4$ or $-4$, and otherwise $\kappa_P=0$. So we may suppose $u,v$ and $w$ are as in the proof of Theorem \ref{npowers}. Since the rectangles in the cases 1. and 2. have sides of length a multiple of $4$, then we only need to prove that the invariant $\kappa$ is trivial modulo $4$ in the case of the stair-like regions. That is, for $z'=x^{4n}(x^{-n}y^m)^4y^{-4m}$ we want to prove that $4|\kappa_{z'}$. In other words, for a three-step stair as in Figure \ref{escalerita}, we want to prove that the number of black squares minus the number of white squares is a multiple of $4$.

\begin{figure}[h] 
\begin{center}
\includegraphics[scale=0.8]{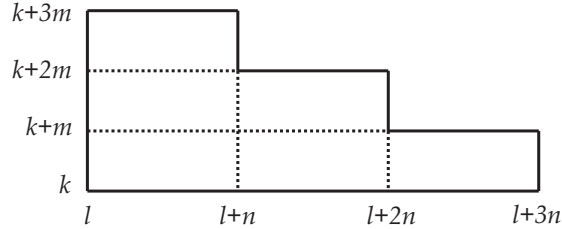}
\caption{A three-step stair.}\label{escalerita}
\end{center}
\end{figure}

The stair consists of six $n\times m$ blocks, three in the first level, two in the second and one in the third. We don't know which squares are black and which white, since the stair may be anywhere. If $m$ is even, then the block in the third step is painted with the same colors as the blocks in the first level, so there are four of these blocks and we only care about the two blocks in the second level. Since $m$ and $2n$ are even it is easy to see that the number of black squares minus the number of white squares is a multiple of $4$. If $m$ is odd, then the block in third step is painted with the opposite colors as the block in the first step (the rightmost block in the first level), so we only care about the $2n\times 2m$ rectangle formed by the remaining four blocks. Again since its sides have even length then the difference between the numbers of squares of each color is a multiple of $4$. 
\end{proof}

\begin{coro}
$[x,y]^2$ is not a product of fourth powers in $\F$. 
\end{coro}

Recall that the Burnside problem is related to the problem of expressing Engel words as a product of powers. It is known that the fifth Engel word $e_5=[y[y[y[y[y,x]]]]]$ is a product of fourth powers in $\F$.

\begin{coro} \label{etres}
The third Engel word $e_3=[y,[y,[y,x]]]$ is not a product of fourth powers in $\F$.
\end{coro} 
\begin{proof}
The winding invariant of $e_3$ is $-(Y-1)^2$, thus $\kappa_{e_3}=2$ is not a multiple of $4$ and Proposition \ref{color} applies.
\end{proof}

As we mentioned in Section \ref{secuno} it is an open problem whether $e_6$ is a product of fifth powers in $\F$. If it is not, then the Burnside group $B(2,5)$ is not finite. The signed area of $e_6$, as for any other Engel word $e_n$ with $n\ge 2$, is trivial, so Theorem \ref{npowers} is not useful. In fact $e_6$ is a product of three fifth powers in $\M$. Concretely, $$e_6=(y[y,x]y[x,y]^2y[y,x]^2y[x,y]y^{-4})^5(xyx^{-1})^5(yxy^{-1}x^{-1}y^{-1})^5 \in \M.$$

More generally we prove the following

\begin{prop} \label{ppowers}
Let $p$ be a positive prime number. Then $e_{p+1}$ is a product of three $p$-th powers in $\M$. In particular, $e_n$ is a product of $p$-th powers in $\M$ for every $n>p$.
\end{prop}
\begin{proof}
For $p=2$ note that $P_{e_3}=-(Y-1)^2=-1+2Y-Y^2$ coincides with the winding invariant of $y[x,y]^2y^{-1}y^4(y^{-1}xy^{-1}x^{-1})^2$, so $e_3$ is a product of three squares in $\M$. If $p\ge 3$, the winding invariant of $e_{p+1}$ is $(1-Y)^p=1-Y^p+pQ$ for certain $Q\in \ZZ$. Let $z\in \F'$ be such that $P_z=Q$, so $P_{z^p}=pQ$. The winding invariant of $w=(xyx^{-1})^py(xy^{-1}x^{-1})^py^{-1}$ is $P_w=1-Y^p$. Thus $e_{p+1}=z^pw\in \M$ is a product of three $p$-th powers. Since $e_n=[y,e_{n-1}]=ye_{n-1}y^{-1}e_{n-1}^{-1}$, the fact that $e_{n-1}$ is a product of $p$-th powers implies that $e_n$ is a product of $p$-th powers. The second part of the statement follows. 
\end{proof}


\section{Cocommutative presentations and the Andrews-Curtis conjecture}	\label{sectionac}

In this section we use the winding invariant to study $Q$-equivalence of presentations with relators in the commutator subgroup. 

\subsection{Nielsen equivalence}
	
Let $w_1, w_2, \ldots , w_m\in \FF$ and let $w\in \F$. If $w$ is in the normal subgroup of $\F$ generated by the $w_j$, then $w\in \FF$ and $P_w$ is a $\Z-$linear combination of the polynomials $X^kY^lP_{w_j}$ for $1\le j\le m$, $k,l\in \Z$. This follows directly from Proposition \ref{basica}. In particular, in this case, $P_w(1,1)$ is a linear combination of the integers $P_{w_j}(1,1)$. In other words the greatest common divisor $\textrm{gcd}\{P_{w_1}(1,1), P_{w_2}(1,1),\ldots , P_{w_m}(1,1) \}$ (if defined) divides $P_{w}(1,1)$. With the same idea one proves that $\textrm{gcd}\{P_{w_1}(\epsilon,\epsilon'), P_{w_2}(\epsilon,\epsilon'),\ldots , P_{w_m}(\epsilon,\epsilon') \}$ divides $P_{w}(\epsilon,\epsilon')$ for every $\epsilon, \epsilon'\in \{1,-1\}$. The idea of defining a map $\phi: \FF\to A$ to an abelian group $A$ such that $\phi(w) \notin \langle \phi(w_j) \rangle_{1\le j\le m}$ is what Conway and Lagarias call a generalized coloring argument in \cite[Section 5]{CL}.  

\begin{ej}
Let $w_1=x^2yx^{-1}yx^{-1}y^{-2}$, $w_2=[x,y]x^2[x,y]x^{-2}$, $w=x^2y^2x^{-1}y^{-1}x^{-1}y^{-1} \in \FF$. Is $w$ in the normal subgroup $N$ of $\F$ generated by $w_1$ and $w_2$?

Following the idea of the previous paragraph we compute $P_{w_1}=1+X+Y$, $P_{w_2}=1+X^2$, $P_w=1+X+XY$. $P_{w_1}(1,-1)=P_{w_1}(-1,1)=1$  and $P_{w_1}(-1,-1)=-1$. Moreover, $P_{w_1}(1,1)=3$ while $P_{w_2}(1,1)=2$. Therefore $\textrm{gcd}\{P_{w_1}(\epsilon, \epsilon'), P_{w_2}(\epsilon, \epsilon')\}$ divides $P_{w}(\epsilon, \epsilon ')$ for every $\epsilon, \epsilon ' \in \{1,-1\}$. However a further computation gives $P_{w_1}(3,1)=5$, $P_{w_2}(3,1)=10$, $P_w(3,1)=7$. If $w\in N$, then $P_w$ is a linear combination of polynomials $X^kY^lP_{w_j}$ for $k,l \in \Z$. In particular $7=P_w(3,1)=\sum\limits_{k\in \Z} (5n_k3^k+10m_k3^k)$ for certain integers $n_k, m_k$, only finitely many of them nonzero. Thus, $5$ divides $3^k7$ for some $k\ge 0$, a contradiction. Therefore, $w\notin N$. 
\end{ej}

Let $S,T$ be two subsets of $\F'$ and suppose that $\langle x,y | S \rangle$, $\langle x,y | T \rangle$ are presentations of isomorphic groups with an isomorphism which fixes the generators (that is the normal closures $N(S)$ and $N(T)$ are equal). Let $(\epsilon, \epsilon ')\in \{1,-1\}^2$ and suppose there exist $w\in S, u\in T$ with $P_w(\epsilon, \epsilon')\neq 0 \neq P_u(\epsilon, \epsilon')$. Then by the comments above $$\textrm{gcd}\{P_{w}(\epsilon,\epsilon') | w\in S \}=\textrm{gcd}\{P_{w}(\epsilon,\epsilon') | w\in T \}.$$

Two presentations $\PP=\langle x_1,x_2, \ldots, x_n | S \rangle$, $\Q=\langle x_1,x_2,\ldots, x_n | T \rangle$ with the same generator set are said to be Nielsen equivalent if there exists an automorphism $\phi$ of $\mathbb{F}_n$ such that $\phi(N(S))=N(T)$. Of course, Nielsen equivalent presentations present isomorphic groups. The converse is not true, not even for one-relator groups (see \cite{Bru, Dun-1, DP, MP, Rap, Pri1, Pri15}). However, it does hold for the presentations that we are interested in. Call a presentation $\langle x,y| S \rangle$ \textit{cocommutative} if $S\subseteq \F'$. We have already studied these presentations in Section \ref{subcoco}.  

\begin{lema} [Dunwoody] \label{propdun}
Let $\langle x,y | S\rangle$ and $\langle x,y | T\rangle$ be cocommutative presentations of a group $G$. Then $\PP=\langle x,y | S\cup \F'' \rangle$ and $\Q=\langle x,y | T \cup \F''\rangle$ are Nielsen equivalent.
\end{lema}
\begin{proof}
Since $G\simeq \F / N(S) \simeq \F/ N(T)$, then $G/G''\simeq \F/N(S)\F''\simeq \F/N(T)\F''$. Thus, $\PP$ and $\Q$ present $G/G''$.  This group is a $(0,2)$-group in the sense of Dunwoody \cite[pp. 18]{Dun0}: it is generated by two elements and there is an epimorphism $G/G'' \to \Z \times \Z$. Since $G/G''$ is metabelian, the proposition follows from Theorem 4.10 in \cite{Dun0}.  
\end{proof}

Perhaps it is worth recalling that the isomorphism problem for metabelian groups remains open \cite{Bau8, Bau9}.

\begin{prop}
Suppose that $\PP=\langle x,y | w_1, w_2, \ldots , w_m \rangle$ and $\Q=\langle x,y | u_1, u_2, \ldots , u_{m'} \rangle$ are cocommutative presentations of the same group. Moreover suppose not every relator of $\PP$ has signed area equal to $0$ and the same holds for the relators of $\Q$. Then $$\textrm{gcd}\{P_{w_1}(1,1), P_{w_2}(1,1),\ldots , P_{w_m}(1,1) \}=\textrm{gcd}\{P_{u_1}(1,1), P_{u_2}(1,1),\ldots , P_{u_{m'}}(1,1) \}.$$ 
\end{prop}
\begin{proof}
By Lemma \ref{propdun} there is an automorphism $\phi \in \textrm{Aut}(\F)$ such that $\phi(N(w_1,w_2,\ldots, w_m)$ $\F'')=N(u_1,u_2, \ldots, u_{m'})\F''$. By Corollaries \ref{endo} and \ref{corounit}, the signed area of $\phi (w_j)$ is equal to the signed area of $w_j$ up to sign. Since $\phi(w_j)\in N(u_1,u_2,\ldots, u_{m'})\F''$, $P_{\phi(w_j)}(1,1)$ (and then $P_{w_j}(1,1)$) is a linear combination of the numbers $P_{u_l}(1,1)$, $1\le l\le m'$. Thus the greatest common divisor of these numbers divides $P_{w_j}(1,1)$. The same holds in the other direction. 
\end{proof}

Recall that for a subset $S\subseteq \F'$, $I(W(S))$ denotes the ideal of $\ZZ$ generated by the polynomials $P_w$ with $w\in S$. 

\begin{prop} \label{isorings}
If $\PP=\langle x,y | S\rangle$ and $\Q=\langle x,y | T \rangle$ are cocommutative presentations of the same group, $\ZZ /I(W(S))$ and $\ZZ/I(W(T))$ are isomorphic rings.
\end{prop}
\begin{proof}
Let $\phi \in \textrm{Aut}(\F)$ be as in the previous proposition. Let $\overline{\phi}:\Z^2\to \Z^2$ denote the map induced in the abelianization. By Proposition \ref{propzdephi} and Corollary \ref{corounit}, $\Z[\overline{\phi}]:\ZZ\to \ZZ$ induces the desired ring isomorphim.
\end{proof}

\subsection{The Andrews-Curtis conjectures}

Given a presentation $\PP=\langle x_1,x_2,\ldots, x_n | w_1,w_2,$ $ \ldots,$ $w_m \rangle$, one way to obtain another presentation $\Q=\langle x_1,x_2,\ldots, x_n | u_1,u_2,\ldots , u_m \rangle$ of the same group with an isomorphism fixing the generators, is by a sequence of $Q$-transformations, that is transformations of the following type:

\medskip

\noindent (i) replace a relator $w_j$ by $w_j^{-1}$,

\noindent (ii) replace a relator $w_j$ by $w_jw_i$, for some $i\neq j$,

\noindent (iii) replace a relator $w_j$ by a conjugate $uw_ju^{-1}$, for some $u\in \mathbb{F}_n$.

\medskip

If we furthermore allow the following transformation

\medskip

\noindent (iv) replace each relator $w_j$ by $\phi(w_j)$ for some automorphism $\phi$ of $\mathbb{F}_n$,

\medskip

\noindent we obtain what are called $Q^*$-transformations.

Presentations connected by a sequence of $Q$-transformations or $Q^*$-transformations are called $Q$-equivalent and $Q^*$-equivalent respectively. $Q^*$-equivalence implies Nielsen equivalence. 
Finally, we add the moves

\medskip

\noindent (v) add a generator $x_{n+1}$ and a relator $w_{m+1}=x_{n+1}$ and

\noindent (vi) the inverse of (iv), if possible,

\medskip

\noindent to obtain the $Q^{**}$-transformations, which generate the $Q^{**}$-equivalence relation.

The original Andrews-Curtis conjecture \cite{AC} states that any balanced (same number of relators as generators) presentation $\langle x_1,x_2,\ldots , x_n | w_1, w_2,\ldots , w_n \rangle$ of the trivial group is $Q$-equivalent to $\langle x_1,x_2,\ldots , x_n | x_1, x_2,\ldots , x_n \rangle$.  The weak version of the Andrews-Curtis conjecture states that any balanced presentation of the trivial group is $Q^{**}$-equivalent to the trivial presentation $\langle \ | \ \rangle$ with no generators and no relators. In this particular case it can be proved that move (iv) is not needed in virtue of Nielsen's theorem that any automorphism of a free group is given by a sequence of Nielsen transformations (moves (i) and (ii)).

The most interesting feature of these problems is probably their geometric counterpart, which is given by the standard construction which associates a $2$-dimensional CW-complex to each presentation. Two CW-complexes $K$ and $L$ are simple homotopy equivalent if we can obtain one from the other by a sequence of collapses and expansions \cite{Coh}. Simple homotopy equivalence implies homotopy equivalence and the converse is false as proved by Whitehead. In fact there are $2$-dimensional complexes which are homotopy equivalent and not simple homotopy equivalent \cite{Lus, Met}. When the complexes are simply connected, or more generally their Whitehead group is trivial, simple homotopy and homotopy are the same. Nielsen equivalence classes and simple homotopy types are distinguished in \cite{Lus, Lus2} with the 1 and 2 dimensional versions of the same invariant. If the sequence of collapses and expansions above involves only complexes of dimension less than or equal to $3$, we say that $K$ $3$-deforms to $L$. The weak Andrews-Curtis conjecture is equivalent to the statement that any finite and contractible $2$-dimensional complex $3$-deforms into a point.

The generalized Andrews-Curtis conjecture \cite[Section 4.1]{HM} states that any two finite presentations with simple homotopy equivalent standard complexes are $Q^{**}$-equivalent. Of course, this version implies the weak version of the conjecture. This conjecture is equivalent to the following conjecture: two simple homotopy equivalent finite $2$-complexes $3$-deform one into the other. All these three versions of the Andrews-Curtis conjecture are open.

Cocommutativity is preserved by $Q^*$-transformations and we can use the winding invariant to study $Q^*$-equivalence of such presentations. We can associate with each cocommutative presentation $\PP=\langle x,y | w_1,w_2, \ldots, w_m \rangle$ the vector $\Lambda(\PP)=\{P_{w_1},P_{w_2},\ldots, P_{w_m}\}$. The effect on $\Lambda(\PP)$ of performing a $Q$-transformation to $\PP$ is to change a polynomial $P_{w_j}$ by $-P_{w_j}$, or by $P_{w_j}+P_{w_i}$ for certain $i\neq j$ or by $X^kY^lP_{w_j}$ for certain $k,l\in \Z$. In particular the new vector is the original column vector $\Lambda(\PP)$ multiplied by an elementary or diagonal matrix in $GL_m(\ZZ)$. Recall that a square matrix is said to be elementary if all the diagonal coefficients are $1$ and all the other coefficients but one are $0$. We exploited this idea in \cite{Barac} to obtain the presentations $$\PP=\langle x,y | [x,y],1 \rangle$$ and $$\Q=\langle x,y | [x,[x,y^{-1}]]^2y[y^{-1},x]y^{-1},[x,[[y^{-1},x],x]] \rangle $$ with vectors $\Lambda(\PP)=(1,0)$, $\Lambda(\Q)=(1-2(X-1)Y^{-1}, -(X-1)^2Y^{-1})$ which do not differ in a multiplication by a product of elementary and diagonal matrices. Then $\PP$ and $\Q$ are not $Q$-equivalent. Moreover, if $\phi \in Aut(\F)$, then by Corollary \ref{corounit}, $P_{\phi([x,y])}$ is a unit in $\ZZ$ so $\Lambda(\PP)=(1,0)$ and $\Lambda(\langle x,y | \phi([x,y]),1\rangle)$ differ in a multiplication by a diagonal invertible matrix. Then $\PP$ and $\Q$ are not $Q^*$-equivalent either. On the other hand, it can be proved that the standard complexes $K_{\PP}$ and $K_{\Q}$ are homotopy equivalent, and since their Whitehead group is trivial, we conclude:

\begin{teo} (\cite{Barac})
The presentations $\PP$ and $\Q$ are not $Q^*$-equivalent, though their standard complexes are simple homotopy equivalent. 
\end{teo}  

The idea of mapping the relators of a presentation to a test group $G$ and studying $Q$ or $Q^{**}$-transformations of vectors in $G$ is not new (see \cite{HM3}). However this is the first counterexample we know about of the stronger conjecture that simple homotopy equivalence implies $Q^*$-equivalence. It is not clear how to extend our methods to study the original version of Andrews-Curtis.

\subsection{$Q$-equivalence in the free metabelian group of rank $2$}

The notion of $Q$-equivalence can be extended to arbitrary groups. An $m$-tuple $(w_1,w_2,\ldots, w_m)$ of elements in a group $G$ can be transformed into another $m$-tuple by replacing one $w_j$ by $w_j^{-1}$, by $w_jw_i$ for some $i\neq j$ or by $gw_jg^{-1}$ for some $g\in G$. We call these \textit{$Q$-transformations} as well, and they generate the notion of $Q$-equivalence.

Myasnikov proved in \cite{Mya} the following

\begin{teo}[Myasnikov]
Let $F$ be a finitely generated free group and let $N\trianglelefteq F$ be a normal subgroup such that $F/N$ is free abelian of rank $n\ge 2$. Let $G=F/N'$. Then two $n$-tuples in $G$ whose normal closure is $G$ are always $Q$-equivalent.
\end{teo}

When the normal closures are not $G$, the result does not hold. In \cite[Theorem 1.6]{MMS}, Myasnikov, Myasnikov and Shpilrain proved the following

\begin{teo}[Myasnikov, Myasnikov, Shpilrain]
Let $n\ge 3$ and let $N$ be the normal closure of $\{[x_1,x_2],x_3\}$ in $F=F(x_1,x_2,\ldots ,x_n)$. Then in $G=F/N'$ there are infinitely many pairwise $Q$-inequivalent 2-element sets which have normal closure $N/N'$.
\end{teo}

We use the winding invariant to prove that something very similar holds in the free metabelian group of rank 2, $\M=\F/\F''$. The proof is similar to that in \cite{MMS}, but we need stronger results to be able to deal with non-aspherical presentations.

Denote $R=\ZZ$. Let $E_2(R)$ be the subgroup of $SL_2(R)$ generated by the elementary matrices and $GE_2(R)$ the subgroup of $GL_2(R)$ generated by the elementary matrices and the diagonal matrices. \label{pagmatrices} Since the subgroup $D_2(R)\subseteq GL_2(R)$ of diagonal matrices normalizes $E_2(R)$, then $GE_2(R)=D_2(R).E_2(R)$. Since $D_2(R)\cap SL_2(R)\subseteq E_2(R)$ by Whitehead's Lemma, $SL_2(R)\cap GE_2(R)= E_2(R)$. 

\begin{teo} [Bachmuth, Mochizuki \cite{BM}] \label{bach}
Any generating set of $SL_2(R)$ contains infinitely many elements which are not in $E_2(R)$.
\end{teo}

We will need a version of Theorem \ref{bach} for $GL_2(R)$ and $GE_2(R)$. The proof is inspired by the proof of McCullough of \cite[Theorem 3]{Mcc}.

\begin{teo} \label{bachnuevo}
Any generating set of $GL_2(R)$ contains infinitely many elements which are not in $GE_2(R)$.
\end{teo}
\begin{proof}
The group $R^*$ of units of $R$ is generated by $-1,X,Y$. Let $R_0\subseteq R$ be the subgroup generated by $X^2$ and $Y^2$ and let $H=\det ^{-1}(R_0) \subseteq GL_2(R)$.
 
Suppose $A=GE_2(R) \cup \{M_1,M_2,\ldots ,M_k\}$ is a generating set of $GL_2(R)$. A left transversal of $H$ in $GL_2(R)$ is given by the eight matrices $N_u=\left(
\begin{array}{cc}
u & 0 \\
0 & 1 
\end{array} \right)$ with $u\in \{1,-1,X,Y,$ $XY,$ $-X,-Y,-XY\}$. By Schreier's Lemma, a generating set $B$ of $H$ is given by the matrices $N_uEN_v^{-1}\in H$ with $E\in GE_2(R)$ and the matrices $N_uM_iN_v^{-1}\in H$.

As in \cite{Mcc}, let $f:H\to SL_2(R)$ be defined by $f(M)=\det(M)^{-\frac{1}{2}}M$, where for $u=X^{2n}Y^{2m}\in R_0$, $u^{\frac{1}{2}}$ denotes the element $X^nY^m\in R^*$. Then $f$ is a retraction, so $f(B)$ is a generating set for $SL_2(R)$. Now, suppose $N_uEN_v^{-1}\in H$ is one of the elements in $B$ for some $E\in GE_2(R)$, $u,v\in \{1,-1,X,Y,XY,-X,-Y,-XY\}$. Then 

$$N_uEN_v^{-1}=\left(
\begin{array}{cc}
uv^{-1}\det(E) & 0 \\
0 & 1 
\end{array} \right) \left(
\begin{array}{cc}
v & 0 \\
0 & 1 
\end{array} \right) \left(
\begin{array}{cc}
\det(E)^{-1} & 0 \\
0 & 1 
\end{array} \right) E \left(
\begin{array}{cc}
v^{-1} & 0 \\
0 & 1 
\end{array} \right).$$ Since $ \left(
\begin{array}{cc}
\det(E)^{-1} & 0 \\
0 & 1 
\end{array} \right) E \in GE_2(R)\cap SL_2(R)=E_2(R)$ and $D_2(R)$ normalizes $E_2(R)$, then $N_uEN_v^{-1}= \left(
\begin{array}{cc}
uv^{-1}\det(E) & 0 \\
0 & 1 
\end{array} \right) E'$ for some $E'\in E_2(R)$. 

Finally, $f(N_uEN_v^{-1})=\left(
\begin{array}{cc}
(uv^{-1}\det(E))^{\frac{1}{2}} & 0 \\
0 & (uv^{-1}\det(E))^{-\frac{1}{2}} 
\end{array} \right)E'$, which lies in $E_2(R)$ by Whitehead's Lemma. Therefore, $f(B)$ contains only finitely many elements not in $E_2(R)$, which contradicts Theorem \ref{bach}.
\end{proof}

\begin{teo} \label{qineq}
There are infinitely many pairwise $Q$-inequivalent 2-element sets in $G=\M$ which have normal closure $\M'$.
\end{teo}
\begin{proof}
We use an idea similar to that in \cite[Proposition 5.1, Theorem 1.6]{MMS} with some modifications. By Theorem \ref{bachnuevo} there exists a sequence $M_1,M_2, \ldots$ of matrices in $GL_2(R)$ such that $M_i$ does not lie in the subgroup generated by $GE_2(R)$ and $M_1, M_2, \ldots, M_{i-1}$ for every $i\ge 1$. For each $i\ge 1$ take $w_i, w_i'\in \F'$ such that $(P_{w_i}, P_{w_i'})^t\in R^{2\times 1}$ is the first column of $M_i$. Taking classes modulo $\F''$ we obtain the pairs $S_i=(\overline{w}_i, \overline{w_i'}) \in \M^2$. We will prove that the sets $S_i$ satisfy the statement of the theorem.

Since $w_i,w_i' \in \F'$, the normal closure $N(S_i)$ of $S_i$ is contained in $\F'/\F''=\M'$. We prove the other inclusion. Since $M_i\in GL_2(R)$, its determinant lies in $R^*$. Then there exist Laurent polynomials $Q_i,Q_i'\in R$ such that $P_{w_i}Q_i+P_{w_i'}Q_i'=1\in R$. By Proposition \ref{basica} there exists $u_i$ in the normal closure of $\{w_i,w_i'\}$ such that $P_{u_i}=1=P_{[x,y]}$. Then $\F'/\F''=N(\overline{[x,y]})=N(\overline{u}_i)$ is contained in $N(S_i)$.

Now, suppose that $S_i$ and $S_j$ are $Q$-equivalent for some $i>j$. As in the beginning of the section, we can associate a vector $\Lambda (S)\in R^2$ to each $S\in (\M')^2$. Namely, for $S=(\overline{w}, \overline{w'})$ we define $\Lambda(S)=(P_{w},P_{w'})$. Performing a $Q$-transformation on $S$ amounts to multiplying $\Lambda(S)$ by a matrix in $GE_2(R)$. Thus, there exists $E\in GE_2(R)$ such that $EM_i\binom{1}{0}=E\Lambda(S_i)^t=\Lambda(S_j)^t=M_j\binom{1}{0}$. Then the first column of $M_j^{-1}EM_i$ is $\binom{1}{0}$. As in \cite{Barac} we can multiply by another matrix $E'\in GE_2(R)$ to obtain $E'M_j^{-1}EM_i=Id$. In particular $M_i$ lies in the subgroup of $GL_2(R)$ generated by $GE_2(R)$ and $M_j$, a contradiction.
\end{proof}

%


In Section \ref{seccayley} we will see the relationship between Theorem \ref{qineq}, the relation lifting problem, and the existence of infinitely many presentations $\langle x,y| r_i,s_i\rangle$ which are simple homotopy equivalent but pairwise $Q$-inequivalent.  

To finish this section we want to mention a connection with a question raised in \cite{BM}. It is known that $GE_2(\Z[X])=GL_2(\Z[X])$ (see \cite{Cohn}), but the following question is still open.

\begin{ques}
Is $GE_2(\Z[X,X^{-1}])$ equal to $GL_2(\Z[X,X^{-1}])$?
\end{ques} 

The relationship between this question and $Q$-equivalences was first noticed in \cite[Proposition 1.7]{MMS}. We state now a result similar to that one in the same spirit of Theorem \ref{qineq}.

\begin{teo} \label{xxmenos1}
Let $N$ be the normal closure of $y$ in $\F$ and let $G=\F/ N'$. Then every pair $(w_1,w_2)\in G^2$ with normal closure equal to $N/N'$ is $Q$-equivalent to $(\overline{y},1)$ if and only if $GE_2(\Z[X,X^{-1}])=GL_2(\Z[X,X^{-1}])$.
\end{teo} 
\begin{proof}
We define a map $W':N\to \Z[X,X^{-1}]$ as follows. Consider the space $\widetilde{K}$ obtained from the real line $\mathbb{R}$ by attaching a 1-dimensional sphere at each integer $n\in \Z\subseteq \mathbb{R}$. For $w\in N$ construct the curve $\gamma_w$ in $\widetilde{K}$ which walks from the origin following the letters of $w$: it moves one unit in the line to the right or to the left if the corresponding letter is $x$ or $x^{-1}$, and goes through the corresponding loop if the letter is $y$ or $y^{-1}$. The $n$-th coefficient of $W'(w)$ is the winding number of $\gamma_w$ around the loop attached at $n\in \mathbb{R}$. The complex $\widetilde{K}$ is the cover of the figure eight $K$ corresponding to the subgroup $N\triangleleft \pi_1(K)=\F$ and $W'$ is the abelianization $N\to N/N'=H_1(\widetilde{K})$. This variation of the winding invariant will be discussed with more generality in Section \ref{sectionhow}.

To each pair $S=(\overline{w}_1,\overline{w}_2)\in N/N'$ we associate $\Lambda(S)=(W'(w_1),W'(w_2))\in \Z[X,X^{-1}]^2$. A $Q$-transformation is reflected in $\Lambda(S)^t$ as a multiplication by a matrix in $GE_2(\Z[X,X^{-1}])$. If the normal closure of $\{\overline{w}_1,\overline{w}_2\}$ is $N/N'$ then there exist $A_1,A_2\in \Z[X,X^{-1}]$ such that $A_1W'(w_1)+A_2W'(w_2)=1$, so $\left(
\begin{array}{cc}
A_1 & A_2 \\
-W'(w_2) & W'(w_1) 
\end{array} \right) \Lambda(S)^t=\binom{1}{0}$.

If $GE_2(\Z[X,X^{-1}])=GL_2(\Z[X,X^{-1}])$, the matrix above lies in $GE_2(\Z[X,X^{-1}])$ and then there is a sequence of $Q$-transformations that we can perform on $S$ to obtain a pair $T=(\overline{u}_1,\overline{u}_2)$ with $\Lambda(T)=(1,0)$. Then $T=(\overline{y},1)\in G^2$.

Conversely, if $M=\left(
\begin{array}{cc}
A & B \\
C & D 
\end{array} \right) \in GL_2(\Z[X,X^{-1}])$, take $w_1,w_2\in N$ such that $W'(w_1)=A$ and $W'(w_2)=C$. Then the fact that $AD-CB$ is a unit in $\Z[X,X^{-1}]$ says that the normal closure of $\{\overline{w}_1,\overline{w}_2\}$ is $N/N'$. By hypothesis $(\overline{w}_1,\overline{w}_2)$ is $Q$-equivalent to $(1,0)$ so there exists $E\in GE_2(\Z[X,X^{-1}])$ such that the first column of $EM$ is $\binom{1}{0}$. Then a multiplication by a new matrix $E'\in GE_2(\Z[X,X^{-1}])$ gives $E'EM=Id$, so $M\in GE_2(\Z[X,X^{-1}])$.   
\end{proof}


\section{Asphericity and relation modules} \label{sectionhow}

We will see that the winding invariant is a very particular case of a more general idea which can be used to attack problems that appear in previous works.

\subsection{A general version of the winding invariant} \label{subversion}

In the previous sections we have used the winding invariant to study $\M'=\F ' / \F ''$ and more generally $G' / G''$ when $G$ is a group presented by a cocommutative presentation. The winding invariant $\W : \FF \to \ZZ$ can be interpreted as follows. Let $K$ be the standard complex of the presentation $\langle x,y| \ \rangle$, that is the CW-complex with a unique $0$-dimensional cell $e^0$ and two cells of dimension $1$, $x$ and $y$, so $K$ is a wedge of two circles, the ``figure eight". The fundamental group of $K$ is $\pi_1(K)=\F=\langle x,y \rangle$. Let $p:\widetilde{K}\to K$ be the covering associated to the subgroup $\FF \vartriangleleft \F$, that is $p_*(\pi_1(\widetilde{K}))=\FF$. The complex $\widetilde{K}$ can be identified with the grid $\Z\times \R \cup \R\times \Z \subseteq \R^2$. An element $w$ of $\FF$ has a representative loop $\gamma$ in $K$ which can be lifted to a loop $\widetilde{\gamma}$ in $\widetilde{K}$ from a fixed $0$-dimensional cell $\widetilde{e}^0=(0,0)\in \R^2$ of $\widetilde{K}$. The loop $\widetilde{\gamma}$ is the curve $\gamma_w$ of Definition \ref{definicion}. The abelianization $\FF=\pi_1(\widetilde{K})\to H_1(\widetilde{K})=\FF/\F''$ which maps $w\in \FF$ to the homology class of $\widetilde{\gamma}$ is the winding invariant. From this interpretation the statement of Theorem \ref{conway} ($\W$ is surjective and $\textrm{ker}(W)=\F ''$) is obvious. The deck transformation group $G=\F / \FF=\Z \times \Z$ turns $H_1(\widetilde{K})$ into a $G$-module. In this case $H_1(\widetilde{K})$ is a free $G$-module of rank $1$, so $H_1(\widetilde{K})=\Z[G]=\ZZ$. The group $\F$ acts on $\FF$ by conjugation and it is clear that $W: \FF \to \FF/\F''$ is equivariant: if $u\in \F$ and $w\in \FF$, then $\gamma_{uwu^{-1}}=\gamma_u* \overline{u}\cdot \gamma_w * (\gamma_u^{-1})$, which in homology is the class $\overline{u} \cdot W(w)$ of $\overline{u} \cdot \gamma_w$. Here $\overline{u}$ denotes the class of $u$ in $\F/\FF$. The group $H_1(\widetilde{K})=\FF / \FFF$ is the relation module of the presentation $\langle x,y | [x,y] \rangle$. In general the relation module of a presentation $\mathcal{P}=\langle X | R \rangle$ of a group $G$ is the $\Z[G]$-module $\frac{N(R)}{N(R)'}$, where $N(R)\trianglelefteq F(X)$ is the normal closure of $R$ in the free group with generator set $X$. The action of $G$ on $N(R)/N(R)'$ is again $\overline{u}\cdot \overline{w}=\overline{uwu^{-1}}$ for $u\in F(X)$ and $w\in N(R)$. We can call the abelianization map $W:N(R) \to N(R)/N(R)'$ the \textit{winding invariant} of $\mathcal{P}$ by analogy with the original case. In Theorem \ref{xxmenos1} we have already used the winding invariant of $\langle x,y | y\rangle$. In many cases a similar geometric interpretation makes sense (see Section \ref{seccayley}). The winding invariant satisfies then $\W(uwu^{-1})=\overline{u}\cdot \W(w)$ for any $w\in N(R)$ and $u\in F(X)$. In general, the relation module of a presentation can be non-free. Dunwoody provides in \cite{Dun} a presentation of the trefoil group $G=\langle x,y | x^2y^{-3} \rangle$ with two generators and two relators whose relation module is a projective $\Z[G]$-module which is not free (see also \cite{BD,HJ}). Nevertheless, the relation module is always a submodule of a free $\Z[G]$-module (see \cite{Mag2} and Appendix A).

Recall that a path-connected space $X$ is said to be aspherical if its homotopy groups $\pi_n(X)$ are trivial for $n\ge 2$. If $X$ is a connected $2$-dimensional CW-complex, then it is aspherical if and only if $\pi_2(X)=0$. A presentation $\mathcal{P}=\langle X|R\rangle$ is aspherical if its standard complex $K_{\mathcal{P}}$ is aspherical. Asphericity is a natural yet elusive property. The Whitehead conjecture stated in 1941 says that if a $2$-dimensional CW-complex is aspherical, then any connected subcomplex is also aspherical. This problem remains open despite decades of effort. Important progress has been made, for instance by Bestvina and Brady in connection with the Eilenberg-Ganea Conjecture and by Howie and others in connection with labeled oriented trees and the Andrews-Curtis Conjecture.

If $\mathcal{P}=\langle X| R \rangle$ is a presentation of a group $G$, and $K$ is its standard complex, the cellular chain complex of the universal cover of $K$ gives the following exact sequence of $\Z[G]$-modules

\begin{equation}
0\to \pi_2(K) \to \Z[G]^{\oplus |R|} \overset{\tau}{\to} \Z[G]^{\oplus |X|} \to \Z[G] \to \Z \to 0 
\label{mika}
\end{equation}

\bigskip

A description of the maps involved can be found for instance in \cite[Theorem 1.39]{MS}. The image of $\tau$ is exactly the relation module $N(R)/N(R)'$ of $\mathcal{P}$. In particular one has the following

\begin{prop} \label{asf}
If $\mathcal{P}$ is aspherical, then the relation module is a free $\Z[G]$-module with basis $\{rN(R)'\}_{r\in R}$.
\end{prop}

For more details on this result see Appendix A.

Dyer and Vasquez proved in \cite[Theorem 2.1]{DV} (1973) that if $\mathcal{P}=\langle X| r\rangle$ is a one-relator presentation such that $r$ is not a proper power in the free group $F(X)$, then $\mathcal{P}$ is aspherical. The proof is implicit in Lyndon's work \cite{Lyn} (1950) and in Cockroft's \cite{Coc} (1954).

In the words of Lyndon:

\begin{teo} (Lyndon's Simple Identity Theorem)
Let $F$ be a free group and $r\in F$ not a proper power. If $$\prod\limits_{i =1}^{n} a_i r^{\epsilon_i} a_i^{-1}=1$$ for certain $n\ge 1$, $a_i\in F$, $\epsilon_i=\pm 1$, then the indices $1,2,\ldots, n$ can be grouped into pairs $(i,j)$ with $\epsilon_i=-\epsilon_j$ and $a_i \equiv a_j$ mod $N(r)$.  
\end{teo}

Recall that an aspherical presentation presents a torsion-free group (see \cite[Proposition 2.45]{Hat}, for instance).


\subsection{Residual properties of one-relator groups} \label{resfin}

In the fifties R.L. Vaught made the following conjecture:

\begin{conj} \label{vaughtconjecture}
Let $u,v$ be two elements in a free group $F$ such that $u^2v^2$ is a square $w^2$ for some $w\in F$. Then $u$ and $v$ commute.
\end{conj}

Lyndon proved \cite{Lyn2} that in the hypothesis of the conjecture the group generated by $u,v,w$ is cyclic, and in particular $u$ and $v$ commute. Later, Edmunds gave in \cite{Edm} a short proof of the same fact by showing that the equation $u^2v^2=w^2$ is equivalent up to an automorphism to the equation $[u,v]=w^2$. Then he used Wicks characterization of commutators, Theorem \ref{wicks}. In the following we will prove some results in the same spirit as the Vaught conjecture: If $u,v$ are elements of a free group $F$ which satisfy certain equation, then they must commute. See \cite{Sch, LS} for generalizations of the Vaught conjecture. 

One of the key results in the article \cite{BMT} by Baumslag, Miller and Troeger, is the following \cite[Lemma 1]{BMT}

\begin{lema} (Baumslag, Miller, Troeger) \label{propbmt}
Let $u$ and $v$ be two words in a free group $F$ which do not commute. Then $vuv^{-1}u^{-2}$ is not conjugate in $F$ to either $u$ or $u^{-1}$.
\end{lema}

Their proof is two pages long and considers several cases taking into account whether the words involved are reduced or if there are cancellations to be performed. Here we give a shorter proof of a more general statement using the appropriate version of the winding invariant.

\begin{lema} \label{Baum}
Let $u$ and $v$ be two words in a free group $F$ which do not commute. Let $n,m,k\in \Z$ with $n,m\neq 0$. Then $vu^nv^{-1}u^{m}$ is not conjugate in $F$ to $u^k$.
\end{lema}
\begin{proof}
Suppose first $k\neq 0$ and that there exists $w\in F$ such that $vu^nv^{-1}u^{m}=wu^{k}w^{-1}$. Let $r\in F$ be the root of $u$, that is $r$ is not a proper power and there exists $l\ge 1$ such that $u=r^l$. The subgroup $H$ of $F$ generated by $r,v,w$ is free. If $\textrm{rk}(H)=3$, then $\{r,v,w\}$ is a basis of $H$ and $vr^{ln}v^{-1}r^{lm}$, $wr^{lk}w^{-1}$ are different reduced words in $F(r,v,w)$, which is absurd. Since $u$ and $v$ do not commute we must have $\textrm{rk}(H)=2$. We may assume $H=\F=\langle x,y\rangle$. Let $\mathcal{P}=\langle x,y| r\rangle$. Then $\mathcal{P}$ is aspherical and by Proposition \ref{asf}, the relation module $N(r)/N(r)'$ is a free $\Z[G]$-module generated by $\rho=rN(r)'$. The group $G$ is of course $\F/N(r)$. We apply the corresponding version of the winding invariant, that is the abelianization map $W: N(r)\to N(r)/N(r)'$. By hypothesis $ln\overline{v} \rho+lm \rho =W(vr^{ln}v^{-1}r^{lm})=W(wr^{lk} w^{-1})=lk \overline{w} \rho$. Here, $\overline{v}$ and $\overline{w}$ denote classes in $G$. Therefore $n\overline{v}+m=k \overline{w} \in \Z[G]$. Since $n,m,k\neq 0$, this implies that $\overline{v}=\overline{w}=1\in G$, that is $v,w\in N(r)$. Since $H=\F$ is generated by $r,v,w$, then $\F \subseteq N(r)$ and then $\mathcal{P}$ is a presentation of the trivial group with two generators and one relator. This is absurd. The case $k=0$ is simpler because we have $vu^nv^{-1}u^m=1\in F$ and since $u,v$ do not commute, the group generated by $r,v$ is free of rank $2$. Exactly as before we obtain $n\overline{v} +m=0 \in \Z[G]$ and since $n,m\neq 0$, $r,v\in N(r)$, which is absurd.
\end{proof}

\begin{coro} \label{doscom2}
Let $u,v$ be two elements in a free group $F$, and let $n,k\in \Z$ with $k\neq n\neq 0$. Suppose that there exists $w\in F$ such that $[v,u^n]=[w,u^k]$. Then $u$ and $v$ commute.
\end{coro}

The authors of \cite{BMT} use Lemma \ref{propbmt} to show that if $u$ and $w$ are two non-commuting elements in the free group $F(X)$ generated by a set $X$, then $\langle X | [wuw^{-1},u]u^{-1}\rangle$ is not residually finite. Recall that a group $G$ is said to be residually finite if for every $1\neq g\in G$ there exists a finite group $H$ and a homomorphism $G\to H$ which maps $g$ to a non-identity element. Equivalently, for every $1\neq g\in G$ there exists a normal subgroup $N \triangleleft G$ of finite index not containing $g$. Every finitely generated residually finite group is hopfian, so in particular Baumslag-Solitar groups give examples of one-relator non-residually finite groups. There is no general method known which decides for a given one-relator group whether it is residually finite or not. A well-known conjecture by Baumslag \cite[Conjecture A]{Bau5} proved by Wise in \cite[Section 18]{Wis} says:

\begin{teo} (Wise)
Every one-relator group with torsion is residually finite.
\end{teo}

For more results on residually finite one-relator groups see \cite{All0, All1, All2, Bau, Bor}.

The proof of Baumslag, Miller and Troeger in \cite{BMT} can be easily adapted to prove the following.

\begin{prop}(Baumslag, Miller, Troeger \cite[Theorem 1]{BMT}) \label{bmtgeneral}
Let $u$ and $w$ be two non-commuting elements in the free group $F(X)$ generated by a set $X$. Then the group $G=\langle X | \ [wu^{k}w^{-1},u]u^{-1}\rangle$ is not residually finite for any $k\neq 0$.
\end{prop}
\begin{proof}
Since $u$ and $w$ do not commute, they freely generate a free group of rank 2. Thus, $v=wu^kw^{-1}$ and $u$ do not commute either. By Lemma \ref{propbmt}, $vuv^{-1}u^{-2}$ is not a conjugate of $u$ nor $u^{-1}$. A well-known result by Magnus says that if two elements in a free group have the same normal closure, then one is conjugate to the other or to its inverse. Thus, $u\neq 1\in G$. We show that for any homomorphism $G\to H$ onto a finite group $H$, $u=1\in H$. Suppose that $u\neq 1\in H$. Let $m\ge 2$ be the order of $u$ in $H$ and let $p$ be the smallest positive prime dividing $m$. Since $vuv^{-1}=u^2$ in $H$, then $vu^{\frac{m}{p}}v^{-1}=u^{\frac{2m}{p}}$, so $p\neq 2$ and the conjugation by $v$ induces an automorphism $\mu$ of $\Z_p \simeq \langle u^{\frac{m}{p}} \rangle \leq H$. The order $\textrm{ord}(\mu)$ of $\mu$ divides $p-1$. On the other hand $\mu^m$ is the conjugation by $v^m=wu^{km}w^{-1}=1\in H$, so $\textrm{ord}(\mu)$ divides $m$. If $\textrm{ord}(\mu)\ge 2$, a prime $q$ dividing $\textrm{ord}(\mu)$ yields a contradiction. Thus $\mu$ is the identity of $\langle u^{\frac{m}{p}}\rangle$, which means that $u^{\frac{m}{p}}=u^{\frac{2m}{p}}$, so $u^{\frac{m}{p}}=1\in H$, which is again a contradiction. This proves that $u=1\in H$.      
\end{proof}

The presentations in Theorem \ref{bmtgeneral} are of the form $\langle X| vuv^{-1}u^{-2}\rangle$ for $u$ and $v$ non-commuting. The presentations of the form $\langle X| vu^lv^{-1}u^{m}\rangle$ for $u,v$ non-commuting are known to be non-residually finite if $|l|, |m|$ and $1$ are pairwise different, by a result of Meskin \cite[Theorem B]{Mes}. According to Meskin, no example of a non-residually finite presentation of the form above is known if $l=-m$ or $l=m\neq \pm 1$. 

Lemma \ref{Baum} can be used to obtain examples of non-residually solvable one-relator groups.
Recall that a group $G$ is said to be residually solvable if the intersection $\bigcap\limits_{i\ge 0} G^{(i)}$ of the members in the derived series is the trivial group. Equivalently $G$ is residually solvable if for every non-identity element $g\in G$ there is a normal subgroup $N\vartriangleleft G$ such that $g\notin N$ and $G/N$ is solvable. Kropholler proves in \cite{Kro} that Baumslag-Solitar groups have second derived subgroup free, and in particular they are residually solvable, though they are non-Hopfian and thus not residually finite. In \cite{Bau2} and \cite[4.3]{Bau} Baumslag describes other classes of one-relator groups which are residually solvable. Residually nilpotent implies residually solvable, so the results of \cite{BM2} give more examples of residually solvable one-relator groups. See \cite{KDB} for more on this topic.

All the groups in the statement of Theorem \ref{bmtgeneral} are non-residually solvable. We can use our generalization Lemma \ref{Baum} to prove that this holds for a larger class.

\begin{prop} \label{ressol}
Let $u$ and $v$ be two non-commuting elements in $F(X)$ such that $v$ lies in the normal closure $N(u)$ of $u$. Then $G=\langle X | [v, u^k]u^{-1}\rangle$ is not residually solvable for any $k\neq 0,-1$.
\end{prop} 
\begin{proof}
It is clear that $u$ is in every term of the derived series of $G$. We must show that $u\neq 1$. Otherwise, by Magnus' result, $[v,u^k]u^{-1}=vu^kv^{-1}u^{-k-1}$ would be conjugate to $u$ or $u^{-1}$.
\end{proof}

\begin{prop} \label{resnil}
Let $u$ and $v$ be two non-commuting elements in $F(X)$. Then the group $G=\langle X | [v, u^k]u^{-1}\rangle$ is not residually nilpotent for any $k\neq 0,-1$.
\end{prop} 
\begin{proof}
Note that $u$ is in every element of the lower central series of $G$ and $u\neq 1\in G$ by Lemma \ref{Baum}.
\end{proof}

\subsection{More applications}

This subsection contains a list of comments which do not intend to be original. They are simple applications of the ideas used in the rest of the section and they are motivated by the proofs in \cite{BMT} and \cite{LN} which could be substantially shortened by using these methods. 

From Lyndon's Identity Theorem we immediately deduce the following

\begin{prop} \label{aumentacion}
Let $u$ be a nontrivial element of a free group $F$. Let $n\ge 1$, $m_1,m_2,\ldots ,m_n\in \Z$ and let $a_1,a_2,\ldots ,a_n\in F$ be such that $\prod\limits_{i=1}^n a_iu^{m_i}a_i^{-1}=1\in F$. Then $\sum\limits_{i=1}^n m_i=0$.
\end{prop}

The previous result can also be proved using the idea in the proof of \cite[Lemma 1]{BMT}: Since $u\neq 1\in F$ and $F$ is residually nilpotent, there exists $k\ge 1$ such that $u\in \gamma_k(F) \smallsetminus \gamma_{k+1}(F)$. The equation of the statement reduces modulo $\gamma_{k+1}(F)$ to $u^{\sum m_i}=\prod\limits_{i=1}^n u^{m_i}=1$. Since $\gamma_k(F)/\gamma_{k+1}(F)$ is torsion-free, we must have $\sum\limits_{i=1}^n m_i=0$. 

Proposition \ref{aumentacion} says that if $a_1,a_2,\ldots, a_n$ are elements in a free group $F$ whose product $a_1a_2\ldots a_n$ is $1$ and the equation $a_1t^{\epsilon_1}a_2t^{\epsilon_2}\ldots a_nt^{\epsilon_n}=1$ has a nontrivial solution in $F$, then the equation is singular (the total exponent of $t$ is zero). When one looks for solutions in an overgroup of the coefficient group, then singularity is the only obstruction according to the Kervaire-Laudenbach conjecture. When the coefficient group $G$ is torsion-free the only equations without solutions in an overgroup are the conjugates of elements of $G$ in $G*\langle t\rangle$ according to a conjecture of Levin \cite{Bro, Edj, Eva2, GR, How, Hol3, Kly, Lev, Neu}.

\begin{prop} \label{main1}
Let $u$ be a nontrivial element in a free group $F$. Let $n\ge 1$, $m_1,m_2,\ldots ,m_n\in \Z$ and let $a_1,a_2,\ldots ,a_n\in F$ be such that $\prod\limits_{i=1}^n a_iu^{m_i}a_i^{-1}=1\in F$. Suppose moreover that $a_j=1$ for some $1\le j\le n$. Finally, suppose that there is no proper subset $\emptyset\neq P \subsetneq \{1,2,\ldots, n\}$ such that the sum of the exponents $m_i$ with $i\in P$ is zero. Then all the coefficients $a_i$ commute and they also commute with $u$. 
\end{prop}

Note that $P=\{1,2,\ldots, n\}$ satisfies that the sum of the corresponding $m_i$ is zero by Proposition \ref{aumentacion}. This result is a generalization of Lemma \ref{Baum}.

\bigskip

\textit{Proof of Proposition \ref{main1}}.
As in the proof of Lemma \ref{Baum}, let $r\in F$ be the root of $u$ with $u=r^l$ and let $H$ be the subgroup of $F$ generated by $r,a_1,a_2,\ldots, a_n$. Let $X$ be a free basis of $H$ and let $\mathcal{P}=\langle X| r\rangle$ be the presentation of a group $G$. It is aspherical. When we apply the winding invariant of $\mathcal{P}$ to the equation of the statement, we obtain $\sum\limits_{i=1}^n lm_i\overline{a}_i=0\in \Z[G]$. Let $P=\{1\le i\le n | \ \overline{a}_i=1\in G\}$. Then $P\neq \emptyset$ and the sum of the $m_i$ with $i\in P$ is zero. By hypothesis $P=\{1,2,\ldots, n\}$. Therefore, $a_i\in N(r)$ for every $i$ and then $H\subseteq N(r)$. Then $\textrm{rk}(H)=1$, so all the $a_i$ and $u$ lie in a cyclic group. \qed

\bigskip

\begin{ej}
Let $u,v,w,z$ be elements in a free group $F$ with $u\neq 1$. Then $$vu^2v^{-1}wu^{-3}w^{-1}u^2zu^{-1}z^{-1}=1$$ if and only if $v,w$ and $z$ commute with $u$. To prove this simply observe that no proper multiset of $\{2,-3,2,-1\}$ satisfies that the sum of its members is $0$. 
\end{ej}

\begin{prop}
Let $u,v$ be two elements in a free group $F$. Let $n\ge 1$ and $m_1,m_2,\ldots ,m_n\in \Z$ not all zero. If $\prod\limits_{i=1}^{n} u^{m_i}v^{m_i}=1$, then $u$ and $v$ commute.
\end{prop}
\begin{proof}
This proof can be done in different ways. We choose one which uses the winding invariant. If $uv=1$, we are done. Suppose then $uv\neq 1$ and let $r$ be the root of $uv$, $uv=r^l$. Let $X$ be a basis of the group $H$ generated by $r$ and $v$. Then $\mathcal{P}=\langle X| r\rangle$ is aspherical. Let $\rho=rN(r)'$ and let $\W:N(r)\to N(r)/N(r)'$ be the winding invariant of $\mathcal{P}$. Let $g$ be the class of $v$ in $G=F(X)/N(r)$. If $m\ge 1$, $$u^mv^m=(uv)v^{-1}(uv)vv^{-2}(uv)v^2v^{-3}\ldots v^{1-m}(uv) v^{m-1},$$ so $\W(u^mv^m)=l(1+g^{-1}+g^{-2}+\ldots +g^{1-m})\rho$. If $m\le -1$, $$u^{m}v^{m}=(u^{-1}v^{-1})v(u^{-1}v^{-1})v^{-1}v^2(u^{-1}v^{-1})v^{-2}\ldots v^{-m-1}(u^{-1}v^{-1})v^{m+1}$$ and since $u^{-1}v^{-1}=v(uv)^{-1}v^{-1}$, $\W(u^{-1}v^{-1})=-lg\rho$ and then $\W(u^mv^m)=-l(g+g^2+g^3+\ldots +g^{-m})\rho$. Therefore 
\begin{equation}
\sum\limits_{m_i\ge 1} (1+g^{-1}+g^{-2}+\ldots +g^{1-m_i})- \sum\limits_{m_i\le -1} (g+g^2+g^3+\ldots +g^{-m_i})=0\in \Z[G]
\label{zg}
\end{equation}
Since $\mathcal{P}$ is aspherical, $G$ is torsion-free. Therefore, if $g\neq 1\in G$, since not all the $m_i$ are zero, the coefficient of $1\in G$ or $g\in G$ in the left hand side of Equation (\ref{zg}) is nonzero, a contradiction. Therefore $g=1$, i.e. $v\in N(r)$. Thus $H\subseteq N(r)$ and then $H$ is cyclic. Then $r$ and $v$ commute, so $uv=r^l$ and $v$ commute, and then $u$ and $v$ commute.
\end{proof}

If $v\in \F= \langle x,y\rangle$, then $[[x,y],v]\in \F ''$ if and only if $v\in \F '$. This follows easily from an application of the winding invariant $\F '\to \ZZ$. The following is a generalization of that result.

\begin{prop}
Let $u,v$ be elements of a free group $F$ with $u\neq 1$. Let $r$ be the root of $u$. Then $[u,v]\in N(r)'$ if and only if $v\in N(r)$.
\end{prop}
\begin{proof}
If $v\in N(r)$, then trivially $[u,v]\in N(r)'$. Conversely, suppose $[u,v]\in N(r)'$. We use the winding invariant of $\langle X|r\rangle$, where $X$ is a basis of $\langle r,v\rangle$. Since $r^lvr^{-l}v^{-1}\in N(r)'$, $l-l\overline{v}=0\in \Z[G]$. Thus, $v=1\in G=F(X)/N(r)$. 
\end{proof}

We have used the hypothesis of asphericity to make sure that the relation module is a free $\Z[G]$-module. In general, for presentations with more that one relator, asphericity is a difficult property to check. For example, the Whitehead conjecture open for 78 years can be restated in terms of presentations as follows. If $\langle X| R\rangle$ is an aspherical presentation, then any subpresentation $\langle Y|S\rangle$ (i.e. a presentation with $Y\subseteq X$ and $S\subseteq R$) must be aspherical. The particular case in which $S$ has only one element was proved by Cockroft in \cite[Theorem 2]{Coc}. This can be seen with the ideas we have used so far: Let $\mathcal{P}=\langle x_1,x_2,\ldots, x_n | r_1,r_2,\ldots, r_m \rangle$ be an aspherical presentation. Then no relator $r_j$ is a proper power in $F(x_1,x_2,\ldots, x_n)$. Indeed, suppose $1\le j\le m$ is such that $r_j=s^l$ for some $s\in F(x_1, x_2,\ldots,x_n)$ and $l\ge 1$. Since $\mathcal{P}$ is aspherical, the presented group $G$ is torsion-free. Therefore $s$ must be in the normal subgroup $N$ generated by the relators. Let $W:N\to N/N'$ be the winding invariant. By hypothesis $N/N'$ is a free $\Z[G]$-module with basis $\{r_kN'\}_{1\le k\le m}$, so $r_jN'=W(s^l)=lW(s)$ implies $l=1$.

Although there is no general criterion to prove that a presentation is aspherical when we have more than one relator, there are many conditions which are easy to check in some cases and guarantee asphericity \cite{BarMin, Ger, How, Huc, Pri2}. Presentations with more that one relator can be used to study equations over groups which are not free. We mention just one example.

\begin{ej}
Let $\langle x,y| u\rangle$ be a presentation of a group $H$ and let $v\in \F$. If $\mathcal{P}=\langle x,y | u,v \rangle$ is aspherical, then the equation $vtvt^{-1}=1$ over $H$ has no solution in $H$.

Suppose $t\in \F$ is such that $vtvt^{-1}=1$ in $H$, that is $vtvt^{-1}\in \F$ lies in the normal closure $N(u) \vartriangleleft \F$ of $u$. Let $N=N(u,v)\trianglelefteq \F$ and $G=\F/N$. The relation module $N/N'$ of $\mathcal{P}$ is a free $\Z[G]$-module with basis $\{uN',vN'\}$. We apply the winding invariant of $\mathcal{P}$ to obtain that $W(vtvt^{-1})=(1+\overline{t})vN'=TuN'$ for some $T\in \Z[G]$. Then $1+\overline{t}=0\in \Z[G]$, which is absurd.
\end{ej}

\subsection{Relation with the Magnus embedding} \label{submagnus}

Let $\PP=\langle x_1,x_2,\ldots, x_n| R\rangle$ be a presentation of a group $G$ and let $M$ be a free $\Z[G]$-module with basis $\{\alpha_1,\alpha_2,\ldots ,\alpha_n\}$. Consider the group $H$ of matrices   
\begin{displaymath}
\left( \begin{array}{cc}
g & 0 \\
a & 1 
\end{array} \right)
\end{displaymath} with $g\in G$, $a\in M$.
The homomorphism $\mu: \mathbb{F}_n \to H$ determined by $\mu (x_i)=\left( \begin{array}{cc}
\overline{x}_i & 0 \\
\alpha_i & 1 
\end{array} \right)$ has kernel $N(R)'$ and induces an isomorphism $\mathbb{F}_n/ N(R)' \to H$ which is known as the Magnus embedding \cite{Fox, Mag, Mag2}. 
Every element in $N(R)/N(R)'$ is mapped by this homomorphism to a matrix of the form $\left( \begin{array}{cc}
1 & 0 \\
a & 1 
\end{array} \right)$, so we can compose the restriction to $N(R)/N(R)'$ with the projection $p:H\to M$ onto the $(2,1)$-coefficient to obtain a homomorphism $\overline{\mu}:N(R)/N(R)'\to \Z[G]^n$. This map also receives the name of Magnus embedding and it coincides with the inclusion $\textrm{Im}(\tau)\to \Z[G]^n$ in the exact sequence (\ref{mika}).  

We will concentrate now on the original version of the winding invariant, so $\PP=\langle x,y| [x,y]\rangle$, $N(R)=\F'$, $G=\Z \times \Z$ and $N(R)/N(R)'=\F'/\F''$ can be identified with $\Z[G]=\ZZ$ by means of $W$. The map $W$ associates a Laurent polynomial to every element $w\in \F'$ while the map $\overline{\mu}$ associates two Laurent polynomials to $w$. The Magnus embedding gives in this case duplicate information, in the sense that $\overline{\mu}(w)$ is determined by $W(w)$. The map $\overline{\mu}:\F'/\F'' \to \ZZ^2$ is up to an isomorphism the invariant considered by Sarkar in \cite{Sar} (see also \cite[Section 3.2]{Sal}). The next result shows the explicit relation between these two invariants.  

Let $M$ be a free $\Z [G]$-module with basis $\{\alpha,\beta \}$ and $H$ the group of matrices 

\begin{displaymath}
\left( \begin{array}{cc}
g & 0 \\
a & 1 
\end{array} \right)
\end{displaymath} with $g\in G$, $a\in M$.
Let $t=(1-\overline{y}^{-1})\overline{x}^{-1}\alpha + (\overline{x}^{-1}-1)\overline{y}^{-1}\beta\in M$. There is a well-defined group homomorphism $\varphi :\ZZ \to H$ which maps a monomial $X^nY^m$ to 

\begin{displaymath}
\left( \begin{array}{cc}
1 & 0 \\
\overline{x}^{-n}\overline{y}^{-m}t & 1 
\end{array} \right).
\end{displaymath} 

\begin{prop} \label{embedding}
The map $\varphi \W: \FF \to H$ can be extended to the Magnus homomorphism $\mu : \F \to H$. In particular, $\varphi$ is a monomorphism and the Magnus embedding $\overline{\mu}$ is a monomorphism $p\varphi$ composed with $\W: \F'/\F'' \to \ZZ$.
\end{prop}
\begin{proof}
Since $\F '$ is generated by the conjugates $x^ny^m[x,y]y^{-m}x^{-n}$ of $[x,y]$, we only have to check that $\varphi \W (x^ny^m[x,y]y^{-m}x^{-n})=\mu (x^ny^m[x,y]y^{-m}x^{-n})$. The left hand side is $\varphi (X^n Y^m)=\left( \begin{array}{cc}
1 & 0 \\
\overline{x}^{-n}\overline{y}^{-m}t & 1 
\end{array} \right)$. Now, the right hand side is $\mu (x^ny^m) \mu ([x,y]) \mu (y^{-m}x^{-n})$.
A direct computation shows that $\mu ([x,y])=\left( \begin{array}{cc}
1 & 0 \\
t & 1 
\end{array} \right)$ and the effect of conjugating by $\mu (x^ny^m)$ such matrix is replacing $t$ by $\overline{x}^{-n}\overline{y}^{-m}t$.
\end{proof}


\section{Planar Cayley graphs} \label{seccayley}

If $\langle x_1,x_2,\ldots, x_n | r_1,r_2,\ldots, r_m \rangle$ has a planar Cayley graph, then the winding invariant $\W:N(R) \to N(R)/N(R)'$ defined in Subsection \ref{subversion} can be interpreted geometrically as in the case of $\langle x,y| [x,y]\rangle$. We can take a particular embedding of $\widetilde{K}$ in $\R^2$ for $\widetilde{K}$ the covering of the figure eight $K$ corresponding to $N(R)\leqslant \pi_1(K)$ and then consider some points $p_0,p_1, \ldots, p_{k-1} \in \R^2 \smallsetminus \widetilde{K}$. Then we can define a map $\W':N(R)\to \Z^k$ as follows. For $w\in N(R)$, take the loop in $K$ representing $w$, lift it to a loop $\widetilde{w}$ in $\widetilde{K}$ and consider the winding numbers $\w(\widetilde{w},p_i)$ about the points $p_i$. Then $\W':N(R)\to \Z^k$ defined by $\W'(w)=(\w(\widetilde{w},p_1), \w(\widetilde{w},p_2), \ldots, \w(\widetilde{w},p_k))$ is just $\W:N(R)\to H_1(\widetilde{K})$ composed with a map $H_1(\widetilde{K}) \to \Z^k$.

\subsection{Two examples}

If $k\ge 1$, the Cayley graph of $\langle x,y | x^k, y\rangle$ consists of a $k$-gon with vertices $0,1,\ldots, k-1$ with a circle $C_i$ attached to each vertex $i$. We consider the embedding of this graph $\widetilde{K}$ in $\R^2$ illustrated in Figure \ref{planar1}.

\begin{figure}[h] 
\begin{center}
\includegraphics[scale=0.8]{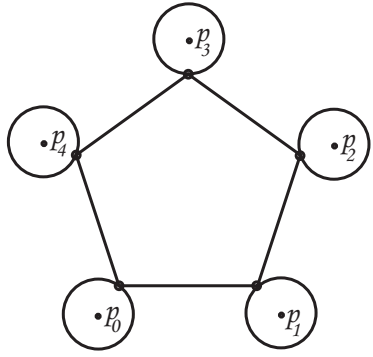}
\caption{The Cayley graph $\Gamma(\Z_k,\{x,y\})$ for $k=5$.}\label{planar1}
\end{center}
\end{figure}

Take a point $p_i$ in the interior of each circle $C_i$. In this case $H_1(\widetilde{K})$ is a free abelian group of rank $k+1$, but $\W: N\to \Z^k$ maps the normal closure $N$ of $\{x^k,y\}$ in $\F$ to a free abelian group of rank $k$ (we have chosen to ignore the winding numbers about the polygon). The group $\Z^k \leqslant H_1(\widetilde{K})$ is in fact a $G=\F /N=\Z_k$-submodule of $H_1(\widetilde{K})$. If we identify $\Z^k$ with $\Z[G_k]$ for $G_k$ the cyclic group of order $k$ generated by an element $g$, then the action of $G$ on $\Z^k=\Z[G_k]$ is described as follows. If $w\in \F$ with $\exp(x, w)=j$ and $v\in \Z[G_k]$, then $\overline{w}.v=g^j.v$. In other words, $\overline{w}$ shifts the coordinates of $v$ $j$ times. We will apply this idea in the following criterion. Note that if $w\in \F$ is a product $a^kb^k$ of two $k$-th powers, then $k| \exp(x,w)$. The next result concerns products $w=a^pb^p$ for $p$ prime in which $p^2| \exp(x,w)$.

\begin{prop}
Let $w=x^{n_1}y^{m_1}x^{n_2}y^{m_2}\ldots x^{n_s}y^{m_s} \in \F$, $n_i,m_i\in \Z$ for every $i$. Let $p\ge 2$ be a prime number such that $p^2$ divides $\exp(x,w)$ and let $G_p=\langle g\rangle$ be the cyclic group of order $p$. Let $$\alpha = \sum\limits_{i=1}^s m_i g^{n_1+n_2+\ldots+n_i} \in \Z[G_p].$$ If $\alpha$ is not a multiple of $p\in \Z$ and it is not a multiple of $1+g+g^2+\ldots +g^{p-1}\in \Z[G_p]$, then $w=a^pb^p$ has no solution $a,b\in \F$. 
\end{prop}  
\begin{proof}
Suppose $w=a^pb^p$ for certain $a,b\in \F$. Thus $$w=(ab).b^{-1}(ab)b.b^{-2}(ab)b^2 \ldots b^{-(p-1)}(ab)b^{p-1}.$$ We consider the Cayley graph of $\langle x,y | x^p,y\rangle$, $N$ the normal closure of $\{x^p,y\}$ in $\F$, and the map $\W':N\to \Z[G_p]$ defined above. Note that $\W'(w)=\alpha$. Since $\exp(x,w)=p \exp(x,ab)$, by hypothesis $p | \exp(x,ab)$. It is easy to see then that $ab\in N$. Therefore $$\W'(w)=(1+(\overline{b})^{-1}+(\overline{b})^{-2}+\ldots+(\overline{b})^{-(p-1)})\W'(ab).$$ If $p|\exp(x,b)$ then $\W'(w)=p\W'(ab)$. If $p\nmid \exp(x,b)$, then $(1+(\overline{b})^{-1}+(\overline{b})^{-2}+\ldots+(\overline{b})^{-(p-1)})\W'(ab)=(1+g+g^2+\ldots+g^{p-1})\W'(ab)$. 
\end{proof}

\begin{ej}
For example $w=x^7yxy^{-2}xy^4$ is not a product of two cubes. Note that for $p=3$, $p^2| \exp(x,w)=9$ and $\alpha=g^7-2g^8+4g^9=4e+g-2g^2\in \Z[G_3]$ is not a multiple of $3$ nor a multiple of $1+g+g^2$. Thus, the proposition applies. On the other hand note that $w$ is a product of cubes by Theorem \ref{caracterizacioncubos}.
\end{ej}

A second application of the same idea is with the presentation $\langle x,y|x^2,y^2,[x,y]\rangle$ of $\Z_2 \times \Z_2$. In this case $N(R)$ is the subgroup of $\F$ consisting of the words $w$ such that $\exp(x,w)$ and $\exp(y,w)$ are both even. The Cayley graph $\widetilde{K}=\Gamma(\Z_2\times \Z_2,\{x,y\})$ has an embedding in the plane illustrated in Figure \ref{z2z2}.

\begin{figure}[h] 
\begin{center}
\includegraphics[scale=0.9]{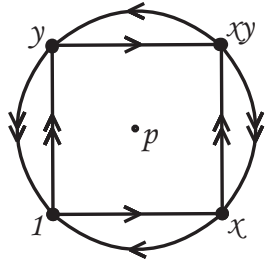}
\caption{The Cayley graph $\Gamma(\Z_2\times \Z_2, \{x,y\})$.}\label{z2z2}
\end{center}
\end{figure}

 In this case we take only one point $p$ in the interior of the square, so $W':N(R)\to \Z$ maps a word $w\in N(R)$ to the winding number of $\widetilde{w}$ around $p$. The number $W'(w)$ is very easy to compute. Given $w\in N(R)$, consider $q(w)\in \Z_2* \Z_2$, where $q:\F=\langle x,y \rangle \to \Z_2*\Z_2=\langle x,y\rangle$ maps $x$ to $x$ and $y$ to $y$. If the reduced word representing $q(w)$ has length $l$, then $W'(w)=\frac{l}{4}$ if $x$ is the first letter of the word, and $W'(w)=-\frac{l}{4}$ if $y$ is the first letter. For instance, for $w=x^2yx^{-1}y^{-2}x^4yx$, $q(w)=yxyx$, so $W'(w)=-1$.
We equip $\Z$ with an action of $G=\Z_2\times \Z_2$: $\overline{x}\cdot n=\overline{y}\cdot n=-n$. Then $W':N(R)/N(R)'\to \Z$ is a homomorphism of $\Z[G]$-modules. For $w\in N(R)$, the curves $\widetilde{w}$ and $\overline{x}\widetilde{w}$ are symmetric with respect to the vertical line through $p$, while $\overline{y}\widetilde{w}$ is symmetric to $\widetilde{w}$ with respect to the horizontal line through $p$.

\begin{prop}
Let $w\in \F$ be such that $4|\exp(x,w)$ and $4|\exp(y,w)$. If $W'(w)\in \Z$ is odd, $w$ is not a product of two squares.
\end{prop}
\begin{proof}
Suppose $w=a^2b^2$ for some $a,b\in \F$. Then $ab\in N(R)$ by hypothesis and $W'(w)=(1+(\overline{b})^{-1})W'(ab)$. This is $0$ or $2W'(ab)$.
\end{proof}

This provides one last proof of Proposition \ref{ln}: $q([x,y])=xyxy\in \Z_2*\Z_2$, so $W'([x,y])=1$ and then $[x,y]$ is not a product of two squares. 

\subsection{Relation lifting and an example of Dunwoody}

If $\langle x_1,x_2,\ldots, x_n |r_1,r_2,\ldots, r_m \rangle$ is a presentation of a group $G$, and $N$ is the normal closure of $\{r_1,r_2,\ldots, r_m\}$, then the classes of $r_1,r_2,\ldots, r_m$ generate the relation module $N/N'$ as $\Z[G]$-module. The relation lifting problem asks for a given set $\{s_1,s_2,\ldots,s_k\}\subseteq N$ whose classes generate $N/N'$ if there exists a set $\{s_1',s_2',\ldots, s_k'\} \subseteq N$ whose normal closure is $N$ and such that the class of $s_i'$ in the relation module coincides with the class of $s_i$ for every $1\le i\le k$. In this case we say that the classes $\overline{s}_i\in N/N'$ can be lifted. This is connected to the relation gap problem as follows. If for a given presentation $\langle x_1,x_2,\ldots, x_n |r_1,r_2,\ldots, r_m \rangle$ any set of generators of $N/N'$ can be lifted, in particular the minimum number $d_{\Z[G]}(N/N')$ of generators of $N/N'$ coincides with the minimum cardinality $d(N)$ of a set $S\subseteq N$ whose normal closure in $F(x_1,x_2,\ldots, x_n)$ is $N$. When $d(N)=d_{\Z[G]}(N/N')$ we say that the presentation does not have a relation gap. It is an open problem whether there exists a finite presentation with a relation gap. For more details on the relation gap problem see \cite{Har1}.

Wall conjectured in \cite{Wal} that the relation lifting of generators of $N/N'$ was always possible. Dunwoody proved that this conjecture is false in \cite{Dun}. We will explain Dunwoody's example using the ideas of this section. Our argument is very similar to Harlander's in \cite[Section 1.5]{Har2}.

Let $\mathcal{P}=\langle x,y | x^5\rangle$ be a presentation of $G=\Z_5*\Z$. Then the class $\overline{x^5}$ generates $N/N'$ as $\Z[G]$-module. Here $N=N(x^5)$ is the normal closure of $x^5$ in $\F$. Note that $1-x+x^2\in \Z[G]$ is a unit, with inverse $x+x^2-x^4$. Therefore, $(1-x+x^2)y$ is also a unit and then $(1-x+x^2)y\overline{x^5}$ is also a generator of $N/N'$. This element is the class of $s=yx^5y^{-1}xyx^{-5}y^{-1}x^{-1}x^2yx^5y^{-1}x^{-2}\in N$ in $N/N'$. We claim that this class cannot be lifted. Otherwise there exists $s'\in N$ whose normal closure is $N(s')=N$ and $\overline{s'}=\overline{s}\in N/N'$. Since $N(s')=N(x^5)$, by Magnus' result $s'$ is a conjugate of $x^5$ or $x^{-5}$.

The Cayley graph $\Gamma (\Z_5*\Z, \{x,y\})$ is planar and can be embedded in the plane as shown in Figure \ref{fractal}.

\begin{figure}[h] 
\begin{center}
\includegraphics[scale=0.7]{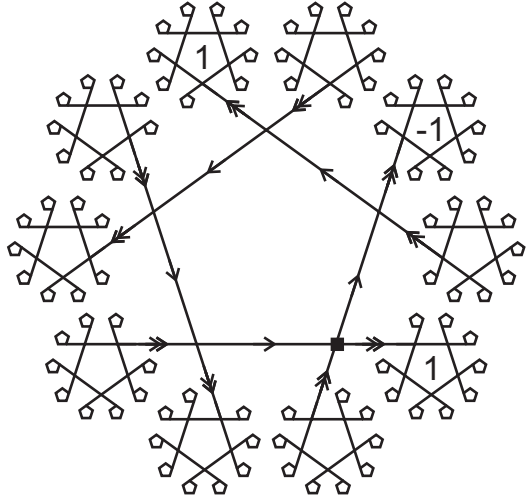}
\caption{The Cayley graph $\Gamma(G, \{x,y\})$.}\label{fractal}
\end{center}
\end{figure}

$\Gamma (G, \{x,y\})$ contains countably many pentagons $P_i$, $i\in \mathbb{N}$, and we can choose a point $p_i$ inside each $P_i$. Then $W':N\to \Z^{\mathbb{N}}$ maps an element $w\in N$ to the collection of winding numbers $\w(\widetilde{w},p_i)$, $i\in \mathbb{N}$. Note that $W'(x^5)$ has one coordinate equal to $1$ and all the others equal to $0$. Since $s'$ is conjugate to either $x^5$ or $x^{-5}$, then $W'(s')$ has one coordinate equal to $1$ or $-1$ and all the other coordinates are $0$. Since $\overline{s'}=\overline{s}$, $W'(s')=W'(s)$. But $W'(s)$ has three nonzero coordinates, two equal to $1$ and the other equal to $-1$ (see Figure \ref{fractal}). This is a contradiction. 

\bigskip

To finish this section we would like to make a comment connected to relation lifting and the results in Section \ref{sectionac} and \cite{Barac}. The following is inspired by a discussion with Jens Harlander.

Dunwoody example used that $\Z_5*\Z$ has non-trivial units. There is no example known of a torsion free group for which relation lifting fails. 

\begin{teo}
If every $2$-generator set of the relation module of $\langle x,y| [x,y]\rangle$ has a lifting, then there exist infinitely many presentations $\mathcal{P}_i=\langle x,y | r_i,r'_i\rangle$ which are simple homotopy equivalent but pairwise $Q$-inequivalent.
\end{teo}
\begin{proof}
Consider the infinitely many pairs $S_i=(\overline{w}_i, \overline{w_i'}) \in \F/\F'' \times \F/\F''$ of Theorem \ref{qineq}. Since their normal closure is $\F'/\F''$, they generate $\F'/\F''$ as $\Z[G]$-module, where $G=\F/\F'=\Z \times \Z$. Suppose we can lift each $S_i$ to a pair $(r_i,r_i')\in N\times N$ with normal closure $N=\F'$, $\overline{r}_i=\overline{w}_i \in \F'/\F''$ and $\overline{r_i'}=\overline{w_i'} \in \F'/\F''$. Let $\mathcal{P}_i=\langle x,y | r_i,r_i'\rangle$. Then the standard complexes $K_{\mathcal{P}_i}$ all have fundamental group isomorphic to $\Z \times \Z$ and Euler characteristic $1$. By the arguments in \cite{Barac}, all these complexes have the same homotopy type. Since the Whitehead group of $\Z\times \Z$ is trivial, they are simple homotopy equivalent. By Theorem \ref{qineq}, the pairs $S_i=(\overline{w}_i, \overline{w_i'})=(\overline{r}_i, \overline{r_i'}) \in \F/\F''\times \F/\F''$ are pairwise $Q$-inequivalent, and then so are the pairs $(r_i,r_i')\in \F \times \F$. 
\end{proof}

\section{Bisections and tilings} \label{secbisec}

Although this work has followed a different path, the original motivation for defining the winding invariant was a problem about bisections of regions in the integer lattice and an idea by Michael Hitchman.

\subsection{The motivation}

In 2015, Michael Hitchman published a paper called ``The topology of tile invariants" \cite{Hit}. A \textit{tile} is a subspace of $\R^2$, homeomorphic to a disk, which is a union of finitely many squares in the integer lattice. A \textit{region} is formally defined in the same way as a tile, but it plays a different role and thus receives a different name. Given a set $T=\{t_1,t_2,\ldots, t_n\}$ of tiles and a region $\Gamma$, one wants to know if $\Gamma$ can be tiled by $T$, that is if it can be covered without overlapping by translates of tiles in $T$. For a tiling $\alpha$ of $\Gamma$ by $T$ and $1\le i\le n$, we define $b_i(\alpha)$ to be the number of translates of $t_i$ used in $\alpha$. A \textit{tile invariant} is a relation among the numbers $b_i(\alpha)$ which depends on $\Gamma$ but not on $\alpha$. For instance, when $T=\{t_1,t_2,t_3,t_4\}$ is the set of tiles of Figure \ref{tiles}, Conway and Lagarias proved \cite{CL} that for any region $\Gamma$, $b_3(\alpha)-b_2(\alpha)$ is independent of $\alpha$. Hitchman gave in \cite{Hit} an alternative proof of this result using that a particular presentation $\Q$ has an associated complex $K_{\Q}$ which is Cockroft. Recall that this means that the Hurewicz map $h:\pi_2(K_{\Q})\to H_2(K_{\Q})$ is trivial. Tile invariants have been further studied in \cite{MPak,MuP, Pak}. The connection between tilings and $\pi_2$ is given by the following idea by Hitchman. Given a tile $t$ and a point $p$ with integer coordinates in the boundary, we define the \textit{boundary word} of $t$ from $p$ to be the word $w_{t,p}\in \F$ read when we travel in clockwise direction along the boundary of $t$: when we move one step to the right we write $x$, when we move to the left we write $x^{-1}$, upwards $y$ and downwards $y^{-1}$. If $p=(0,0)\in \R^2$, the curve $\gamma_{w_t}$ that we used in the definition of the winding invariant is the curve which starts in $p$ and travels along the boundary of $t$. If we consider a different basepoint $p'$, the boundary word $w_{t,p'}$ differs from $w_{t,p}$ by a conjugation.

\begin{figure}[h] 
\begin{center}
\includegraphics[scale=0.5]{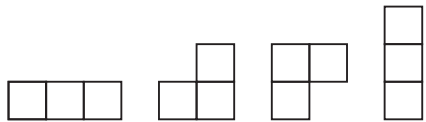}
\caption{The four tiles in $T$.}\label{tiles}
\end{center}
\end{figure}

If $T=\{t_1,t_2,\ldots, t_n\}$ is a set of tiles and $w_i$ is the boundary word of $t_i$ for certain basepoint, then the \textit{tile boundary group} $G_T$ is the group presented by $\PP_T=\langle x,y | w_1,w_2, \ldots, w_n\rangle$. The tile boundary group was originally defined by Conway and Lagarias in \cite{CL} (in fact they work with the derived group $G_T'$) and used by Thurston \cite{Thu} and Reid \cite{Rei}. Now, if $\Gamma$ is a region which can be tiled by $T$, then each tiling $\alpha$ gives a van Kampen diagram over $\PP_T$ with boundary label the boundary word of $\Gamma$. The dual of this diagram is a picture $\mathbb{B}_{\alpha}$ over $\PP_T$. If $\beta$ is a second tiling of $\Gamma$ by $T$, then $\mathbb{B}_{\beta}$ has the same boundary label as $\mathbb{B}_{\alpha}$ and $\mathbb{B}_{\alpha}\mathbb{B}_{\beta}^{-1}$ is equivalent to a spherical picture, which represents an element of $\pi_2(\PP_T)$. For basic notions on van Kampen diagrams and the theory of pictures see \cite{BP}.

In 2000 and 2001 the Argentinian Puzzle Championship counted with the following type of tiling problems created by Jaime Poniachik and Ivan Skvarca: given a region $\Gamma$, exhibit a bisection of $\Gamma$, that is a tiling by only two tiles, one a rotation or reflection of the other. See Figure \ref{anubisl} for an example.

\begin{figure}[h] 
\begin{center}
\includegraphics[scale=0.7]{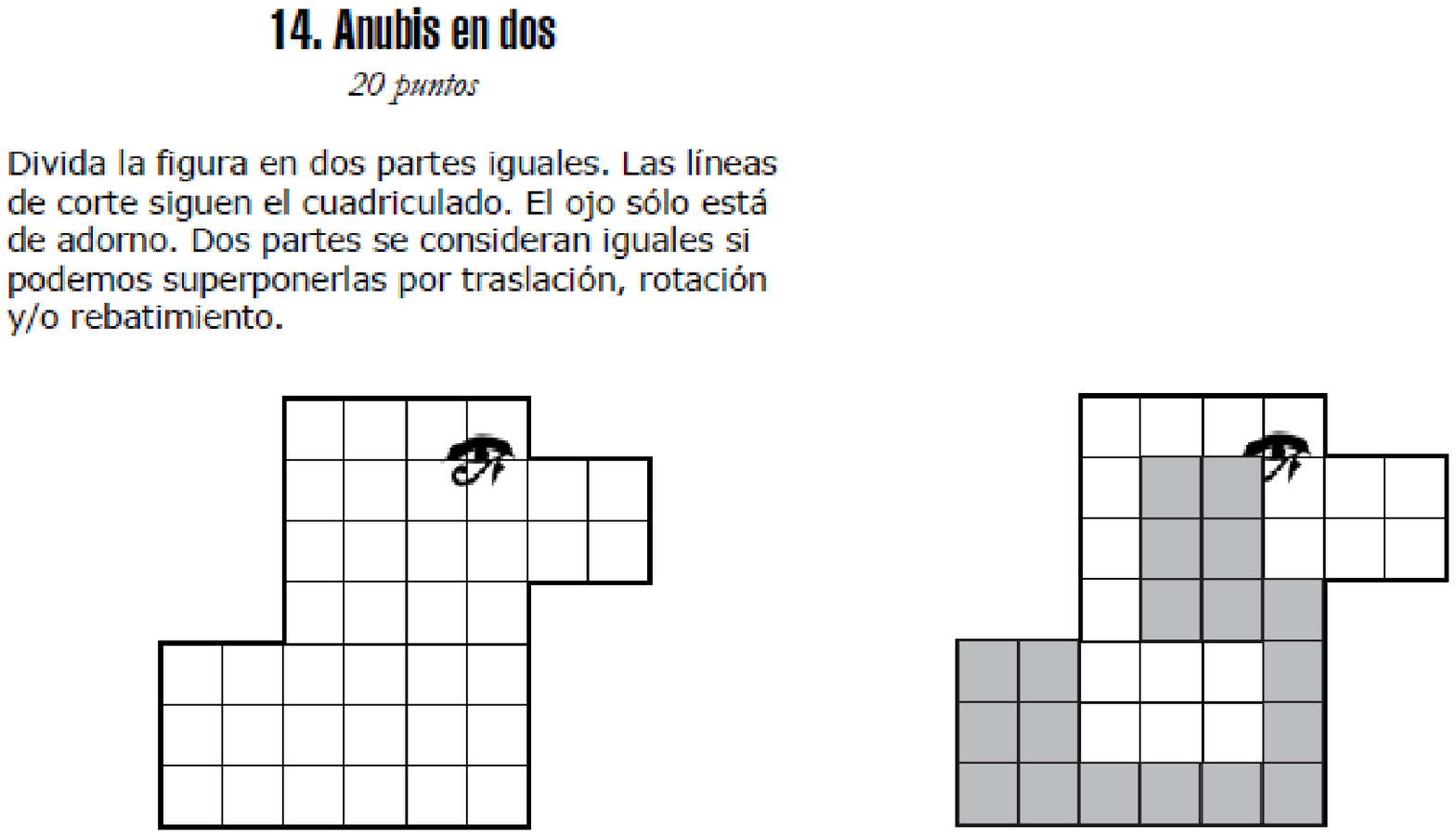}
\caption{A region and a bisection.}\label{anubisl}
\end{center}
\end{figure}

Our first plan was to try to adapt Hitchman's ideas to study this problem, though this plan quickly changed to the point of view of the present article. Of course, an exhaustive search solves the problem for any given region.

A particular case is when we only allow translations. We have then the following problem: Given a region $\Gamma$ in the integer lattice can it be tiled with exactly two tiles $t$, $t'$, one a translate of the other?

We can associate a Laurent polynomial $P_{\Gamma}\in \ZZ$ to each region $\Gamma$: it is the sum of the monomials $X^nY^m$ such that the square with lower left corner $(n,m)$ appears in $\Gamma$. If $t'$ is a translate of $t$, then $P_{t'}$ and $P_{t}$ differ in a multiplication by a monomial $X^kY^l$. Therefore, if $\Gamma$ can be tiled by two equal tiles, which differ only by a translation, then $P_{\Gamma}=(1+X^kY^l)P_t$.

On the other hand the boundary word of $\Gamma$ is a product of two squares! This is clear as the tiling describes a van Kampen diagram over $\PP_{\{t\}}$ with two regions. In other words, suppose $\Gamma$ is tiled by $t$ and a translate $t'$. The assumptions imply that $t \cap t'$ is homeomorphic to a segment. Let $p\in \Z^2$ be one of the endpoints of $t \cap t'$. When we travel through the boundary of $\Gamma$ in clockwise direction starting from $p$, we first travel through part of the boundary of one tile, say $t$, and then through part of the boundary of the other, $t'$. The boundary word $w_{\Gamma,p}$ is the concatenation of two words $w$ and $w'$, each corresponding to one of these parts. Let $u$ be the word associated to $t\cap t'$ starting at $p$. Then $w_{\Gamma}=ww'=wu^{-1}uw' \in \F$. But $wu^{-1}=w_{t,p}$ is the boundary word of $t$ starting at $p$ and $uw'=w_{t',p}$ is the boundary word of $t'$ starting at $p$. Since $t'$ is a translate of $t$, $w_{t',p}$ and $w_{t,p}$ differ in a conjugation. Thus $w_{\Gamma}=abab^{-1}=(ab)^2b^{-2}$ for some $a,b\in \F$.

In conclusion, when $\Gamma$ can be tiled with two equal tiles, only translation allowed, the boundary word of $\Gamma$ is a product of two squares and the polynomial $P_{\Gamma}$ is a multiple of $1+X^kY^l$ for some $k,l\in \Z$. The definition of the winding invariant was originally introduced to generalize this result. For a word $w\in \F'$ we wanted to associate a polynomial which was a multiple of $1+X^kY^l$ if $w$ is a product of two squares. 

What can be said in the general case, when $\Gamma$ is tiled by two equal tiles and rotation and symmetry are allowed? In this case $P_{\Gamma}=P(X,Y)+X^kY^lQ(X,Y)$ for some $k,l\in \Z$ and $Q(X,Y)$ one of the following eight: $P(X,Y), P(Y,X^{-1}), P(X^{-1},Y^{-1}), P(Y^{-1},X),$ $P(X^{-1},Y),$ $P(Y^{-1},X^{-1}), P(X,Y^{-1}), P(Y,X)$. And the boundary words of such regions are of the form $w_{\Gamma}=a\varphi(a)\varphi(b)b$ for some automorphism $\varphi$ of $\F$, if $t'$ is obtained from $t$ by an orientation preserving transformation, while $w_{\Gamma}=a\varphi(b)\varphi(a)^{-1}b^{-1}$ if the transformation is orientation reversing. They are a twisted product of squares and a twisted commutator, respectively. If not a topological or group theoretical answer, maybe another simple necessary condition for the existence of a bisection can be obtained from the polynomial condition.

\subsection{The equation $a^2b^2c^2d^2=1$ and normal roots}

The equation $a^2=1$ has only one solution in $\F$. The equation $a^2b^2=1$ has infinitely many solutions in $\F$: they are exactly the pairs $a=w,b=w^{-1}$ with $w\in \F$. The equation $a^2b^2c^2=1$ has its solutions described by the proof of the Vaught conjecture: they are of the form $a=w^n, b=w^m, c=w^k$ for $w\in \F$ and $n,m,k\in \Z$ such that $n+m+k=0$. The equation $a^2b^2c^2d^2$ has a set of solutions which is more difficult to describe. A solution must satisfy that $\langle a,b,c,d\rangle \leqslant \F$ is a free group of rank $1$ or $2$ \cite[Section 5]{Lyn4}. See also \cite[Section 4]{CE} to learn about some solutions of this equation (see \cite{BEF} and \cite{BEF2} for the solution of $[a,b]=[c,d]$ in free groups and free products). One way to produce solutions of this equation is by considering regions $\Gamma$ that admit two different bisections, each of them a tiling by a tile and a translate.

\begin{ej}

Figure \ref{cuatro} shows a region $\Gamma$ and two different bisections by translates of a tile.

\begin{figure}[h] 
\begin{center}
\includegraphics[scale=0.5]{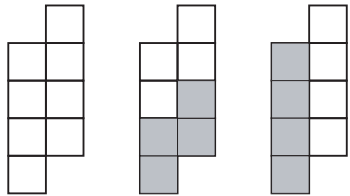}
\caption{A region with two bisections by a tile and a translate.}\label{cuatro}
\end{center}
\end{figure}

The boundary word $w=y^4xyxy^{-4}x^{-1}y^{-1}x^{-1}$ has then two descriptions as a product of squares. Let $u=y^2xyxy^{-2}x^{-1}y^{-1}x^{-1}$ be the the boundary word of the tile in the first bisection. Then $w=uvuv^{-1}=(uv)^2v^{-2}$ for $v=xyxy^2x^{-1}y^{-1}x^{-1}$. Let $u'=y^4xy^{-4}x^{-1}$ be the boundary word of the tile in the second bisection. Then $w=u'v'u'(v')^{-1}=(u'v')^2(v')^{-2}$ for $v'=xy$. Thus, we have the following solution for $a^2b^2c^2d^2=1$: $a=uv=y^2$, $b=v^{-1}=xyxy^{-2}x^{-1}y^{-1}x^{-1}$, $c=v'=xy$, $d=(u'v')^{-1}=y^3x^{-1}y^{-4}$. 

\end{ej} 
    
Of course, a similar idea can be used to produce solutions of the equation $a^nb^n=c^md^m$ for $n,m\ge 2$. In this case we should find a region which admits a tiling by $n$ translates of tile $t$ and another tiling by $m$ translated of a tile $t'$. In each case the translation vectors must be of the form $0, \vv, 2\vv, \ldots, (n-1)\vv$ and $0, \vv', 2\vv',\ldots, (m-1)\vv'$.

\bigskip

If a region $\Gamma$ can be tiled by a unique tile $t$, then the boundary word of $\Gamma$ lies in the normal closure of the boundary word of $t$. When a word $w\in \F$ is in the normal closure of a word $r\in \F$ we say that $r$ is a \textit{normal root} of $w$. If $r$ is a normal root of $w$ then any conjugate of $r$ or $r^{-1}$ is also a normal root. There are words which admit many normal roots, none of them a conjugate of another. In fact, if $w\in \F'$ then any primitive element of $\F$ is a normal root: if $\{u,v\}$ is a basis of $\F$, $w$ is a product of powers of $u$ and $v$, and since $w\in \F'$ the total exponent of $v$ in this expression is $0$, so $w\in N(u)$. The following remains open (\cite[Problem F27]{BMS}).

\begin{ques} (Magnus)
Can an element in $\mathbb{F}_n$ which is not in $\mathbb{F}_n'$ have infinitely many pairwise non-conjugate normal roots?
\end{ques}		

Magnus studied some cases in \cite{Mag3}, McCool \cite{Mccool} extended results of Steinberg \cite{Ste} about normal roots of $x^ny^m$, and also characterized the normal roots of $[x^n,y]$.

One way to produce elements in $\F'$ with many (finite) normal roots, also in $\F'$, is to find a region $\Gamma$ with many tilings by a unique tile.

\begin{ej}
Figure \ref{muchos} shows a region $\Gamma$ and five tilings by a unique tile. The two tilings in the left are real tiles. The partitions in the middle are not tilings strictly speaking, because the pieces are not homeomorphic to disks. But they are simply connected and the same ideas work. The partition in the right is done with a piece which is not even connected, but if we allow maximal $1$-cells, we can connect both squares to obtain a simply connected piece and everything works fine. There is one trivial tiling by a unique tile and another trivial tiling by $1\times 1$ squares which does not appear in the picture. An eighth potential tiling by another non-connected piece (of size $4$) does not produce a normal root.
\begin{figure}[h] 
\begin{center}
\includegraphics[scale=0.7]{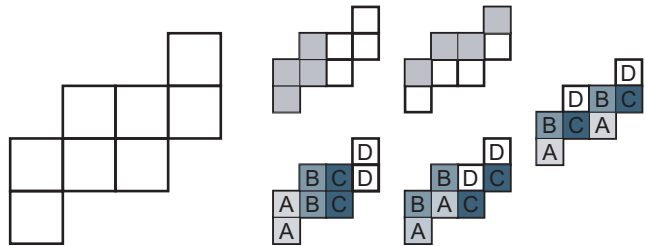}
\caption{Tilings by tiles and other pieces.}\label{muchos}
\end{center}
\end{figure}
The seven normal roots of $w_{\Gamma}=y^2xyx^2yxy^{-2}x^{-1}y^{-1}x^{-2}y^{-1}x^{-1}=[y^2,xyx^2yx]$ which are obtained from the boundary words of the tiles are: $w_{\Gamma}, [y^2,xyx], [y^2,x], [y,xyx^2yx], [y,xyx],$ $[yx^2y,x], [x,y]$.
\end{ej}

\subsection{Full invariant tile sets}

To finish we want to mention an application of Hitchman's ideas. We will say that a set $T=\{t_1,t_2,\ldots, t_n\}$ of tiles is \textit{full invariant} if for each tileable region $\Gamma$, the number $b_i(\alpha)$ of tiles of type $i$ in a tiling $\alpha$ by $T$ is uniquely determined by $\Gamma$. In other words, for any two tilings $\alpha$, $\beta$ of a region $\Gamma$, the numbers $b_i(\alpha)$ and $b_i(\beta)$ are equal for every $1\le i\le n$.

\begin{ej}

The set $T=\{t_1,t_2\}$ of tiles in Figure \ref{nofull} is not full invariant since there is a region which can be tiled by $2$ translates of $t_1$ and $4$ of $t_2$ and it can also be tiled by $8$ translates of $t_1$ and none of $t_2$.

\begin{figure}[h] 
\begin{center}
\includegraphics[scale=0.5]{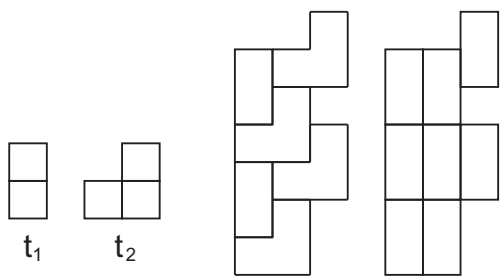}
\caption{Two tilings with different number of tiles of each type.}\label{nofull}
\end{center}
\end{figure}

\end{ej}

To be full invariant means that the \textit{tile counting group} is the whole group $\Z[T]$ (see \cite{Hit}).

The next proposition follows immediately from Hitchman's results.

\begin{prop}
Let $T=\{t_1,t_2,\ldots, t_n\}$ be a set of tiles. If $\PP_T$ is aspherical, $T$ is full invariant.
\end{prop}
\begin{proof}
This follows directly from \cite[Theorem 3.1]{Hit} since in this case $h(J)=0$. Alternatively we can avoid the argument with pictures and use the fact that the relation module of $\PP_T$ is a free $\Z[G_T]$-module generated by the classes of the relators $w_i=w_{t_i}$. Two tilings $\alpha$ and $\beta$ of a region $\Gamma$ determine van Kampen diagrams with the same boundary word $w_{\Gamma}$, so we have $$w_{\Gamma}=\prod\limits_{j\in J} u_jw_{i_j}u_j^{-1}=\prod\limits_{j\in J'} v_jw_{k_j}v_j^{-1}$$ for certain $u_j,v_j\in \F$. The number $b_i(\alpha)$ is the number of $j\in J$ such that $i_j=i$ and $b_i(\beta)$ is the number of $j\in J'$ such that $k_j=i$. In the relation module $\Z[G]^n$ we have $\sum \overline{u}_j w_{i_j}=\sum \overline{v}_j w_{k_j}$. Applying the augmentation $\epsilon^n :\Z[G]^n\to \Z^n$ we obtain the required identities $b_i(\alpha)=b_i(\beta)$. 
\end{proof}


Thus, usual strategies for proving asphericity can be used to deduce properties of tile invariants.

\begin{ej}
The set $T=\{t_1,t_2\}$ of tiles in Figure \ref{sifull} is full invariant.
\begin{figure}[h] 
\begin{center}
\includegraphics[scale=0.5]{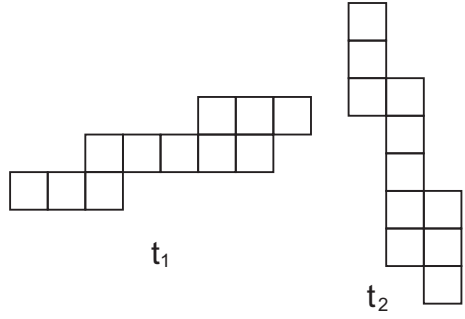}
\caption{A full invariant tile set.}\label{sifull}
\end{center}
\end{figure}
It is easy to check that $\PP_T=\langle x,y | w_1,w_2\rangle$ is aspherical by noting it satisfies the small cancellation condition $C(6)$: no relator is a product of fewer than $6$ pieces. The relator $w_1=yx^2yx^3yx^3y^{-1}x^{-1}y^{-1}x^{-4}y^{-1}x^{-3}$ contains six letters $y$ but no piece of $w_1$ contains two $y$. The same happens with $w_2=xy^{-2}xy^{-3}xy^{-3}x^{-1}yx^{-1}y^4x^{-1}y^3$ and the letter $x$.
\end{ej}

\section*{A. Aspherical presentations and freedom of relation modules} \label{appendix}

If a presentation is aspherical its relation module is free with a basis given by the classes of the relators. This result is well-known for the researchers in the area but we could not find a good reference for a proof.
Let $\PP=\langle X | R \rangle$ be a presentation of a group $G$ and let $K$ be its standard complex. So $K$ has a unique $0$-cell, one $1$-cell for each generator and one $2$-cell for each relator. The attaching map of the $2$-cell corresponding to $r\in R$ follows the $1$-cells associated to the letters in $r$ with orientations given by the signs. The fundamental group of $K$ is $\pi_1(K)=G$. Let $p:\widetilde{K} \to K$ be the universal cover of $K$. The restriction $q:(\widetilde{K})^1 \to K^1$ to $1$-skeletons is also a covering. The fundamental group of $K^1$ is $\pi_1(K^1)=F(X)$ and the covering $q$ corresponds to the subgroup $N(R) \leqslant F(X)$. Indeed, if $w\in N(R)$, the corresponding loop in $K^1$ becomes trivial in $\pi_1(K)=F(X)/N(R)$ so it lifts to (the $1$-skeleton of) $\widetilde{K}$. This proves that $N(R)\leqslant q_*(\pi_1((\widetilde{K})^1))$. On the other hand, if $i:K^1\to K$ denotes the inclusion, $i_*q_*:\pi_1((\widetilde{K})^1)\to \pi_1(K)$ factorizes through $\pi_1(\widetilde{K})=0$. Thus, $q_*(\pi_1((\widetilde{K})^1))\subseteq \textrm{ker}(i_*)=N(R)$. Note that $(\widetilde{K})^1$ is the Cayley graph $\Gamma (G,X)$. The relation module of $\PP$ is $N(R)/N(R)'=H_1((\widetilde{K})^1)$. This coincides with the kernel of $d_1:C_1((\widetilde{K})^1)\to C_0((\widetilde{K})^1)$ in the cellular chain complex. On the other hand, since $\widetilde{K}$ is simply-connected, we have an exact sequence of $\Z[G]$-modules

$$0\to \textrm{ker}(d_2) \to C_2(\widetilde{K}) \overset{d_2}{\to} C_1(\widetilde{K}) \overset{d_1}{\to} C_0(\widetilde{K}) \to \Z \to 0. $$

\bigskip

Of course, $\textrm{ker}(d_2)=H_2(\widetilde{K})=\pi_2(\widetilde{K})=\pi_2(K)$. So, if $\PP$ is aspherical ($\pi_2(K)=0$), $d_2:C_2(\widetilde{K})\to \textrm{ker}(d_1)=N(R)/N(R)'$ is an isomorphism and then the relation module is a free $\Z[G]$-module. For a relator $r\in R$, the corresponding $2$-cell of $K$ is lifted to one $2$-cell $g\widetilde{r}$ of $\widetilde{K}$ for each element $g\in G$. $C_2(\widetilde{K})$ is the free $\Z[G]$-module with basis $\{\widetilde{r}\}_{r\in R}$. The homomorphism $d_2$ maps $\widetilde{r}$ to $rN(R)'\in N(R)/N(R)'$. Thus, if $\PP$ is aspherical, $N(R)/N(R)'$ is the free $\Z[G]$ module with basis $B=\{rN(R)'\}_{r\in R}$. Conversely, if the relation module is free with basis $B$, then $d_2$ is injective and $\PP$ is aspherical.

\section*{List of symbols}




\begin{tabular}{l l}
\noindent $\F$ & the free group of rank $2$ with basis $\{x,y\}$ \\
$[u,v]=uvu^{-1}v^{-1}$ \\
$G'=[G,G]$ & the derived/commutator subgroup \\
$G''=[G',G']$ \\
$\Gamma(G,X)$ & the Cayley graph of $G$ with generating set $X$ \\
$\ZZ$ & the ring of Laurent polynomials in two variables \\
$\exp (x,w), \exp(y,w)$ & the total exponents of $x$ and $y$ in $w\in \F$ \\
$\gamma_w$ & the curve associated to $w\in \F$, p. \pageref{definicion} \\
$\w(\gamma,p)$ & the winding number of $\gamma$ around $p\in \R^2$ \\
$P_w=\W(w)$ & the winding invariant of $w\in \F'$, p. \pageref{definicion} \\
$N(r), N(R)$ & the normal closure of the element $r$ or the subset $R$ \\
$[u_1,u_2,\ldots, u_n]$ & $=[u_1, [u_2,\ldots, u_n]]$ \\
$F(X)$ & the free group with basis $X$ \\
$\mathbb{F}_n=F(x_1,x_2,\ldots, x_n)$ & the free group of rank $n$ \\
$c_n$ & $=[x_1,x_2,\ldots, x_{n+1}]\in \mathbb{F}_{n+1}$ \\
$e_{n}=[y,e_{n-1}]$ & the Engel words ($e_1=[y,x]$) \\
$\M=\F/\F''$ & the free metabelian group of rank $2$ with generators $\{x,y\}$ \\ 
$\gamma_{n}(G)=[G,\gamma_{n-1}(G)]$ & the groups in the lower central series of $G$ ($\gamma_1(G)=G$) \\
$\tau_k(G)$ & $=\{g \in G | g^n\in \gamma_k(G) $ for some $n\ge 1 \}$ \\
$u(G)$ & the verbal subgroup associated to the word $u\in \mathbb{F}_n$, p. \pageref{pagverbal} \\
$E_2(R), D_2(R), GE_2(R)$ & p. \pageref{pagmatrices} \\
$B(n,m)$ & the Burnside group, p. \pageref{pagburnside} \\
$I(W(S))$ & the ideal of $\ZZ$ generated by $\{P_s\}_{s\in S}$\\
$\vv_a$ & $=(\exp (x,a),\exp (y,a))\in \Z^2$, for $a\in \F$ \\
$\iota, \iota_{a,b}, \iota_{p,a,b}$ & p. \pageref{pagiota} \\
$L=L_{a,b}$ & the subgroup of $\Z^2$ generated by $\vv_a$ and $\vv_b$ \\
$P\{(n,m)\}$ & the coefficient of $X^nY^m$ in $P\in \ZZ$ \\
$\g _G(g)$ & the commutator length of $g\in G'$, p. \pageref{genero} \\ 
$\Sq(G)$ & the subgroup of $G$ generated by the squares \\
$\Sq(g)$ &  the square length of $g\in \Sq(G)$, p. \pageref{pagsq} \\ 
$m_w$ & $=X^{\exp (x,w)}Y^{\exp(y,w)}$ for $w\in \F$\\
\end{tabular}

%

\end{document}